\documentclass[11pt,leqno]{article}

\usepackage[margin=2.5cm]{geometry} 
\usepackage{microtype} 

\usepackage[pagebackref]{hyperref}
\usepackage{amsfonts}
\usepackage{amsthm}
\usepackage{amsmath}
\usepackage{amscd}
\usepackage{amssymb}
\usepackage{mathtools}
\usepackage{tikz}
\usepackage{tikz-cd}
\usetikzlibrary{patterns}
\tikzset{>=latex}
\usepackage[shortlabels]{enumitem}

\def\AA {{\mathbb A}}     
\def\CC {{\mathbb C}}     
\def\FF {{\mathbb F}}     
\def\KK {{\mathbb K}}     
\def\PP {{\mathbb P}}     
\def\QQ {{\mathbb Q}}     
\def\RR {{\mathbb R}}     
\def\ZZ {{\mathbb Z}}     

\def\mc {\mathcal}

\def\Ob {\mathrm{Ob}}
\def\Hom {\mathrm{Hom}}
\def\Aut {\mathrm{Aut}}
\def\Ext {\mathrm{Ext}}

\newtheorem{theorem}{Theorem}[section]

\newtheorem{lemma}[theorem]{Lemma}
\newtheorem{prop}[theorem]{Proposition}

\theoremstyle{definition}

\newtheorem{example}[theorem]{Example}
\newtheorem{remark}[theorem]{Remark}
\theoremstyle{plain}

\newtheorem{conj}[theorem]{Conjecture}

\numberwithin{equation}{section}

\newlength{\miniwidth}

\newcommand{\Addresses}{{
  \bigskip
  \footnotesize

	(F.~Haiden) \textsc{University of Oxford, Mathematical Institute, Andrew Wiles Building, Woodstock Road, Oxford OX2 6GG, UK} \par\nopagebreak
	\textit{E-mail:} \texttt{Fabian.Haiden@maths.ox.ac.uk}
	\medskip
	
}}

\begin{document}

\title{Legendrian skein algebras and Hall algebras}
\author{Fabian Haiden}
\date{}
\maketitle

\begin{abstract}
We compare two associative algebras which encode the ``quantum topology'' of Legendrian curves in contact threefolds of product type $S\times\mathbb R$.
The first is the skein algebra of graded Legendrian links and the second is the Hall algebra of the Fukaya category of $S$.
We construct a natural homomorphism from the former to the latter, which we show is an isomorphism if $S$ is a disk with marked points and injective if $S$ is the annulus.
\end{abstract}


\setcounter{tocdepth}{2}
\tableofcontents

\section{Introduction}

This work relates two different constructions of associative algebras which capture the ``quantum topology'' of Legendrian curves in contact threefolds.
\begin{enumerate}[1)]
\item
Legendrian skein algebras
\item
Hall algebras of Fukaya categories of surfaces
\end{enumerate}
The algebras 1) are defined by imposing linear relations between curves which differ in some small ball, while 2) involves first constructing an $A_\infty$-category whose structure constants count immersed disks and then passing to its Hall algebra, a kind of decategorification.
We relate these by constructing a natural homomorphism from 1) to 2).
Two previous works contain evidence of a general connection between (non-Legendrian) skein algebras and Hall algebras.
Morton--Samuelson~\cite{morton_samuelson} show that the HOMFLY-PT skein of the thickened torus is isomorphic to a specialization of Hall algebra of the elliptic curve, and Cooper--Samuelson~\cite{cooper_samuelson} give a conjectural presentation of the Hall algebra of the Fukaya category by skein-like relations.
Our original motivation was to better understand the general theory behind these results.
While we make progress towards this goal, some aspects still remain more mysterious, such as the appearance of the HOMFLY-PT skein relations instead of the legendrian ones.

\subsection{Two algebras from Legendrian curves}

\subsubsection*{Skein modules}

Let $M$ be a contact threefold with oriented contact distribution, i.e. a smooth threefold equipped with a completely nonintegrable oriented rank two subbundle $\xi \subset TM$.
We use a variant of the skein module for graded curves which depends on an additional choice of grading structure on $M$: a rank one subbundle of the contact distribution $\xi$.
The skein module of $M$ is then defined as the $\ZZ[q^{\pm},(q-1)^{-1}]$-module generated by isotopy classes of graded embedded closed Legendrian curves (refer to them as \textit{links} for brevity) in $M$ modulo the following linear skein relations, where $\delta_{m,n}=1$ for $m=n$ and $0$ otherwise, as usual.
(The labels $m,n$ show the grading, see Subsection~\ref{sssec_grading}.)
\begin{equation}\label{gs_1}
\begin{tikzpicture}[baseline=-\dimexpr\fontdimen22\textfont2\relax,scale=.6]
\draw[thick] (-1,1) to (1,-1);
\draw[white,line width=2mm] (-1,-1) to (1,1);
\draw[thick] (-1,-1) to (1,1);
\node[blue,rotate=45,below] at (-.7,-.6) {\scriptsize{m}};
\node[blue,rotate=-45,above] at (-.7,.6) {\scriptsize{n}};
\draw[dashed] (0,0) circle [radius=1.414];
\end{tikzpicture}%
-q^{(-1)^{m-n}}\;
\begin{tikzpicture}[baseline=-\dimexpr\fontdimen22\textfont2\relax,scale=.6]
\draw[thick] (-1,-1) to (1,1);
\draw[white,line width=2mm] (-1,1) to (1,-1);
\draw[thick] (-1,1) to (1,-1);
\node[blue,rotate=45,below] at (-.7,-.6) {\scriptsize{m}};
\node[blue,rotate=-45,above] at (-.7,.6) {\scriptsize{n}};
\draw[dashed] (0,0) circle [radius=1.414];
\end{tikzpicture}%
=\delta_{m,n}(q-1)\;
\begin{tikzpicture}[baseline=-\dimexpr\fontdimen22\textfont2\relax,scale=.6]
\draw[thick] (-1,-1) to (-.62,-.62) to [out=45,in=-45] (-.62,.62) to (-1,1);
\draw[thick] (1,-1) to (.62,-.62) to [out=135,in=-135] (.62,.62) to (1,1);
\node[blue,rotate=45,below] at (-.7,-.6) {\scriptsize{n}};
\node[blue,rotate=45,above] at (.7,.6) {\scriptsize{n}};
\draw[dashed] (0,0) circle [radius=1.414];
\end{tikzpicture}%
-\delta_{m,n+1}(1-q^{-1})\;
\begin{tikzpicture}[baseline=-\dimexpr\fontdimen22\textfont2\relax,scale=.6]
\draw[thick] (-1,-1) to (-.62,-.62) to [out=45,in=135] (.62,-.62) to (1,-1);
\draw[thick] (-1,1) to (-.62,.62) to [out=-45,in=-135] (.62,.62) to (1,1);
\draw (0,.25) to (0,.5);
\draw (0,-.25) to (0,-.5);
\node[blue,rotate=-45,below] at (.7,-.6) {\scriptsize{n}};
\node[blue,rotate=-45,above] at (-.7,.6) {\scriptsize{n}};
\draw[dashed] (0,0) circle [radius=1.414];
\end{tikzpicture}%
\tag{S1}
\end{equation}
\begin{equation}\label{gs_2}
\begin{tikzpicture}[baseline=-\dimexpr\fontdimen22\textfont2\relax,scale=.6]
\draw[thick] (.1,.1) to [out=45,in=90] (1,0) to [out=-90,in=-45] (0,0) to [out=135,in=90] (-1,0) to [out=-90,in=-135] (-.1,-.1);
\draw[dashed] (0,0) circle [radius=1.414];
\end{tikzpicture}
=(q-1)^{-1}\;
\begin{tikzpicture}[baseline=-\dimexpr\fontdimen22\textfont2\relax,scale=.6]
\draw[dashed] (0,0) circle [radius=1.414];
\end{tikzpicture}
\tag{S2}
\end{equation}
\begin{equation}\label{gs_3}
\begin{tikzpicture}[baseline=-\dimexpr\fontdimen22\textfont2\relax,scale=.6]
\draw[thick] (-1.414,0) to [out=0,in=-135] (0,0) to [out=45,in=0] (0,.8) to [out=180,in=135] (-.1,.1);
\draw[thick] (.1,-.1) to [out=-45,in=180] (1.414,0);
\draw[dashed] (0,0) circle [radius=1.414];
\end{tikzpicture}
=0
\tag{S3}
\end{equation}
These relations should be interpreted as follows.
The links involved are identical outside a small ball $B$ in $M$. 
Inside $B$ Darboux coordinates have been chosen, i.e. an identification with an open ball in standard contact $\RR^3$ with $\xi=\mathrm{Ker}(dz-ydx)$, and what is displayed is the projection to the $(x,y)$-plane.
Furthermore, at the visible crossings the two strands should be sufficiently close to each other in the sense that we can move them arbitrarily close via an isotopy supported in $B$. 

In the simplest case, when $M$ is the standard contact $\RR^3$, it follows from work of Rutherford~\cite[Theorem 3.1]{rutherford06} that skein module is freely generated by the empty link.
Thus, the class of an arbitrary link $L$ is equal to $R_L\emptyset$ for some $R_L\in\ZZ[q^{\pm},(q-1)^{-1}]$, which is, up to change of variables and normalization (described in Subsection~\ref{subsec_skeinrel}), the \textit{graded ruling polynomial} of $L$.
This invariant of graded Legendrian links can be defined more directly by counting \textit{graded normal rulings} of the front projection of $L$, see Chekanov--Pushkar~\cite{cp_4conj}.
Thus, for general threefold $M$ the image of a link in the skein can be viewed as the appropriate generalization of a knot polynomial.

The skein module has an algebra structure in the case where $M=S\times\RR$ with contact form $p_1^*\theta+p_2^*dz$
where $p_1$, $p_2$ are the projections to $S$ and $\RR$ respectively, $\theta$ is a 1-form on $S$ with $d\theta\neq 0$ pointwise (a \textit{Liouville form}), and $z$ is the standard coordinate on $\RR$.
Also the grading structure should be pulled back from a foliation $\eta$ on $S$.
The product $L_1L_2$ is defined by ``stacking $L_2$ on top of $L_1$'', i.e. translating $L_2$ in sufficiently far in the positive $z$-direction so that it is entirely above $L_1$ and then taking the union of the two links.

We also allow the following extension of our setup.
Suppose that $S$ has boundary and pick a discrete subset $N\subset\partial S$.
Allow links $L$ which are compact Legendrian curves with $\partial L\subset N\times \RR$.
Thus, when projecting to $S$, links should have endpoints in $N$.
Furthermore, impose the following \textit{boundary skein relations} near $N\times\RR$, where the dotted line is a part of $\partial S$ and $s(m-n):=(-1)^{m-n}$ if $m>n$ and $s(m-n):=0$ if $m\leq n$.
\begin{equation}\label{gs_1b}
\begin{tikzpicture}[baseline=-\dimexpr\fontdimen22\textfont2\relax,scale=.6]
\draw[thick] (0,0) to (1,-1);
\draw[white,line width=2mm] (-.1,-.1) to (1,1);
\draw[thick] (0,0) to (1,1);
\node[blue,rotate=45,above] at (.7,.6) {\scriptsize{m}};
\node[blue,rotate=-45,below] at (.7,-.6) {\scriptsize{n}};
\draw[densely dotted] (0,-1.414) to (0,1.414);
\draw[dashed] (0,0) circle [radius=1.414];
\end{tikzpicture}%
-q^{s(m-n)}\;
\begin{tikzpicture}[baseline=-\dimexpr\fontdimen22\textfont2\relax,scale=.6]
\draw[thick] (0,0) to (1,1);
\draw[white,line width=2mm] (-.1,.1) to (1,-1);
\draw[thick] (0,0) to (1,-1);
\node[blue,rotate=45,above] at (.7,.6) {\scriptsize{m}};
\node[blue,rotate=-45,below] at (.7,-.6) {\scriptsize{n}};
\draw[densely dotted] (0,-1.414) to (0,1.414);
\draw[dashed] (0,0) circle [radius=1.414];
\end{tikzpicture}%
=\delta_{m,n}(q-1)\;
\begin{tikzpicture}[baseline=-\dimexpr\fontdimen22\textfont2\relax,scale=.6]
\draw[thick] (1,-1) to (.62,-.62) to [out=135,in=-135] (.62,.62) to (1,1);
\node[blue,rotate=45,above] at (.7,.6) {\scriptsize{n}};
\draw[densely dotted] (0,-1.414) to (0,1.414);
\draw[dashed] (0,0) circle [radius=1.414];
\end{tikzpicture}%
\tag{S1b}
\end{equation}
\begin{equation}\label{gs_2b}
\begin{tikzpicture}[baseline=-\dimexpr\fontdimen22\textfont2\relax,scale=.6]
\draw[thick] (.1,.1) to [out=45,in=90] (1,0) to [out=-90,in=-45] (0,0);
\draw[densely dotted] (0,-1.414) to (0,1.414);
\draw[dashed] (0,0) circle [radius=1.414];
\end{tikzpicture}
=(q-1)^{-1}\;
\begin{tikzpicture}[baseline=-\dimexpr\fontdimen22\textfont2\relax,scale=.6]
\draw[densely dotted] (0,-1.414) to (0,1.414);
\draw[dashed] (0,0) circle [radius=1.414];
\end{tikzpicture}
\tag{S2b}
\end{equation}

\subsubsection*{The Fukaya category and its Hall algebra}

For the purpose of defining the Fukaya category we assume that $S$ is a compact surface with boundary and that the Liouville form $\theta$ is chosen so that its dual vector field points outward along $\partial S$.
As above, we have a finite set $N\subset\partial S$ and a foliation $\eta$ on $S$ which provides the grading.
Given a choice of ground field $\KK$ one defines two variants of the Fukaya category, $\mc F=\mc F(S,N,\theta,\eta,\KK)$ and $\mc F^{\vee}=\mc F^{\vee}(S,N,\theta,\eta,\KK)$, whose objects are compact graded Legendrian curves $L$ with $\KK$-linear local system $E$ of finite rank and Maurer--Cartan element (formal deformation) $\delta\in\Hom^1((L,E),(L,E))_{>0}$, and where for $\mc F$ we require $\partial L\subset(\partial S\setminus N)\times \RR$ and for $\mc F^{\vee}$ we require $\partial L\subset N\times\RR$.
The setup is described in more detail in Subsection~\ref{subsec_setup}.
Let us make two remarks to relate this to the existing literature.
First, in the approach to Fukaya categories of surfaces based on arc systems or ribbon graphs, see for example \cite{stz,dk_triangulated, hkk}, the category $\mc F$ is defined as a homotopy colimit, while $\mc F^{\vee}$ is defined as a homotopy limit in the category of dg-categories up to Morita equivalence.
Second, in the Legendrian knot theory literature Maurer--Cartan elements and rank one local systems appear in a different guise as augmentations of the Chekanov--Eliashberg DGA, see also Subsection~\ref{subsubsec_ce}.

Hall algebras were first considered by Steinitz~\cite{steinitz} and later rediscovered by Hall~\cite{hall_algebra}.
Their definition immediately generalizes to abelian categories satisfying suitable finiteness conditions.
For the case of triangulated dg-categories one needs to modify the naive definition to take into account negative $\Ext$-groups, as was pointed out by To\"en~\cite{toen_derived_hall}.
Conceptually, one replaces groupoid cardinality with homotopy ($\infty$-groupoid) cardinality as defined by Baez--Dolan~\cite{baez_dolan}.
Let us state the definition used here.
Assume $\mc C$ is an extension closed $A_\infty$ category over a finite field $\KK=\FF_q$ such that $\Ext^k(A,B)$ is finite-dimensional for all $A,B\in\Ob(\mc C)$, $k\in\ZZ$, and vanishes for $k$ less than some constant depending on $A,B$.
These conditions are satisfied for $\mc F^{\vee}(S,N,\theta,\eta,\KK)$ if $\KK$ is a finite field.
Define $\mathrm{Hall}(\mc C)$ to be the algebra with underlying $\QQ$-vector space with basis the set of isomorphism classes of objects in $\mc C$ and product
\[
[A]\cdot [C]:=\left(\prod_{i=0}^\infty\left|\Ext^{-i}(C,A)\right|^{(-1)^{i+1}}\right)\sum_{f\in\Ext^1(C,A)}\left[\mathrm{Cone}(C[-1]\xrightarrow{f} A)\right]
\]
where $[A]$ denotes the basis element of $\mathrm{Hall}(\mc C)$ corresponding to the isomorphism class of the object $A$.
We should note that this is not the formula of To\"en in~\cite{toen_derived_hall}, but gives an isomorphic algebra after rescaling the basis vectors.
Instead we are following the conventions of Kontsevich--Soibelman~\cite{ks}, specialized to the finite setting.
The difference comes in regarding the elements of the Hall algebra either as functions or as measures. 
We adopt the latter view.

As a general remark, there are some limitations to using the version of the Hall algebra based on counting, as opposed to  motivic/cohomological variants~\cite{ks}.
In the setting of Fukaya categories of punctured surfaces one encounters only union of tori (Artin--Tate motives), so a more sophisticated approach would require introducing a lot of machinery for a rather small payoff.
It would however, via the Serre polynomial, give structure constants depending on a formal variable $q$ instead of the number of elements of the finite field.

\subsection{Main result}

We fix a compact surface $S$ with boundary, $N\subset\partial S$, a Liouville form $\theta$ on $S$, a grading structure $\eta$, and a finite field $\KK$ as before.
If $L$ is a graded Legendrian link in $S\times \RR$, then the element $\Phi(L)\in\mathrm{Hall}(\mc F^{\vee})$ we assign to it is given, up to scalar factor, by the sum of all objects supported on $L$.
More precisely, we set
\[
\Phi(L)=(q-1)^{-|\pi_0(L)|}q^{-e(L)}\sum_{E}\sum_{\delta\in\mc{MC}(L,E)}[(L,E,\delta)]
\]
where $e(L)\in\ZZ$ is a self-intersection number, the first sum runs over all rank one $\KK$-linear local systems on $L$ and $\mc{MC}(L,E)$ is the set of Maurer--Cartan elements.
In the main text (Subsection~\ref{subsec_skeinhall}) we give a more conceptual definition and prove the explicit formula above.
The following is our central result, see Theorem~\ref{thm_homomorphism} in the main text.

\begin{theorem}
Let $S$ be a compact surface with boundary, $N\subset\partial S$, $\theta$ a Liouville form on $S$, $\eta$ a grading structure, and $\KK$ a finite field.
Then $\Phi$ defined above induces a homomorphism
\[
\Phi:\mathrm{Skein}(S,N,\theta,\eta)\otimes_{\ZZ[t^{\pm},(1-t)^{-1}]}\QQ\longrightarrow\mathrm{Hall}(\mc F^{\vee}(S,N,\theta,\eta,\KK))
\]
of $\QQ$ algebras where $t\mapsto q=|\KK|$.
\end{theorem}

The proof is based on a precise analysis of the relation between formal deformations (Maurer--Cartan elements) and geometric deformation (resolving an intersection point), see Proposition~\ref{prop_resolve}, which in turn uses the algebraic machinery of curved $A_\infty$-categories developed in Section~\ref{sec_algebra}.

An obvious question is whether this homomorphism is injective and/or surjective.
The dimension of the Hall algebra in the way it is defined here depends in general on the size $|\KK|$ of the finite field, while the dimension of the skein algebra does not.
In practice, one usually passes to a subalgebra of the Hall-algebra whose dimension is independent of $|\KK|$, so perhaps the image of $\Phi$ should be viewed as a better behaved substitute for the full Hall algebra.
In those cases where all objects of $\mc F^{\vee}$ are rigid, $\Phi$ does have a chance to be an isomorphism and indeed we show:

\begin{theorem}
If $S$ is a disk with $n+1=|N|$ marked points on the boundary, then $\Phi$ is an isomorphism.
Thus the graded Legendrian skein algebra of a disk with $n+1$ marked points on the boundary is, after specialization to a prime power $q$, isomorphic to the Hall algebra of the bounded derived category of $\FF_q$-linear representations of an $A_n$ type quiver.
\end{theorem}

See Theorem~\ref{thm_phi_disk} in the main text. 
By contrast: 

\begin{theorem}
Let $S=S^1\times [0,1]$ be the annulus, $N=\emptyset$, and $\eta$ the standard grading, then $\Phi$ is injective.
Thus, the graded Legendrian skein algebra of the annulus is isomorphic to a subalgebra of the Hall algebra of the bounded derived category of finite-dimensional $\KK[x^{\pm}]$-modules.
\end{theorem} 

This is Theorem~\ref{thm_annulus} in the main text.
The skein of the annulus is of particular interest because its elements correspond to various types of Legendrian satellite invariants.
The known relation between Legendrian satellite invariants and counts of representations of the Chekanov--Eliashberg differential algebra, see \cite{ng_rutherford, leverson_rutherford}, can presumably be extended or at least given additional justification by our result and its extension to the $\ZZ/n$-graded context. 

More generally, we propose the following conjecture.
\begin{conj}
The algebra homomorphism is injective for general surface $S$. 
\end{conj}
Work in progress by Ben Cooper and the author aims to prove this using gluing techniques for Hall algebras and skein algebras.

\subsection{Categorification}

In recent years, Fukaya category-type constructions have been increasingly applied to Legendrian knots~\cite{bourgeois_chantraine,shende_treumann_zaslow,nrssz}.
To a Legendrian knot $L$ in $\RR^3$ one assigns its \textit{augmentation category} $\mc C(L)_1$ whose objects are augmentations of the Chekanov--Eliasherg DGA.
This category has a geometric interpretation as (a rank one part of) the Fukaya--Seidel category with boundary condition $L$.
By a result of Ng--Rutherford--Shende--Sivek~\cite{nrss} the homotopy cardinality of $\mc C(L)_1$ is the graded ruling polynomial $R_L$.
Thus, the category $\mc C(L)_1$ is a kind of \textit{categorification} of the knot polynomial of $L$.
The full category $\mc C(L)$, which was defined in~\cite{shende_treumann_zaslow} in terms of constructible sheaves, should be thought of as categorifying the satellite invariants of $L$.

For more general contact threefolds, the generalization of the ruling polynomial is the image of the link in the skein.
By our main theorem this element in the skein is categorified by a functor $F:\mc C(L)_1\to \mc F^{\vee}$ from the category of the link to the Fukaya category in the sense that the pushforward along $F$ of the (weighted) counting measure on $\mc C(L)_1$ gives the element $\Phi(L)$ in $\mathrm{Hall}(\mc F^{\vee})$.
The following table summarizes this discussion.

\begin{center}
\begin{tabular}{c|c}
Classical & Categorical \\
\hline
ruling polynomial of $L$ & $\mc C(L)_1$ \\
satellite invariants of $L$ & $\mc C(L)$ \\
$\mathrm{Skein}(S\times \RR)$ & $\mc F^\vee(S)$ \\ 
$L\in\mathrm{Skein}(S\times\RR)$ & $F:\mc C(L)_1\to\mc F(S)$ 
\end{tabular}
\end{center}

\subsection{Further directions and speculation}

\subsubsection{$\ZZ/n$-grading}

In this work we restrict throughout to $\ZZ$-graded curves, but it seems plausible that everything extends to the $\ZZ/n$-graded case.
The skein relations \eqref{gs_1}, \eqref{gs_2}, \eqref{gs_3} make sense for $\ZZ/(2n)$-graded curves and with some tweaks one can get the odd case as well, see \cite{rutherford06}.
The boundary skein relations \eqref{gs_1b}, \eqref{gs_2b} do not immediately work in the periodic case though, and probably require more radical modification.
One the other hand, while $\ZZ/(2n)$-graded versions of the Fukaya category exist, defining the Hall algebra of say, a $\ZZ/2$-graded triangulated category is a famous problem.
Approaches of Bridgeland~\cite{bridgeland_quantum} and Kontsevich~\cite{kontsevich_hall} require additional structure and are thus not intrinsic to the periodic category itself.
At least in the cases where $N=\emptyset$, i.e. $S$ does not have any marked points on the boundary, the definition of the $\ZZ/(2n)$-graded skein is clear, and so a good test for any proposed definition of the Hall algebra of the $\ZZ/(2n)$-graded Fukaya category would be if the analogs of our results hold.

\subsubsection{The $q=1$ limit}

Our skein relations assign $(q-1)^{-1}$ to the unknot, and so do not immediately specialize to the classical limit $q=1$.
However, this could be seen as just a defect of our particular choice of model ($\ZZ[q^{\pm}]$-submodule) and we expect that a more suitable one can be found using ideas of Turaev~\cite{turaev}.
Ideally, we would like to have a definition of the ``Fukaya category of $S$ over $\FF_1$'' and its Hall algebra and compare this to the specialization of the correct model of the skein algebra. 
Perhaps the two are isomorphic.

One can be much more precise when replacing triangulated categories by their abelian subcategories.
Fukaya categories of surfaces often have bounded t-structures whose hearts are categories of representations of quivers with quadratic monomial relations~\cite{hkk}.
Categories of representations of quivers over $\FF_1$ as well as their Hall algebras can be defined, see the work of Szczesny~\cite{szczesny12} and also the very general approach of Dyckerhoff--Kapranov~\cite{dk_higher_segal}.
The idea is that the category of vector spaces over $\FF_1$ is the category of pointed sets $(X,x_0)$ and maps $(X,x_0)\to (Y,y_0)$ are functions $f:X\to Y$ with $f(x_0)=y_0$ and $f$ is injective away from the preimage of $y_0$.
In this way one gets a category which has many of the features of Abelian categories but where $\Hom$ is just a pointed set.
On the skein side one can restrict to those links which give objects in the heart of the chosen t-structure on the Fukaya category.
For this submodule one already has the correct model and the $q=1$ limit gives the above Hall algebra of the category of representations over $\FF_1$.
Extending this to the full triangulated category remains an intriguing problem.
The most basic question is what the right (for our purposes) notion of a ``triangulated category over $\FF_1$'' is.

\subsubsection{Higher dimensions}

The construction which assigns to a Legendrian link an element in the Hall algebra of the Fukaya category extends in principle to higher dimensional Liouville domains replacing $S$, though the definitions are less elementary.
Presumably one should also use a version of the Hall algebra based on cohomology instead of counting, as the moduli spaces of objects are no longer constructible unions of algebraic tori, but can have much more complicated motives.
We can define skein relations in a very implicit manner as those linear relations which hold among the images of links in the Hall algebra, then the skein algebra is by construction a subalgebra of the Hall algebra.
It is a priori not clear that these relations are generated by local ones.
It would be interesting to find a small set of generating relations in, say, the case of Legendrian surfaces.
The guiding philosophy is thus that
\[
Skein=Hall
\]
is a general phenomenon which we are studying here only in a relatively tame case.

\subsection{Outline}

Section~\ref{sec_algebra} provides background on $A_\infty$-categories and their Hall algebras.
We also discuss filtered $A_\infty$-categories with curvature and prove some basic results about them which are needed later.
Section~\ref{sec_fukaya} is on Fukaya categories of surfaces.
For our purposes we need to consider immersed curves (with Legendrian lift) together with local system and Maurer--Cartan element.
The main novelty in this section is the interplay between smoothing intersections and Maurer--Cartan elements, which provides the basis for proving that the skein relation \eqref{gs_1} holds in the Hall algebra.
Finally, in Section~\ref{sec_skein} we construct the homomorphism $\Phi$ from the Legendrian skein algebra to the Hall algebra and study its properties in special cases.
This section also contains some results about the skein of Legendrian tangles.

\subsection{Acknowledgments}
 
We thank Benjamin Cooper, Mikhail Kapranov, Ludmil Katzarkov, Maxim Kontsevich, Dan Rutherford, Peter Samuelson, and Vivek Shende for helpful discussions. 
Comments of an anonymous referee were very helpful in improving the text.
The author enjoyed the hospitality and perfect working conditions of the University of Vienna, IHES, and the University of Miami while writing this article.

\section{$A_\infty$-categories}
\label{sec_algebra}

This section contains the algebraic parts of the story.
In the first subsection we give a brief review of the language of $A_\infty$-categories and fix conventions.
Less standard material, on curved $A_\infty$-categories with $\RR$-filtered $\Hom$-spaces, is contained in Subsection~\ref{subsec_curved}.
Subsection~\ref{subsec_hc} is on homotopy cardinality and related notions in the context of $A_\infty$-categories.
Finally, in Subsection~\ref{subsec_hall} we review the definition of the Hall algebra of an extension closed $A_\infty$-category and discuss how certain slicings of the category give a tensor product decomposition of the Hall algebra.

\subsection{Definitions}
\label{subsec_def}

We will use the language of $A_\infty$-categories throughout, as these naturally appear in symplectic topology.
The purpose of this subsection is to review some basic definitions and fix notations and sign conventions, adopting those which are common in the Fukaya category literature, e.g. Fukaya--Oh--Ohta--Ono~\cite{FOOO} or Seidel~\cite{seidel08}.
For an introduction to $A_\infty$-categories see Keller~\cite{keller_a_infinity} and for a more thorough account Lef\`evre-Hasegawa~\cite{lefevre_hasegawa}.

All our categories will be small, or at least essentially small, and linear over a fixed field $\KK$.
An $A_\infty$\textbf{-category} $\mc A$ over $\KK$ is given by a set $\Ob(\mc A)$ of objects, a $\ZZ$-graded vector space $\Hom(A,B)$ for each pair of objects $A,B\in\Ob(\mc A)$, and structure maps
\[
m_n:\Hom(A_{n-1},A_n)\otimes\cdots\otimes\Hom(A_0,A_1)\to\Hom(A_0,A_n)
\]
of degree $2-n$, for each $n\ge 1$, satisfying the $A_\infty$-relations
\begin{equation} \label{A_infty_rels}
\sum_{i+j+k=n}(-1)^{\Vert a_k\Vert+\ldots+\Vert a_1\Vert}m_{i+1+k}(a_n,\ldots,a_{n-i+1},m_j(a_{n-i},\ldots,a_{k+1}),a_k,\ldots,a_1)=0
\end{equation}
where $\Vert a\Vert:=|a|-1$ is the degree in the bar resolution.

We will require $A_\infty$-categories to be \textbf{strictly unital} for convenience.
This means that there is a morphism $1_A\in\Hom^0(A,A)$ for each $A\in\Ob(\mc A)$ such that
\begin{gather*}
m_1(1_A)=0 \\
m_2(a,1_A)=(-1)^{|a|}m_2(1_A,a)=a,\quad a\in\Hom(A,A) \\
m_k(\ldots,1_A,\ldots)=0\text{ for }k\ge 3
\end{gather*}
Strictly unital $A_\infty$-categories with $m_k=0$ for $k\geq 3$ correspond to dg-categories via
\[
da=(-1)^{|a|}m_1(a),\qquad ab=(-1)^{|b|}m_2(a,b).
\]
Indeed, the first three $A_\infty$-relations correspond to $d^2=0$, the Leibniz rule, and associativity of the product.

\subsubsection{Twisted complexes}

There is a canonical way of enlarging an $A_\infty$-category to include extensions by any sequence of objects.
Let $\mc A$ be an $A_\infty$-category and $A_1,\ldots,A_n\in\Ob(\mc A)$.
An \textit{upper triangular deformation} of $A=A_1\oplus\cdots\oplus A_n$ or \textbf{twisted complex} is given by morphisms $\delta_{ij}\in\Hom^1(A_j,A_i)$, $i<j$, forming a strictly upper triangular matrix, $\delta$, such that the $A_\infty$ Maurer--Cartan equation
\begin{equation}
m_1(\delta)+m_2(\delta,\delta)+m_3(\delta,\delta,\delta)+\ldots+m_{n-1}(\delta,\ldots,\delta)=0
\end{equation}
holds.
Here, the structure maps $m_k$ are extended to matrices in the natural way, i.e.
\[
\left(m_k(\delta,\ldots,\delta)\right)_{i_0,i_k} := \sum_{i_0<i_1<\ldots<i_k}m_k\left(\delta_{i_0,i_1},\ldots,\delta_{i_{k-1},i_k}\right).
\]

Twisted complexes form an $A_\infty$-category, $\mathrm{Tw}(\mc A)$, which contains $\mc A$ as a full subcategory by mapping $A$ to its trivial deformation with $\delta=0$.
A morphism, $a$, of degree $d$ from $(A,\alpha)$ to $(B,\beta)$ is given by elements $a_{ij}\in\Hom^d(A_j,B_i)$.
Structure maps are given by ``inserting $\delta$'s everywhere'':
\begin{equation}\label{twisted_structure_maps}
\widetilde{m}_k(a_k,\ldots,a_1):=\sum_{n_0,\ldots,n_k\ge 0}m_{k+n_0+\ldots+n_k}(\underbrace{\delta_k,\ldots,\delta_k}_{n_k\text{ times}},a_k,\underbrace{\delta_{k-1},\ldots,\delta_{k-1}}_{n_{k-1}\text{ times}},a_{k-1},\ldots,a_1,\underbrace{\delta_0,\ldots,\delta_0}_{n_0\text{ times}})
\end{equation}
where $a_i\in\mathrm{Hom}((A_{i-1},\delta_{i-1}),(A_i,\delta_i))$.
Note that the sum above is finite since the $\delta_i$'s are strictly upper triangular.
The $A_\infty$ relations for the $\widetilde{m}_k$ follow from the $A_\infty$ relations for the $m_k$ and the Maurer--Cartan equations for the $\delta_i$'s.

Note that $\mathrm{Tw}(\mathrm{Tw}(\mc A))\cong \mathrm{Tw}(\mc A)$ by the natural equivalence of categories which combines several strictly upper-triangular matrices into one block-matrix.
In particular, $\mc C=\mathrm{Tw}(\mc A)$ is closed under extensions in the following sense. 
There is a zero-object and for every $A,B\in\Ob(\mc C)$ and $f\in\Hom^1(A,B)$ with $m_1(f)=0$ the twisted complex $A\xrightarrow{f} B$ is isomorphic to an object of $\mc C$.
If $\mc A$ is closed under shifts, then $\mathrm{Tw}(\mc A)$ is triangulated (closed under cones and shifts).

\subsection{Curved $A_\infty$-categories}
\label{subsec_curved}

In this subsection we discuss the formalism of curved $A_\infty$-categories.
A general deformation of an $A_\infty$-category has, in addition to the structure maps $m_n$, $n\geq 1$, \textit{curvature terms} $m_0(A)\in\Hom^2(A,A)$ satisfying a generalization of the usual $A_\infty$-equations, see \eqref{A_infty_rels_m0}.
Such categories typically do not have well defined homotopy categories, since $m_1^2\neq 0$, but there is a way of removing the curvature by ``recomputing'' the set of objects.
The new objects correspond to solutions to the $A_\infty$ Maurer--Cartan equation.
Since there are infinitely many terms involved, some topology is needed, and we will consider those coming from $\RR$-filtrations.

\subsubsection{$\RR$-filtrations}

By a \textbf{decreasing $\RR$--filtration} on a vector space $V$ we mean a collection of subspaces $V_{\geq\beta}\subset V$ for $\beta\in\RR$ such that 
\begin{enumerate}[1)]
\item
$V_{\geq\alpha}\supset V_{\geq\beta}$ for $\alpha\leq\beta$,
\item
$V_{\geq\beta}=\bigcap_{\alpha<\beta}V_{\geq\alpha}$ (semicontinuous),
\item
the set of $\beta\in\RR$ with $V_{\geq\beta}/V_{>\beta}\neq 0$, where $V_{>\beta}:=\bigcup_{\alpha>\beta}V_{\geq\alpha}$, is discrete in $\RR$,
\item
$\bigcap_{\beta}V_{\geq\beta}=0$ (separated),
\item 
$\bigcup_{\beta}V_{\geq\beta}=V$ (exhaustive),
\item
$\lim_{\leftarrow}V/V_{\geq\beta}=V$ (complete).
\end{enumerate}
An \textbf{$\RR$-filtered vector space} is a vector space $V$ with decreasing $\RR$-filtration as above, which gives a topology on $V$ as usual.

\subsubsection{Curvature terms}

A \textbf{curved $A_\infty$-category} $\mc C$ over a field $\KK$ is given by a set of objects $\Ob(\mc C)$, for each pair $X,Y\in\Ob(\mc C)$ an $\RR$-filtered vector space $\Hom(X,Y)$ over $\KK$ and structure maps 
\[
m_n:\Hom(A_{n-1},A_n)\otimes\cdots\otimes\Hom(A_0,A_1)\to\Hom(A_0,A_n)
\]
of degree $2-n$, for $n\ge 0$, satisfying the \textit{$A_\infty$-relations with curvature}
\begin{equation} \label{A_infty_rels_m0}
\sum_{i+j+k=n}(-1)^{\Vert a_k\Vert+\ldots+\Vert a_1\Vert}m_{i+1+k}(a_n,\ldots,a_{n-i+1},m_j(a_{n-i},\ldots,a_{k+1}),a_k,\ldots,a_1)=0
\end{equation}
where $i,j,k,n\geq 0$.
By slight abuse of notation we write $m_0$ or $m_0(X)$ for the image of $1$ under $m_0:\KK\to\Hom^2(X,X)$.
The first two $A_\infty$-relations with curvature are then
\begin{gather*}
m_1(m_0)=0 \\
m_1(m_1(a))+m_2(a,m_0)+(-1)^{\|a\|}m_2(m_0,a)=0
\end{gather*}
so in particular $m_1$ is in general not a differential.
Furthermore, we require $m_0(X)\in\Hom^2(X,X)_{>0}$ and
\begin{equation}
a_i\in\Hom(X_{i-1},X_{i})_{\geq\beta_i}\implies m_n(a_n,\ldots,a_1)\in\Hom(X_0,X_n)_{\geq\beta_1+\ldots+\beta_n}
\end{equation}
for $n\geq 1$, as well as the existence of strict units $1_X\in\Hom^0(X,X)_{\geq 0}$.

Let $\mc C$ be a  curved $A_\infty$-category.
A \textit{Maurer--Cartan element} or \textit{bounding cochain} for $X\in\Ob(\mc C)$ is a $\delta\in\Hom^1(X,X)_{>0}$ with
\begin{equation}\label{MC_eqn}
\sum_{k=0}^{\infty}m_k(\underbrace{\delta,\ldots,\delta}_{k\text{ times}})=0.
\end{equation}
where the sum converges since $\delta\in\Hom^1(X,X)_{\geq\epsilon}$ for some $\epsilon>0$ by our assumptions on $\RR$-filtrations, and structure maps are contracting.
Denote by $\mc{MC}(X)\subset \Hom^1(X,X)_{>0}$ the set of Maurer--Cartan elements (which could be empty).
Define $\widetilde{\mc C}$ to be the (filtered, uncurved) $A_\infty$-category whose objects are pairs $(X,\delta)$ with $X\in\Ob(\mc C)$ and $\delta\in\mc{MC}(X)$, morphisms
\[
\Hom_{\widetilde{\mc C}}((X,\delta),(Y,\gamma)):=\Hom_{\mc C}(X,Y)
\]
and structure maps $\widetilde{m}_k$ obtained by inserting the Maurer--Cartan elements as in \eqref{twisted_structure_maps}.

We will also consider the category $\widetilde{\mc C}_{\geq 0}$ which has the same objects as $\widetilde{\mc C}$, morphisms 
\[
\Hom_{\widetilde{\mc C}_{\geq 0}}(X,Y):=\Hom_{\widetilde{\mc C}}(X,Y)_{\geq 0}
\]
and structure maps are restrictions of those of $\widetilde{\mc C}$. 
Finally, the category $\mc C_0$ has the same objects as $\mc C$, morphisms 
\[
\Hom_{\mc C_0}(X,Y):=\Hom_{\mc C}(X,Y)_{\geq 0}/\Hom_{\mc C}(X,Y)_{> 0}
\]
and structure maps induced from $\mc C$.
We then have a diagram of uncurved $A_\infty$-categories and functors
\begin{equation}
\label{curved_zigzag}
\begin{tikzcd}
 & \widetilde{\mc C}_{\geq 0} \arrow[dl,"F"']\arrow[dr,"G"] & \\
\mc C_0 & & \widetilde{\mc C}
\end{tikzcd}
\end{equation}
The functor $F$ is given on objects by $(X,\delta)\mapsto X$, i.e. forgetting the Maurer--Cartan element, and on morphism is the quotient map 
\[
\Hom_{\mc C}(X,Y)_{\geq 0}\twoheadrightarrow\Hom_{\mc C}(X,Y)_{\geq 0}/\Hom_{\mc C}(X,Y)_{>0}
\]
while the functor $G$ is the identity on objects and the inclusion
\[
\Hom_{\mc C}(X,Y)_{\geq 0}\hookrightarrow\Hom_{\mc C}(X,Y)
\]
on morphisms.

\subsubsection{Transporting Maurer--Cartan elements}

Having defined curved $A_\infty$-categories, the next goal is to establish some properties of the functor $\widetilde{\mc C}_{\geq 0}\to\mc C_0$.
The following proposition is an Inverse Function Theorem-type result.

\begin{prop} \label{prop_isoreflect}
The functor $\widetilde{\mc C}_{\geq 0}\to\mc C_0$ is conservative: A closed map $f$ in $\widetilde{\mc C}_{\geq 0}$ is an isomorphism if and only if its reduction modulo $\Hom_{>0}$ in $\mc C_0$ is an isomorphism.
\end{prop}

A note on terminology: A closed morphism $f\in\Hom^0(X,Y)$ in an $A_\infty$-category is called an \textbf{isomorphism} if its image in the homotopy category, i.e. in $\mathrm{Ext}^0(X,Y)$, is an isomorphism. 
This is equivalent to $f$ having an inverse up to homotopy.

\begin{proof}
A map is an isomorphism if and only if its cone is a zero object.
In case $\mathrm{Cone}(f)$ does not exist in $\widetilde{\mc C}$, we can formally add it as a two-step twisted complex.
Thus it suffices to show that if $X\in\Ob\left(\mc C_0\right)$ is a zero object then $(X,\delta)\in\Ob(\widetilde{\mc C}_{\geq 0})$ is a zero object for any $\delta\in\mc{MC}(X)$.
To show this, we construct a series converging to an element $h\in\Hom_{\mc C}^{-1}(X,X)_{\ge 0}$ with $\widetilde{m}_1(h)=1_X$.

Suppose we already have an $h\in \Hom_{\mc C}^{-1}(X,X)_{\geq 0}$ such that
\begin{equation}\label{id_boundary}
\widetilde{m}_1(h)=1_X \mod \Hom_{\geq \beta}
\end{equation}
then the goal is to find $h'\in \Hom_{\mc C}^{-1}(X,X)_{\geq\beta}$ such that $h+h'$ solves \eqref{id_boundary}, but modulo terms in $\Hom_{\geq 2\beta}$, i.e.
\[
\widetilde{m}_1(h')=1_X-\widetilde{m}_1(h) \mod \Hom_{\geq 2\beta}
\]
The right-hand side of the above equation is in $\Hom_{\geq \beta}$, but also $\widetilde{m}_1$-closed, so existence of $h'$ follows from acyclicity of the complex $\Hom_{\geq\beta}/\Hom_{\geq 2\beta}$.
To see this, suppose $x\in\Hom_{\geq\beta}^k$ such that $m_1(x)\in\Hom_{\geq 2\beta}^{k+1}$, then
\begin{equation}
x=m_2(x,1)=m_2(x,m_1(h))=-m_1(m_2(x,h))\qquad \mod \Hom_{\geq 2\beta}
\end{equation}
so $x$ is a boundary. 
To finish the proof, an inductive argument and completeness of $\Hom_{\mc C}(X,X)$ give the desired $h$.
\end{proof}

The following proposition allows us to transport Maurer--Cartan elements along isomorphisms in $\mc C_0$.

\begin{prop} \label{prop_isolift}
The functor $F:\widetilde{\mc C}_{\geq 0}\to\mc C_0$ has the isomorphism lifting property: If $X,Y\in\Ob(\mc C)$, $\delta\in\mc{MC}(X)$, and $f_0\in\Hom_{\mc C_0}(X,Y)$ is an isomorphism, then there exist a $\gamma\in\mc{MC}(Y)$ and $f\in \Hom_{\widetilde{\mc C}_{\geq 0}}((X,\delta),(Y,\gamma))$ which is an isomorphism with $f=f_0 \mod \Hom_{>0}$.
\end{prop}

\begin{proof}
The idea is to construct a countable sequence of increasingly better approximations of $f$ and $\gamma$ and make use of completeness of the filtrations on $\Hom$-spaces.
By Proposition~\ref{prop_isoreflect} it suffices to ensure that $f$ is closed and $f=f_0\mod \Hom_{>0}$ --- such $f$ is then automatically an isomorphism.

Suppose we have already found $\gamma\in\Hom^1(Y,Y)_{>0}$ and $f\in\Hom^0(X,Y)_{\geq 0}$ such that
\begin{gather*}
\widetilde{m}_0(Y):=\sum_im_i(\underbrace{\gamma,\ldots,\gamma}_{i\text{ times}})=0\mod\Hom_{\geq\beta} \\
\widetilde{m}_1(f):=\sum_{i,j}m_{i+1+j}(\underbrace{\gamma,\ldots,\gamma}_{i\text{ times}},f,\underbrace{\delta,\ldots,\delta}_{j\text{ times}})=0 \mod\Hom_{\geq \beta} 
\end{gather*}
and $f$ is invertible up to terms in $\Hom_{\geq\beta}$.
We want to find $\gamma'\in\Hom^1(Y,Y)_{\geq \beta}$ and $f'\in\Hom^1(X,X)_{\geq \beta}$ such that $\gamma+\gamma'$ and $f+f'$ solve the above equations, but modulo terms in $\Hom_{\geq 2\beta}$, not just $\Hom_{\geq \beta}$.
Modulo $\Hom_{\geq 2\beta}$, the nonlinear terms in $\gamma'$ and $f'$ vanish, and we are left to solve
\begin{gather} \label{lift_eq1}
\widetilde{m}_0(Y)+\widetilde{m}_1\left(\gamma'\right)=0\mod \Hom_{\geq 2\beta} \\
 \label{lift_eq2}
\widetilde{m}_1(f)+\widetilde{m}_1\left(f'\right)+\widetilde{m}_2\left(\gamma',f\right)=0\mod\Hom_{\geq 2\beta}.
\end{gather}

By assumption, $f$ has a homotopy inverse $g$ up to terms in $\Hom_{\geq \beta}$ which implies that the induced map
\begin{equation*}
\varphi:\frac{\Hom(Y,Y)_{\geq\beta}}{\Hom(Y,Y)_{\geq 2\beta}}\to \frac{\Hom(X,Y)_{\geq \beta}}{\Hom(X,Y)_{\geq 2\beta}},\qquad x\mapsto \widetilde{m}_2(x,f)
\end{equation*}
is a homotopy equivalence of chain complexes with homotopy inverse $y\mapsto \widetilde{m}_2(y,g)$.
Now, since the $A_\infty$-equations hold for $\widetilde{m}_n$, we have in particular that $\widetilde{m}_1(\widetilde{m}_0(Y))=0$ and $\widetilde{m}_2(\widetilde{m}_0(Y),f)=\widetilde{m}_1(\widetilde{m}_1(f))$, i.e. $\varphi(\widetilde{m}_0(Y))$ is a boundary. 
Hence, since $\varphi$ is a chain homotopy equivalence, there is a $\gamma''$ which solves \eqref{lift_eq1}.
Furthermore, $\widetilde{m}_1(f)+\widetilde{m}_2\left(\gamma'',f\right)$ is then closed $\mod\Hom_{\geq 2\beta}$, so again using the fact that $\varphi$ is a chain homotopy equivalence we can find a closed $\gamma'''$ and $f'$ such that
\begin{equation*}
\widetilde{m}_1(f)+\widetilde{m}_1\left(f'\right)+\widetilde{m}_2\left(\gamma''+\gamma''',f\right)=0\mod\Hom_{\geq 2\beta}
\end{equation*}
thus $\gamma':=\gamma''+\gamma'''$ solves both equations \eqref{lift_eq1} and \eqref{lift_eq2}.
\end{proof}

\subsubsection{Gauge equivalence}

Given $X\in \Ob(\mc C)$ there is a sort of gauge group action on the set of Maurer--Cartan elements $\mc{MC}(X)$.
The analog of the gauge group is
\[
\mc G_X:=\left\{1+x\mid x\in\Hom^0(X,X)_{>0}\right\}
\]
on which $m_2$ gives a not necessarily associative composition.
Also consider for $\delta,\delta'\in\Hom^1(X,X)_{>0}$ the set
\begin{equation*}
I(\delta,\delta'):=\left\{1+x\mid x\in\Hom^0((X,\delta),(X,\delta'))_{>0},\widetilde{m}_1(1+x)=0\right\}
\end{equation*}
which, if $\delta,\delta'\in\mc{MC}(X)$, is the set of isomorphisms $(X,\delta)\to(X,\delta')$ which map to $1$ in $\mc C_0$.
We say that $\delta,\delta'\in\Hom^1(X,X)_{>0}$ are \textbf{gauge equivalent} if $I(\delta,\delta')\neq \emptyset$ and write
$\mc{MC}(X)/\mc G_X$ for the set of gauge equivalence classes in $\mc{MC}(X)$.

\begin{lemma}\label{lem_gauge_action}
Let $\delta\in\Hom^1(X,X)_{>0}$, $1+x\in \mc G_X$, then there is a unique $\delta'\in\Hom^1(X,X)_{>0}$ such that $1+x\in I(\delta,\delta')$ and furthermore $\delta'\in\mc{MC}(X)$ if and only if $\delta\in\mc{MC}(X)$.
\end{lemma}

\begin{proof}
By definition $1+x\in I(\delta,\delta')$ if and only if
\begin{equation*}
\widetilde{m}_1(1+x)=\delta'-\delta+\widetilde{m}_1(x)=0
\end{equation*}
which we write as
\begin{equation*}
\delta'=\delta-\sum_{i,j\geq 0}m_{i+1+j}(\underbrace{\delta',\ldots,\delta'}_{i\text{ times}},x,\underbrace{\delta,\ldots,\delta}_{j\text{ times}}).
\end{equation*}
This can be used to inductively solve for $\delta'$, since $x\in\Hom^0(X,X)_{\geq\epsilon}$ for some $\epsilon>0$, thus if $\delta'$ has been determined $\mod \Hom_{\geq\beta}$, then the right hand side is determined $\mod\Hom_{\geq\beta+\epsilon}$.
It is also clear that $\delta'$ is uniquely determined by $x$ and $\delta$.

Suppose $\delta\in\mc{MC}(X)$, i.e. $\widetilde{m}_0(X,\delta)=0$, then
\begin{equation*}
0=\widetilde{m}_1\left(\widetilde{m}_1(1+x)\right)=\widetilde{m}_2\left(\widetilde{m}_0(X,\delta'),1+x\right)-\widetilde{m}_2\left(1+x,\widetilde{m}_0(X,\delta)\right)
\end{equation*}
hence
\begin{equation*}
\widetilde{m}_0(X,\delta')=-\widetilde{m}_2\left(\widetilde{m}_0(X,\delta'),x\right)
\end{equation*}
so again using the fact that $x\in\Hom^0(X,X)_{\geq\epsilon}$ this implies that if $\widetilde{m}_0(X,\delta')\in\Hom_{\geq\lambda}$ then $\widetilde{m}_0(X,\delta')\in\Hom_{\geq\lambda+\epsilon}$, so $\widetilde{m}_0(X,\delta')=0$ by separatedness.
The reverse direction is similar.
\end{proof}

Note that as a consequence of the above Lemma we have 
\begin{equation*}
\mc G_X\cong \bigsqcup_{\delta'\sim\delta}I(\delta,\delta')
\end{equation*}
for any $\delta\in\mc{MC}_X$, where $\delta'\sim\delta$ means gauge equivalence.

\begin{lemma}\label{lem_gauge_iso}
Let $\delta,\delta',\delta''\in\mc{MC}(X)$ and $1+f\in I(\delta,\delta')$, then the map
\begin{equation*}
I(\delta',\delta'')\to I(\delta,\delta''),\qquad x\mapsto\widetilde{m}_2(x,1+f)
\end{equation*}
is an isomorphism.
Similarly, the map
\begin{equation*}
I(\delta'',\delta)\to I(\delta'',\delta'),\qquad x\mapsto\widetilde{m}_2(1+f,x)
\end{equation*}
is an isomorphism.
\end{lemma}

\begin{proof}
Suppose $\widetilde{m}_2(x,1+f)=y$, then
\begin{equation*}
x=y-\widetilde{m}_2(x,f)
\end{equation*}
which allows us to recursively solve for $x$ in terms of $y$, since $f\in\Hom_{\geq\epsilon}$ for some $\epsilon>0$.
Moreover, if $y=1\mod\Hom_{>0}$ then $x=1\mod\Hom_{>0}$ and if $\widetilde{m}_1(y)=0$ then $\widetilde{m}_1(x)=0$.
\end{proof}

\subsection{Homotopy cardinality}
\label{subsec_hc}

We begin with some remarks to put the definitions in this subsection into context.
Suppose $X$ is a space with $\pi_k(X)$ finite for all $k\geq 0$ and trivial for $k\gg 0$.
The \textit{homotopy cardinality} of $X$ is
\begin{equation}\label{homotopy_card}
\sum_{x\in\pi_0(X)}\prod_{k=1}^{\infty}\left|\pi_k(X,x)\right|^{(-1)^k}
\end{equation}
introduced in \cite{baez_dolan}.
By the homotopy hypothesis, homotopy types of spaces correspond to equivalence classes of $\infty$-groupoids.
A higher category $\mc C$ has an $\infty$-groupoid $\mathcal{I}(\mc C)$ of isomorphism, so one can, under finiteness conditions, ``count'' objects of $\mc C$ using \eqref{homotopy_card}.
In particular if $\mc C$ is a dg- or $A_\infty$-category, then 
\[
\pi_1(\mathcal I(\mc C),X)=\mathrm{Aut}(X)\subset \mathrm{Ext}^0(X,X),\qquad
\pi_{k+1}(\mathcal I(\mc C),X)=\mathrm{Ext}^{-k}(X,X),k\geq 1
\]
see \cite{toen_vaquie} for the case of dg-categories.
Since we are not interested here in the actual space $\mathcal I(\mc C)$ but only its homotopy cardinality, we will simply define everything in terms of $\mathrm{Ext}$-groups. 

So let $\mc C$ be an $A_\infty$-category over a finite field $\FF_q$.
Denote by $\mathrm{Iso}(\mc C)$ the set of isomorphism classes of objects in $\mc C$, and given $A\in\Ob(\mc C)$ denote its class in $\mathrm{Iso}(\mc C)$ by $[A]$.
We say that $\mc C$ is \textbf{locally left-finite} if $\dim\Ext^i(A,B)<\infty$ for all $A,B\in\Ob(\mc C)$, $i\in\ZZ$ and $\Ext^i(A,B)=0$ for $i$ less than some integer depending on $A,B\in\Ob(\mc C)$.
In this situation the \textbf{weighted counting measure} on $\mathrm{Iso}(\mc C)$ assigns to the singleton $\{X\}\subset\mathrm{Iso}(\mc C)$ the rational number
\[
\mu_{\mc C}(X):=|\mathrm{Aut}(X)|^{-1}\prod_{k=1}^\infty\left|\mathrm{Ext}^{-k}(X,X)\right|^{(-1)^{k+1}}
\]
where $\mathrm{Aut}(X)\subset\mathrm{Ext}^0(X,X)$ is the group of automorphisms of $X$.
We think of the vector space $\QQ\mathrm{Iso}(\mc C)$ of finite $\QQ$-linear combinations of elements of $\mathrm{Iso}(\mc C)$ as the space of (signed, $\QQ$-valued) finite measures on $\mathrm{Iso}(\mc C)$.
While $\mu_{\mc C}$ is in general not finite, we can use it to identify the space of finitely supported functions with the space of finite measures via $f\mapsto f\mu_{\mc C}$.

An $A_\infty$ functor $F:\mc C\to \mc D$ induces a linear map
\[
F_*:\QQ\mathrm{Iso}(\mc C)\longrightarrow \QQ\mathrm{Iso}(\mc D),\qquad F_*([A]):=[FA].
\]
If furthermore $\mc C$ and $\mc D$ are linear over $\FF_q$ and locally left-finite, and $F$ has the property that for any $[Y]\in\mathrm{Iso}(\mc D)$ there are only finitely many $[X]\in\mathrm{Iso}(\mc C)$ with $[FX]=[Y]$, then there is a linear map
\[
F^!:\QQ\mathrm{Iso}(\mc D)\longrightarrow \QQ\mathrm{Iso}(\mc C) \\
\]
\begin{align*}
F^!([Y]):&=\sum_{\substack{[X]\in\mathrm{Iso}(\mc C) \\ [FX]=[Y]}}\frac{|\Aut(Y)|}{|\Aut(X)|}\prod_{i\geq 1}\left(\frac{\left|\Ext^{-i}(Y,Y)\right|}{\left|\Ext^{-i}(X,X)\right|}\right)^{(-1)^i}[X] \\
&=\sum_{\substack{[X]\in\mathrm{Iso}(\mc C) \\ [FX]=[Y]}}\frac{\mu_{\mc C}(X)}{\mu_{\mc D}(Y)}[X].
\end{align*}
Consider the special case when $F:\mc C\to *$ is the functor to the final $A_\infty$-category, $*$, with a single object and $\Hom^k=0$, in particular $\QQ\mathrm{Iso}(*)=\QQ$.
In order for $F^!$ to be defined we need $\mc C$ to be locally finite and have only finitely many objects up to isomorphism.
Then $F^!(1)$ is the weighted counting measure and $F_*F^!(1)\in\QQ$ is the homotopy cardinality of $\mc C$.
We remark that if elements of $\QQ\mathrm{Iso}(\mc C)$ are interpreted a \textit{functions} rather then measures, one should instead use $F^*$, which is pullback of functions along the map $\mathrm{Iso}(\mc C)\to\mathrm{Iso}(\mc D)$, and $F_!$ which sends the delta function at $[X]\in\mathrm{Iso}(\mc C)$ to $\frac{\mu_{\mc C}(X)}{\mu_{\mc D}(Y)}[FX]$, c.f. \cite{toen_derived_hall}.

Our next goal is to establish a simpler formula for $F^!$ for a special class of functors.
Assume as before that $\mc C,\mc D$ are $A_\infty$-categories over a finite field $\FF_q$ which are locally left-finite and that the induced map $\mathrm{Iso}(\mc C)\to\mathrm{Iso}(\mc D)$ is finite--to--one.
Furthermore, we require that:
\begin{enumerate}[1)]
\item 
$F$ is full at the chain level, i.e. the maps $\Hom_{\mc C}(X,Y)\to\Hom_{\mc D}(FX,FY)$ are surjective.
\item
$F$ has the isomorphism lifting property: Given an isomorphism $f:FX\to Y$ in $\mc D$ there is an object $\widetilde{Y}\in\Ob(\mc C)$ with $F\widetilde Y=Y$ and an isomorphism $\tilde f:X\to \widetilde Y$ with $F(\tilde f)=f$.
\item
$F$ reflects isomorphisms: If $F(f):FX\to FY$ is an isomorphism then $f$ is an isomorphism.
\end{enumerate}
By the first assumption on $F$ we have an exact sequence of cochain complexes
\begin{equation*}
0\longrightarrow K^\bullet(X,Y)\longrightarrow \Hom^\bullet_{\mc C}(X,Y)\longrightarrow\Hom^\bullet_{\mc D}(FX,FY)\longrightarrow 0
\end{equation*}
for each $X,Y\in\Ob(\mc C)$, where 
\begin{equation*}
K^i(X,Y):=\mathrm{Ker}\left(\Hom^i_{\mc C}(X,Y)\to\Hom^i_{\mc D}(FX,FY)\right)
\end{equation*}
and thus long exact sequences
\begin{equation}\label{full_functor_les}
\ldots\longrightarrow HK^i(X,Y)\longrightarrow \Ext^i_{\mc C}(X,Y)\longrightarrow\Ext^i_{\mc D}(FX,FY)\longrightarrow\ldots.
\end{equation}
Given $Y\in\mc D$ let $F_Y$ be the set of equivalence classes of objects $X\in\mc C$ with $FX=Y$ where $X\sim X'$ if there is an isomorphism $f:X\to X'$ with $F(f)=1_Y$ in $\Hom^0(Y,Y)$ (equivalently: $F(f)=1_Y$ in $\Ext^0(Y,Y)$).

\begin{lemma}
\label{lem_upper_shriek_formula}
Let $F:\mc C\to\mc D$ be an $A_\infty$ functor satisfying the above conditions, then
\begin{equation*}
F^!([Y])=\sum_{X\in F_Y}\prod_{i\geq 0}\left|HK^{-i}(X,X)\right|^{(-1)^{i+1}}[X]
\end{equation*}
for any $[Y]\in\QQ\mathrm{Iso}(\mc D)$.
\end{lemma}

\begin{proof}
Let $Y\in\Ob(\mc D)$, $X\in\Ob(\mc C)$ with $FX=Y$, and $f:Y\to Y$ an isomorphism. 
By the isomorphism lifting property of $F$ there exists an object $\widetilde{Y}\in\Ob(\mc C)$ together with an isomorphism $\tilde{f}:X\to\widetilde Y$ such that $F(\tilde f)=f$.
The class of $\widetilde Y$ in $F_Y$ depends only on the class of $X$ in $F_Y$ and the class of $f$ in $\Aut(Y)$, and we get in this way an action of $\Aut(Y)$ on $F_Y$.
The set of orbits is
\begin{equation*}
F_Y/\Aut(Y)\cong\left\{[X]\in\mathrm{Iso}(\mc C)\mid [FX]=[Y]\right\}
\end{equation*}
while the stabilizer of $X\in F_Y$ is the image of the map $\Aut(X)\to\Aut(Y)$.
In particular $F_Y$ is a finite set for any $Y$.
Note also the exactness of the sequence 
\begin{equation*}
HK^0(X,X)\longrightarrow\Aut(X)\longrightarrow\Aut(Y)
\end{equation*}
where the first map is given by $f\mapsto 1_X+f$ and we use the assumption that $F$ reflects isomorphisms.
We conclude that 
\begin{align*}
\left|\left\{X'\in F_Y \mid [X']=[X]\right\}\right|&=\frac{|\Aut(Y)|}{\left|\mathrm{Im}(\Aut(X)\to\Aut(Y))\right|} \\
&=\frac{\Aut(Y)}{\Aut(X)}\left|\mathrm{Im}(HK^0(X,X)\to\Ext^0(X,X))\right|
\end{align*}
for $[X]\in\mathrm{Iso}(\mc C)$.

On the other hand, the long exact sequence \eqref{full_functor_les} gives
\begin{gather*}
\prod_{i\geq 1}\left(\frac{\left|\Ext^{-i}(Y,Y)\right|}{\left|\Ext^{-i}(X,X)\right|}\right)^{(-1)^i}=\left|\mathrm{Im}(HK^0(X,X)\to\Ext^0(X,X))\right|\prod_{i\geq 0}\left|HK^{-i}(X,X)\right|^{(-1)^{i+1}}.
\end{gather*}
Combing all this,
\begin{align*}
F^!([Y])&=\sum_{\substack{[X]\in\mathrm{Iso}(\mc C) \\ [FX]=[Y]}}\frac{|\Aut(Y)|}{|\Aut(X)|}\prod_{i\geq 1}\left(\frac{\left|\Ext^{-i}(Y,Y)\right|}{\left|\Ext^{-i}(X,X)\right|}\right)^{(-1)^i}[X] \\
&=\sum_{X\in F_Y}\prod_{i\geq 0}\left|HK^{-i}(X,X)\right|^{(-1)^{i+1}}[X]
\end{align*}
\end{proof}

\subsubsection{Counting Maurer--Cartan elements}

Suppose $\mc C$ is a curved $A_\infty$-category over a finite field $\FF_q$ and which is locally left-finite.
Let $F:\widetilde{\mc C}_{\geq 0}\to \mc C_0$ and $G:\widetilde{\mc C}_{\geq 0}\to \widetilde{\mc C}$ be the functors as in \eqref{curved_zigzag}.
The pull--push gives a map
\begin{equation*}
\QQ\mathrm{Iso}\left(\mc C_0\right)\xrightarrow{\quad F^!\quad }\QQ\mathrm{Iso}\left(\widetilde{\mc C}_{\geq 0}\right)\xrightarrow{\quad G_*\quad}\QQ\mathrm{Iso}(\widetilde{\mc C}).
\end{equation*} 
We want to show that this is well-defined and find a simpler formula.
By Proposition~\ref{prop_isoreflect} and Proposition~\ref{prop_isolift} we may apply Lemma~\ref{lem_upper_shriek_formula} to the functor $F$ to conclude that
\begin{equation*}
F^!([X])=\sum_{\delta\in\mc{MC}_X/\mc G_X}\prod_{i\geq 0}\left|H^{-i}\left(\Hom((X,\delta),(X,\delta))_{>0}\right)\right|^{(-1)^{i+1}}[(X,\delta)]
\end{equation*}
for any $X\in\Ob(\mc C_0)$.

\begin{prop}
\label{upper_shriek_formula}
Let $\mc C$ be a curved $A_\infty$-category over a finite field $\FF_q$ which is locally left-finite on the chain level, i.e. $\mathrm{Hom}^k(X,Y)$ is finite dimensional and vanishes for $k\ll 0$, then
\begin{equation*}
F^!(X)=\prod_{i\geq 0}\left|\Hom^{-i}(X,X)_{>0}\right|^{(-1)^{i+1}}\sum_{\delta\in\mc{MC}_X}[(X,\delta)]
\end{equation*}
where $F$ is the canonical functor $\widetilde{\mc C}_{\geq 0}\to\mc C_0$.
\end{prop}

\begin{proof}
By Lemma~\ref{lem_gauge_action} we have 
\begin{equation*}
\mc G_X\cong \bigsqcup_{\delta'\sim\delta}I(\delta,\delta')
\end{equation*}
for any $\delta\in\mc{MC}_X$, but Lemma~\ref{lem_gauge_iso} shows that each of the sets $I(\delta,\delta')$ is either empty or isomorphic to $I(\delta,\delta)$, thus
\begin{equation*}
\left|\mc G_X\right|=\left|\left\{\delta'\mid\delta'\sim\delta\right\}\right|\cdot \left|I(\delta,\delta)\right|.
\end{equation*}
Note that
\begin{gather*}
\left|\mc G_X\right|=\left|\Hom^0(X,X)_{>0}\right| \\
\left|I(\delta,\delta)\right|=\left|\left\{x\in\Hom^0((X,\delta),(X,\delta))_{>0}\mid \widetilde{m}_1(x)=0\right\}\right|=:c_\delta
\end{gather*}
and
\begin{equation*}
\prod_{i\geq 0}\left|H^{-i}\left(\Hom((X,\delta),(X,\delta))_{>0}\right)\right|^{(-1)^{i+1}}=
c_\delta^{-1}\prod_{i\geq 1}\left|\Hom^{-i}(X,X)_{>0}\right|^{(-1)^{i+1}}
\end{equation*}
hence
\begin{align*}
F^!([X])&=\sum_{\delta\in\mc{MC}_X/\mc G_X}\prod_{i\geq 0}\left|H^{-i}\left(\Hom((X,\delta),(X,\delta))_{>0}\right)\right|^{(-1)^{i+1}}[(X,\delta)] \\
&=\sum_{\delta\in\mc{MC}_X}\frac{\left|I(\delta,\delta)\right|}{|\mc G_X|}c_\delta^{-1}\prod_{i\geq 1}\left|\Hom^{-i}(X,X)_{>0}\right|^{(-1)^{i+1}}[(X,\delta)] \\
&=\prod_{i\geq 0}\left|\Hom^{-i}(X,X)_{>0}\right|^{(-1)^{i+1}}\sum_{\delta\in\mc{MC}_X}[(X,\delta)].
\end{align*}
\end{proof}

\subsection{Hall algebra}
\label{subsec_hall}

Let $\mc C$ be a locally left-finite $A_\infty$-category over a finite field $\FF_q$. 
Assume furthermore that $\mc C$ is closed under extensions and has a zero object.
Then we have a diagram of categories and functors
\begin{equation*}
\begin{tikzcd}
 & \mc C^{\mc A_2} \arrow[dl,"F"']\arrow[dr,"G"] & \\
\mc C\times \mc C & & \mc C
\end{tikzcd}
\end{equation*}
where $\mc C^{\mc A_2}$ is the category of exact triangles in $\mc C$, whose objects can be concretely represented by twisted complexes $C\xrightarrow{\delta} A$, $\delta\in\Hom^1(C,A)$, $m_1(\delta)=0$, which $F$ sends to the pair $(A,C)$ and $G$ sends to $\mathrm{Cone}(C[-1]\xrightarrow{\delta} A)$, which exists in $\mc C$ by assumption.
Passing to $\QQ\mathrm{Iso}(\mc C)$, the pull--push along the diagram gives a product map
\begin{equation*}
\QQ\mathrm{Iso}(\mc C)\otimes \QQ\mathrm{Iso}(\mc C)\xrightarrow{\quad F^!\quad }\QQ\mathrm{Iso}(\mc C^{\mc A_2})\xrightarrow{\quad G_*\quad}\QQ\mathrm{Iso}(\mc C).
\end{equation*} 
Using Lemma~\ref{lem_upper_shriek_formula} one derives the following explicit formula for the product, which can also be deduced from \cite{toen_derived_hall}.
\begin{equation}\label{hall_product}
[A]\cdot [C]=\left(\prod_{i=0}^\infty\left|\Ext^{-i}(C,A)\right|^{(-1)^{i+1}}\right)\sum_{f\in\Ext^1(C,A)}\left[\mathrm{Cone}(C[-1]\xrightarrow{f} A)\right]
\end{equation}
The vector space $\QQ\mathrm{Iso}(\mc C)$ together with this product is called the \textbf{Hall algebra} of $\mc C$, denoted $\mathrm{Hall}(\mc C)$.
This is an associative algebra (see below) with unit $[0]$, where $0\in\Ob(\mc C)$ is a zero object.

As noted above, we think of elements of $\QQ\mathrm{Iso}(\mc C)$ as \textit{measures}, following the convention of Kontsevich--Soibelman~\cite[Section 6.1]{ks}. 
To\"en~\cite{toen_derived_hall} uses instead $F^*$ and $G_!$, consistent with the point of view that elements of the Hall algebra are \textit{functions}.
Multiplication by the weighted counting measure, $f\mapsto f\mu_{\mc C}$, defines an isomorphism between the two Hall algebras. 

\begin{prop}
The Hall algebra is associative.
\end{prop}

\begin{proof}
The proof below is adapted from \cite{ks} with some simplifications. 
Passing to a quasi-equivalent category, we may assume that local left-finiteness holds on the chain level, i.e. each $\Hom^i(X,Y)$ is finite-dimensional and vanishes for $i\ll 0$, then
\begin{equation*}
[A]\cdot [C] =\left(\prod_{i=0}^\infty\left|\Hom^{-i}(C,A)\right|^{(-1)^{i+1}}\right)\sum_{\substack{f\in\Hom^1(C,A) \\ m_1(f)=0}}\left[\mathrm{Cone}(C[-1]\xrightarrow{f} A)\right].
\end{equation*}

Fix a triple of objects $A_1,A_2,A_3\in\Ob(\mc C)$ and consider the set $X_{123}$ of upper triangular deformations of $A_1\oplus A_2\oplus A_3$, i.e. triples $a_{12},a_{13},a_{23}$, $a_{ij}\in\Hom^1(A_j,A_i)$, with $m_1(a_{12})=0$, $m_1(a_{23})=0$, $m_1(a_{13})+m_2(a_{12},a_{23})=0$, c.f. \eqref{MC_eqn}.
Since $\mc C$ is assumed to be closed under extensions, each element of $X_{123}$ gives an object in $\mc C$ up to isomorphism.
We have
\begin{align*}
(\left[A_1\right]&\cdot \left[A_2\right])\cdot \left[A_3\right]= \\
&=\left(\prod_{i=0}^\infty\left|\Hom^{-i}(A_2,A_1)\right|^{(-1)^i}\right)^{-1}\sum_{\substack{a_{12}\in\Hom^1(A_2,A_1)\\ m_1(f)=0}}[A_2\xrightarrow{a_{12}} A_1]\cdot [A_3] \\
&=\left(\prod_{i=0}^\infty\left|\Hom^{-i}(A_2,A_1)\right|^{(-1)^i}\right)^{-1}\left(\prod_{i=0}^\infty\left|\Hom^{-i}(A_3,A_2\oplus A_1)\right|^{(-1)^i}\right)^{-1}\cdot \\
&\qquad\qquad\sum_{\substack{a_{12}\\ m_1(f)=0}}\sum_{\substack{a_{13},a_{23} \\ \widetilde{m}_1((a_{13},a_{23}))=0}} \left[A_3\xrightarrow{(a_{13},a_{23})}(A_2\xrightarrow{a_{12}} A_1)\right] \\
&=\left(\prod_{k=0}^\infty\prod_{\substack{i,j\in\{1,2,3\}\\ i<j}}\left|\Hom^{-k}(A_j,A_i)\right|^{(-1)^k}\right)^{-1}\sum_{D\in X_{123}}[D] \\
\end{align*}
and similarly
\begin{equation*}
\left[A_1\right]\cdot (\left[A_2\right]\cdot \left[A_3\right])=\left(\prod_{k=0}^\infty\prod_{\substack{i,j\in\{1,2,3\}\\ i<j}}\left|\Hom^{-k}(A_j,A_i)\right|^{(-1)^k}\right)^{-1}\sum_{D\in X_{123}}[D].
\end{equation*}
which completes the proof.
\end{proof}

The idea in the above proof generalizes to give a formula for the product $\left[A_1\right]\cdots \left[A_n\right]$ in terms of twisted complexes.
Other proofs of various flavors appear in \cite{toen_derived_hall}, \cite{xiao_xu}, and \cite{dk_higher_segal}.

\begin{example}
Let $\mc C$ be the category with only the zero object, then $\mathrm{Hall}(\mc C)=\QQ$.
\end{example}

\begin{example}
Let $\mc C=\mathrm{Perf}(\FF_q)$ be the category of finite-dimensional complexes of vector spaces over $\FF_q$.
Then $\mathrm{Hall}(\mc C)$ has generators $x_k:=\left[\FF_q[-k]\right]$, $k\in\ZZ$, and relations
\begin{gather*}
x_{k+1}x_k-q^{-1}x_kx_{k+1}=q-1,\qquad k\in\ZZ \\
x_{k+m} x_k=q^{(-1)^m}x_k x_{k+m},\qquad k\in\ZZ,m\geq 2.
\end{gather*}
The first is obtained from
\begin{gather*}
\left[\FF_q[-k-1]\right]\left[\FF_q[-k]\right]=\left[\FF_q[-k-1]\oplus\FF_q[-k]\right]+(q-1)\left[0\right] \\
\left[\FF_q[-k]\right]\left[\FF_q[-k-1]\right]=q\left[\FF_q[-k-1]\oplus\FF_q[-k]\right]
\end{gather*}
and similarly for the second.
\end{example}

\subsubsection{Slicings}

The underlying vector space of the Hall algebra often admits a tensor product decomposition coming from a \textit{slicing}.
This notion was introduced by Bridgeland~\cite{bridgeland07} and generalizes that of a t-structure.
More precisely, a \textbf{slicing} of a triangulated category $\mathcal C$ is given by a collection of full additive subcategories $\mc C_\phi$ such that
\begin{enumerate}[1)]
  \item 
  $\mathcal C_{\phi}[1]=\mathcal C_{\phi+1}$
  \item 
  If $\phi_1<\phi_2$, $E_i\in\mathcal C_{\phi_i}$, then $\mathrm{Hom}(E_2,E_1)=0$.
  \item 
  Every $E\in\mathcal C$ has a \textbf{Harder--Narasimhan} filtration:
  A tower of triangles
  \begin{equation*} \begin{tikzcd}
  0=E_0 \arrow{r} & E_1 \arrow{d}          & \cdots & E_{n-1} \arrow{r}  & E_n=E
  \arrow{d} \\
                  & A_1 \arrow[dotted]{ul} &        &                    & A_n
  \arrow[dotted]{ul}
  \end{tikzcd} \end{equation*}
  with $0\neq A_i\in\mathcal C_{\phi_i}$ and
  $\phi_1>\phi_2>\ldots>\phi_n$.
\end{enumerate}

The HN-filtrations are unique as a consequence of the other axioms.
As an example, any bounded t-structure can be interpreted as a slicing with $\mc C_\phi=0$ for $\phi\notin\ZZ$.
A slicing is part of the data of a Bridgeland stability condition, however most slicings do not come from stability conditions.

We consider slicings which satisfy the additional condition that
\begin{equation}\label{slicing_split}
\phi_1<\phi_2, E_i\in\mathcal C_{\phi_i}\qquad \implies\qquad \mathrm{Ext}^1(E_1,E_2)=0.
\end{equation}
For example, if $\mc A$ is a hereditary abelian category ($\mathrm{Ext}^{\geq 2}=0$) and $\mc C=D^b(\mc A)$, then the slicing defined by the standard bounded t-structure on $\mc C$ has this property.
The condition \eqref{slicing_split} implies that all Harder--Narasimhan filtrations are split, so
\[
\mc C=\bigoplus_{\phi\in\RR}\mc C_\phi
\]
and also that if $\phi_1>\phi_2>\ldots>\phi_n$ and $A_i\in \mc C_{\phi_i}$ then
\[
[A_1]\cdots [A_n]=c[A_1\oplus \cdots\oplus A_n]
\]
for some scalar $c$, where the product is taken in the Hall algebra.
We can conclude that, as a vector space
\begin{equation}\label{hall_tprod_decomp}
\mathrm{Hall}(\mc C)\cong \bigotimes_{\phi\in\RR}\mathrm{Hall}(\mc C_\phi)
\end{equation}
where the natural map from the right-hand side to the left-hand side is given by 
\[
[A_1]\otimes \cdots\otimes [A_n]\mapsto [A_1]\cdots [A_n]
\]
where $A_i\in \mc C_{\phi_i}$ and $\phi_1>\ldots>\phi_n$.

\section{Fukaya categories of surfaces}
\label{sec_fukaya}

In this section we discuss Fukaya categories of surfaces.
While there are several works which define and study them, see for example \cite{seidel08,abouzaid_surfaces,stz,dk_triangulated, hkk}, none of the existing approaches are entirely suitable for our purposes.
Fortunately, there are no novel ideas needed, just the right combination of existing ones, e.g. the use of Maurer--Cartan elements, see Subsection~\ref{subsec_mc}. 
What is really new, to our knowledge, is the general relation between smoothing of intersections and Maurer--Cartan elements which we discuss in Subsection~\ref{subsec_smoothing}.
Certain foundational issues in defining Fukaya categories of surfaces will not be addressed in detail here and we instead refer the reader to the reference given above.

\subsection{Setup and conventions}
\label{subsec_setup}

For our purposes it will be essential to have a version of the Fukaya category which is defined over arbitrary base field and $\ZZ$-graded, so that its Hall algebra is defined.
To provide an overview we start by listing the data which enters into the definition.
This is essentially the setup from \cite{hkk}, except that we also want an explicit choice of Liouville 1-form, which was suppressed there.
\begin{enumerate}[1)]
\item
$S$ \ldots compact surface with boundary
\item
$N\subset\partial S$ \ldots finite set of marked points
\item
$\theta$ \ldots Liouville 1-form on $S$
\item
$\eta\in\Gamma(S,\PP(TS))$ \ldots grading structure on $S$
\item
$\KK$ \ldots coefficient field
\end{enumerate}
Let us say a bit more about the third and fourth point.
The 1-form $\theta$ should have the property that $d\theta$ is nowhere vanishing (thus is an area form and provides an orientation of $S$) and that the Liouville vector field $Z$ characterized by
\[
i_Zd\theta=\theta
\]
is outward pointing along $\partial S$ (c.f. see \cite{seidel_biased}).
We can find such a $\theta$ provided $S$ is orientable and every component has non-empty boundary.
The 1-form $\theta$ provides a contact form $\alpha=p_1^*\theta+p_2^*dz$ on $S\times \RR$ where $z$ denotes the standard coordinate on the second factor and $p_1:S\times\RR\to S$, $p_2:S\times\RR\to\RR$ are the projection maps.

The grading structure on $S$ is needed to define the $\ZZ$-grading on morphisms of the Fukaya category and is given by a section $\eta$ of the projectivized tangent bundle $\PP(TS)$, i.e. a foliation on $S$.
The section $\eta$ provides each fiber of $\PP(TS)$ with a basepoint, so there is a well-defined fiberwise universal cover which we denote by $\widetilde{\PP(TS)}$. 
Given such a choice, there is a notion of a graded curve, which is an immersed curve $\gamma:I\to S$ together with a section $\tilde{\gamma}$ of $\gamma^*\widetilde{\PP(TS)}$ such that $\tilde{\gamma}(t)$ is a lift of the tangent space to the curve at $\gamma(t)$.
Thus, locally, there is a $\ZZ$-torsor of choices of gradings.
An immersed Legendrian curve $L$ in $S\times\RR$ projects to an immersed curve in $S$, so it makes sense to speak about gradings of $L$.
More intrinsically, we could replace $TS$ in the above discussion by the rank two subbundle $\xi=\mathrm{Ker}(\alpha)\subset T(S\times \RR)$ cut out by the contact form, which is canonically identified with $p_2^*TS$.

Given the above data we will sketch the definition of the partially wrapped Fukaya category $\mc F(S,N,\theta,\eta,\KK)$ and the infinitesimally wrapped category $\mc F^{\vee}(S,N,\theta,\eta,\KK)$.
If every component of $\partial S$ contains an element of $N$, then these two categories turn out to be isomorphic.
An object of either category is given by an graded Legendrian curve $L$ in $S\times \RR$ together with a local system of finite-dimensional $\KK$-vector spaces $E$ on $L$.
We require $L$ to be compact and embedded in $S\times \RR$ with $\partial L\subset \partial S\times\RR$.
In fact, for now we will also assume that the projection of $L$ to $S$ is also embedded and deal with the more complicated immersed case later.
Additionally $\partial L$ should be either disjoint from $N\times\RR$ or contained in it.
In the former case, $(L,E)$ belongs to $\mc F(S,N,\theta,\eta,\KK)$, while in the latter it belongs to $\mc F^{\vee}(S,N,\theta,\eta,\KK)$.

Before defining morphisms in the Fukaya category, we need a few more remarks about graded curves.

\subsubsection{Grading}
\label{sssec_grading}

Analogously to how one can assign $\pm 1$ to a transverse intersection point of oriented submanifolds, one can assign an integer to a simple crossing of graded curves.
Let $L_0=(I_0,\gamma_0,\tilde{\gamma}_0)$ and $L_1=(I_1,\gamma_1,\tilde{\gamma}_1)$ be graded immersed curves with transverse intersection at $x=\gamma_0(t_0)=\gamma_1(t_1)$. 
Then define the \textbf{intersection index}
\begin{equation}
i(L_0,t_0,L_1,t_1):=\lceil \tilde{\gamma}_1(t_1)-\tilde{\gamma}_0(t_0)\rceil:=\text{smallest }n\in\ZZ\text{ with }n+\tilde{\gamma}_0(t_0)>\tilde{\gamma}_1(t_1)
\end{equation}
where we use the fact that even though $\widetilde{\PP(T_xS)}$ is not canonically identified with $\RR$, it does have a total order (since $S$ is oriented) and action of $\ZZ$.
If $p\in S$ such that there are unique $t_0\in I_0$ and $t_1\in I_1$ with $p=p_1(\gamma(t_0))=p_1(\gamma(t_1))$ then we also write $i_p(L_0,L_1)$ for $i(L_0,t_0,L_1,t_1)$.

When depicting graded curves in the plane we may as well assume that $\eta$ is the horizontal foliation.
A grading on an immersed curve $\gamma:I\to\RR^2=\CC$ is then specified by a function $\phi:I\to\RR$ with $e^{\pi i\phi(t)}$ tangent to $\gamma(t)$.
To specify $\phi$ it suffices to label segments of $\gamma$ where $n=\lfloor \phi(t)\rfloor$ is constant by that integer, see Figure~\ref{fig_graded_curve_example}.
The figure also illustrates the following property: When following the curve (in either direction) and passing a point where the tangent becomes horizontal, the integer increases by one if turning left and decreases by one if turning right. 
Furthermore, the intersection index can be read off as in Figure~\ref{fig_graded_intersect}.

\begin{figure}[h]
\centering
\begin{tikzpicture}
\draw[thick] (0,0) to [out=45,in=180] (1,.7) to [out=0,in=90] (1.7,0) to [out=-90,in=0] (1,-.7) to [out=180,in=0] (-1,.7) to [out=180,in=90] (-1.7,0) to [out=-90,in=180] (-1,-.7) to [out=0,in=-135] (0,0);
\draw (-1,.6) to (-1,.8);
\draw (-1,-.6) to (-1,-.8);
\draw (1,.6) to (1,.8);
\draw (1,-.6) to (1,-.8);
\node [left] at (-1.7,0) {\scriptsize{0}};
\node [right] at (1.7,0) {\scriptsize{0}};
\node [rotate=-30,above] at (-.6,.5) {\scriptsize{-1}};
\node [rotate=30,below] at (-.6,-.5) {\scriptsize{1}};
\end{tikzpicture}
\caption{Example of grading of an immersed curve specified by integer labels.}
\label{fig_graded_curve_example}
\end{figure}
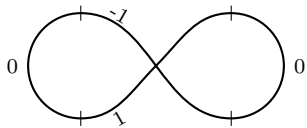

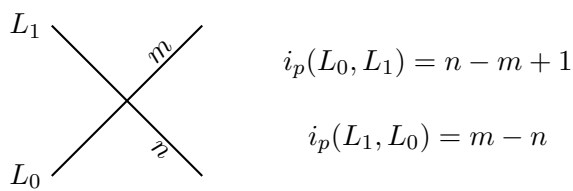
\begin{figure}[h]
\centering
\begin{tikzpicture}
\draw[thick] (-1,-1) to (1,1);
\draw[thick] (-1,1) to (1,-1);
\node [left] at (-1,-1) {$L_0$};
\node [left] at (-1,1) {$L_1$};
\node [rotate=45,above] at (.6,.5) {$m$};
\node [rotate=-45,below] at (.6,-.5) {$n$};
\node at (4,.5) {$i_p(L_0,L_1)=n-m+1$};
\node at (4,-.5) {$i_p(L_1,L_0)=m-n$};
\end{tikzpicture}
\caption{Index at intersection point $p$ of graded curves.}
\label{fig_graded_intersect}
\end{figure}

\subsubsection{Morphisms}
\label{subsubsec_morphisms}

Let $(L_0,E_0)$ and $(L_1,E_1)$ be as above and assume first that their projections to $S$ intersect transversely and $\partial L_i=\emptyset$.
To define morphisms from $(L_0,E_0)$ to $(L_1,E_1)$, we also need to make an auxiliary choice of orientation of $L_1$, then
\[
\Hom\left((L_0,E_0),(L_1,E_1)\right):=\bigoplus_{p\in p_1(L_0)\cap p_1(L_1)}\Hom_{\KK}\left((E_0)_p,(E_1)_p\right)[-i_p(L_0,L_1)]
\]
and if the orientation on $L_1$ is reversed we identify $x\in \Hom\left((L_0,E_0),(L_1,E_1)\right)$ with $(-1)^{|x|}x$, i.e. all odd morphisms get their sign reversed.
Thus, a morphism is formally an equivalence class of triples $(o_1,o_2,x)$ where $o_i$ is an orientation on $L_i$.

If either $p_1(L_0)$ is not transverse to $p_1(L_1)$ or both $L_0$ and $L_1$ have boundary, then it is necessary to perturb $L_0$ as graded Legendrian curve, which is equivalent to perturbing its projection to $S$ by a Hamiltonian diffeomorphism.
In particular this is always necessary when $L_0=L_1$.
Up to quasi-isomorphism, the resulting complex is independent of the choice of perturbation.
Let us describe how to perturb near the boundary, first for objects in $\mc{F}^{\vee}$. 
Endpoints of $L_0$ in $N\times \RR$ should be moved by a small amount along $\partial S$ in the direction of the natural induced orientation on the boundary.
For $\mc F$ there are two cases, depending on whether the component of $\partial S$ is question contains points of $N$ or not.
In the former case move the endpoint of $L_0$ along $\partial S\setminus N$ past any endpoints of $L_1$.
In the latter case it is better to enlarge $S$ to a non-compact surface $\widehat{S}$ as follows.
The backward flow of the vector field $Z$ provides an identification of an open neighborhood of $\partial S$ with $\partial S\times [0,\varepsilon)$ for $\epsilon>0$ sufficiently small and under which $\theta$ takes the from $e^{-t}\theta|_{\partial S}$.
Replacing $\partial S\times [0,\varepsilon)$ by $\partial S\times (-\infty,\varepsilon)$ we obtain a non-compact surface $\widehat{S}$ with 1-form $\hat{\theta}$, the \textit{completion} of $(S,\theta)$, which has infinite ends modeled on the standard half-cylinder.
Extend the curves to $\widehat{S}\times \RR$ so that the projection to $\widehat{S}$ is invariant under the flow outside a compact subset.
When computing morphisms, perturb $L_0$ by a Hamiltonian function of the form $t^2$ along the infinite ends, see Figure~\ref{fig_wrapping}, c.f. \cite{aaeko}.

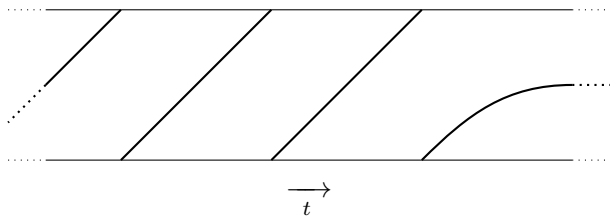
\begin{figure}[h]
\centering
\begin{tikzpicture}
\draw[dotted] (-4.5,1) to (-4,1);
\draw (-4,1) to (3,1);
\draw[dotted] (3,1) to (3.5,1);
\draw[dotted] (-4.5,-1) to (-4,-1);
\draw (-4,-1) to (3,-1);
\draw[dotted] (3,-1) to (3.5,-1);
\draw[thick,dotted] (3,0) to (3.5,0);
\draw[thick] (3,0) to [out=180,in=45] (1,-1);
\draw[thick] (1,1) to (-1,-1);
\draw[thick] (-1,1) to (-3,-1);
\draw[thick] (-3,1) to (-4,0);
\draw[thick,dotted] (-4,0) to (-4.5,-.5);
\node at (-.5,-1.5) {$\xrightarrow[t]{\quad}$};
\end{tikzpicture}
\caption{Wrapping of curves along infinite ends to compute morphisms in $\mc F$. Top and bottom horizontal lines are identified.}.
\label{fig_wrapping}
\end{figure}

At this point we have only used the projections of the curves to $S$.
One way in which we can use the Legendrian lifts is to define an $\RR$-filtration on $\Hom$-spaces, by declaring $\Hom\left((L_0,E_0),(L_1,E_1)\right)_{\geq\beta}$ to be generated by those intersection points with $z_0-z_1\geq\beta$, where $z_i$ is the $z$-coordinate of $L_i$ at the intersection point.
When computing the \textit{heights} $z_0-z_1$ we always regard the perturbation described above as negligibly small.
One way to make this precise is as follows.
Choose a family of perturbations $L_{i,t}$ of $L_i=L_{i,0}$, $t\in [0,1)$, so that $L_{i,t}\subset S$, $t\in(0,1)$, are related by a global isotopy of $S$, i.e. the topological structure remains the same and only areas of regions varies with $t$. (For example choose symplectic tubular neighborhoods of the $L_i$ and correspondingly linear families of perturbations.)
Then morphisms spaces and structure maps counting immersed polygons stay constant with $t$ and only the indexing of the filtrations changes. 
As $t\to 0$, the filtrations converge since $z$-values of $L_{i,t}$ over intersection points in $S$ converge.
The condition on the structure maps to be contractible (filtration preserving) is closed and hence holds in the limit.

\subsubsection{Structure maps}
\label{subsubsec_strmaps}

The $A_\infty$ structure maps of the Fukaya category are defined in terms of immersed polygons with boundary on the given Lagrangian curves.
More precisely let $L_k$, $k=0,\ldots,n$ be graded Legendrian curves intersecting transversely and let $x_k\in p_2(L_k)\cap p_2(L_{k+1})$, $k=0,\ldots,n-1$ and $x_n\in p_2(L_0)\cap p_2(L_n)$.
Consider a smoothly immersed $n+1$-gon $\phi:D\to S$, up to reparameterization, such that $\phi$ sends the $k$-th corner of $D$ to $x_k$ and the side of $D$ from the $(k-1)$-st to the $k$-th corner to $L_k$, see Figure~\ref{fig_disk}.
For each intersection point there is an integer degree $d_k:=i_{x_k}(L_k,L_{k+1})$, $k=0,\ldots,n-1$, and $d_n=i_{x_n}(L_0,L_n)$, and they satisfy
\[
d_n=d_0+\ldots+d_{n-1}+2-n
\]
for topological reasons.
Suppose furthermore that we have chosen orientations of the $L_i$, then a sign $(-1)^{s(D)}\in \ZZ/2$ is defined as follows.
For $k=0,\ldots,n-1$ add $1$ to $s(D)$ if $d_k$ is odd and the orientation on $L_{k+1}$ does not match the counterclockwise orientation around $\partial D$ under $\phi$. 
Also add $1$ to $s(D)$ if $d_n$ is odd and the orientation on $L_{n}$ does not match the counterclockwise orientation around $\partial D$ under $\phi$. (The orientation of $L_0$ is not used.)
This is the same as the sign convention in \cite{seidel08}.
Finally, let $E_k$ be a local systems on $L_k$, $k=0,\ldots,n$, then parallel transport along the edges of $D$ defines a map
\[
\tau(D):\Hom\left((E_{n-1})_{x_{n-1}},(E_n)_{x_{n-1}}\right) \otimes\cdots\otimes \Hom\left((E_0)_{x_0},(E_1)_{x_0}\right)\to \Hom\left((E_{0})_{x_n},(E_n)_{x_n}\right)
\]
and if $X_k:=(L_k,E_k)$ are objects of the Fukaya category we put
\begin{equation}
m_n:\Hom(X_{n-1},X_n)\otimes\cdots\otimes \Hom(X_0,X_1)\to\Hom(X_0,X_n),m_n:=\sum_{D}(-1)^{s(D)}\tau(D)
\end{equation}
where the sum is over all intersection points and immersed disks up to reparameterization as above.

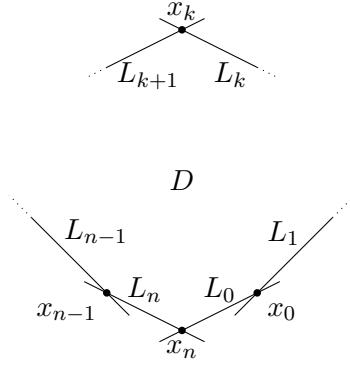
\begin{figure}
\centering
\begin{tikzpicture}
\draw (-.3,-2.15) to (1.3,-1.35);
\draw (.7,-1.8) to (2,-.5);
\draw[dotted] (2,-.5) to (2.3,-.2);
\draw (.3,-2.15) to (-1.3,-1.35);
\draw (-.7,-1.8) to (-2,-.5);
\draw[dotted] (-2,-.5) to (-2.3,-.2);
\draw (-.3,2.15) to (1,1.5);
\draw[dotted] (1,1.5) to (1.3,1.35);
\draw (.3,2.15) to (-1,1.5);
\draw[dotted] (-1,1.5) to (-1.3,1.35);

\fill (0,-2) circle [radius=.05];
\fill (1,-1.5) circle [radius=.05];
\fill (-1,-1.5) circle [radius=.05];
\fill (0,2) circle [radius=.05];

\node[below] at (0,-2) {$x_n$};
\node[below right] at (1,-1.5) {$x_0$};
\node[below left] at (-1,-1.5) {$x_{n-1}$};
\node[above] at (0,2) {$x_k$};
\node[above] at (.5,-1.75) {$L_0$};
\node[above] at (-.5,-1.75) {$L_n$};
\node[above left] at (1.7,-1) {$L_1$};
\node[above right] at (-1.7,-1) {$L_{n-1}$};
\node[below left] at (1,1.7) {$L_k$};
\node[below right] at (-1,1.7) {$L_{k+1}$};
\node at (0,0) {$D$};
\end{tikzpicture}
\caption{Immersed disk labeling conventions.}
\label{fig_disk}
\end{figure}

For the above to be well-defined, we need to know that the set of disks with fixed $L_0,\ldots,L_n$ and $x_0,\ldots,x_{n-1}$ is finite.
For this, the assumption that all curves involved are Legendrian is essential.
Suppose first that the $L_i$ are not infinitely wrapped around cylindrical ends, hence compact.
The area of an immersed disk $\phi:D\to S$ as above is
\[
\int_D\phi^*d\theta=\int_{\partial D}\phi^*\theta=-\sum_{k=0}^n\left(z(L_k,x_k)-z(L_k,x_{k-1})\right)
\]
where $z(L_k,x_k)$ is the $z$-coordinate of $L_k$ over $x_k$ and we use the fact that $\theta=-dz$ along Legendrian curves in $S\times \RR$.
In particular, since $z$ is bounded along all curves, the area of $D$ is bounded by some constant depending only on the $L_k$. But if there were infinitely many polygons, their areas would necessarily tend to infinity, regardless of whether the $L_k$ are Legendrian.
In the case of infinite wrapping there can be infinitely many disks, but only finitely many for fixed choice of intersection points $x_0,\ldots,x_{n-1}$, so the structure maps are still well-defined.

\subsubsection{Relation to ribbon graph approach}

In~\cite{hkk} the category $\mc F$ was defined using a choice of arc system, i.e. decomposition of the surface into polygons, and was also shown to be Morita equivalent to a certain homotopy colimit over a ribbon graph as a special case of the ``Lagrangian skeleton'' approach to Fukaya categories proposed by Kontsevich~\cite{kontsevich_sympalg}. 
For the approach using arc systems it is also convenient to allow $S$ to have corners and replace set $N\subset\partial S$ of marked points by marked intervals connecting corners of $S$.
Morphisms were defined in \cite{hkk} as paths along the boundary instead of explicitly perturbing the arcs and taking intersection points.
There are also some minor differences in convention between \cite{hkk} and the present paper, regarding the grading, direction of wrapping, and signs.

An arc system provides a generator of $\mc F$ given by the direct sum of all arcs (with arbitrary grading and trivial rank one local system).
Generation statements of this type are by now well established in the Fukaya category literature~\cite{aaeko,lee}.
The category $\mc F^{\vee}$ is equivalent to the Morita dual of $\mc F$: It is the category of functors from $\mc F$ to the category of finite dimensional complexes over $\KK$.
Finding the functor corresponding to an object $(L,E)$ in $\mc F^{\vee}$ as define above, i.e. to find the image under the Yoneda-embedding, amounts to intersecting $L$ with all the arcs in the arc system and counting disks which have one side on $L$ and the remaining sides on the arcs.
We will not use the equivalence of the two approaches directly, except in the case of the disk and annulus (see Section~\ref{sec_skein}).

\subsection{Immersed curves and Maurer--Cartan elements}
\label{subsec_mc}

If we allow immersed curves, in particular ones bounding immersed 1-gons (``teardrops''), we get a curved $A_\infty$-category.
According to the general philosophy, the true objects depend on an additional choice of Maurer--Cartan element.

Let $L\subset S\times \RR$ be an embedded graded Legendrian curve with boundary in $\partial S\subset N\times\RR$.
The projection $p_1(L)\subset S$ is not required to be embedded, but should have only transverse self-crossings.
Also fix a local system of vector spaces $E$ on $L$.
The additional choice alluded to above is a $\delta\in\Hom^1((L,E),(L,E))_{>0}$ satisfying the $A_\infty$ Maurer--Cartan equation.
In order the define or compute $\Hom((L,E),(L,E))$, we need to perturb $L$ slightly to some $L'$ so that $p_1(L')$ is transverse to $p_1(L)$. 
Fortunately, the positive part $\Hom((L,E),(L,E))_{>0}$ has an alternative definition which does not require us to perturb $L$.
This is already evident from Chekanov's definition in~\cite{chekanov} which does not require a choice of perturbation.
Namely, define $\Hom((L,E),(L,E))_{>0}$ as a direct sum over self-intersection points of $p_1(L)$ with summands as in Subsection~\ref{subsubsec_morphisms}, going from the upper branch over the self-intersection to the lower branch, and define the structure of a (non-unital) curved $A_\infty$-algebra on this filtered graded vector space as in Subsection~\ref{subsubsec_strmaps} by counting immersed disks with boundary on $L$ and corners at self-intersection points.

\begin{prop}
The curved $A_\infty$-algebra $\Hom((L,E),(L,E))_{>0}$ is ismorphic to $\Hom((L,E),(L',E))_{>0}$ for small perturbations $L'$ of $L$.
\end{prop}

\begin{proof}
The set of intersection points $p_1(L)\cap p_2(L')$ is naturally partitioned into two types: Self-intersection points of $L$ and new intersection points depending on the choice of $L'$.
For a small perturbation the difference in $z$-coordinate of $L$ and $L'$ over an intersection point of the latter type will be much smaller than over an intersection point of the former type, and we always assume that $L'$ is chosen so that this is case.
Also, it is convenient to choose the perturbation resulting in the minimal number of new intersection points, which is two for every component of $L$ diffeomorphic to $S^1$ and one for every component diffeomorphic to $[0,1]$.
When defining the filtration on $\Hom((L,E),(L',E))$ we regard the perturbation as having negligible effect on the $z$-coordinates, c.f. the discussion in Subsection~\ref{subsubsec_morphisms}.
Hence new intersection points contribute to $\Hom_{\geq 0}/\Hom_{>0}$ while self-intersections of $L$ contribute to $\Hom_{>0}$ (upper branch to lower branch) and $\Hom/\Hom_{\geq 0}$ (lower branch to upper branch), which are dual to each other.
The immersed disks which contribute to the structure constants of $\Hom((L,E),(L,E))_{>0}$ and  $\Hom((L,E),(L',E))_{>0}$ differ from each other by small perturbation  and are in particular in bijection with one another.
\end{proof}

Given a $\delta\in\Hom^1((L,E),(L,E))_{>0}$ satisfying the Maurer--Cartan equation~\eqref{MC_eqn}, the triple $(L,E,\delta)$ defines an object in $\mc F^\vee$.
Structure maps are obtained by twisting the structure maps $m_k$ coming from disk counts by the $\delta$'s.

For fixed $L$ we can consider the category $\mc C(L)_1$ whose objects are all choices of local system $E$ and Maurer--Cartan element $\delta\in\Hom^1((L,E),(L,E))_{>0}$, and whose morphisms are 
\[
\Hom_{\mc C(L)_1}((L,E,\delta),(L,E',\delta')):=\Hom((L,E),(L,E'))_{\geq 0}.
\]
This is called the \textit{augmentation category} in \cite{nrssz}.
It is clear from the definition that there is a natural functor $\mc C(L)_1\to\mc F^{\vee}$.
As the notation suggests, there is a bigger category $\mc C(L)$ where $E$ is allowed to have arbitrary rank and be $\ZZ$-graded, so $\mc C(L)$ is independent of the grading on $L$.
A category which turns out to be equivalent to $\mc C(L)$ was defined in \cite{stz}, for $L\subset\RR^3$, in terms of constructible sheaves with singular support on the front projection of $L$.

\subsubsection{Example: Trefoil}

To illustrate the definition we consider the simple but non-trivial example of the (right-handed) trefoil knot $L$, see Figure~\ref{fig_trefoil}.
Equip $L$ with a rank one local system $E$ with monodromy $\lambda\in\KK^\times$ which is trivialized away from some point on the right tear-drop in Figure~\ref{fig_trefoil}.
A basis of $\Hom((L,E),(L,E))_{>0}$ is given by the self-intersection points $u,v,x,y,z$ where $|u|=|v|=2$ and $|x|=|y|=|z|=1$.
Looking for possible immersed disks one finds the following non-zero terms of the $A_\infty$-structure:
\[
m_0=\lambda v+u,\quad m_1(x)=m_1(z)=u-v,\quad m_3(x,y,z)=u,\quad m_3(z,y,x)=-v
\]
The Maurer--Cartan equations for $\delta=ax+by+cz$ are thus
\[
\begin{cases}
1+a+c+abc=0 \\
\lambda-a-c-abc=0
\end{cases}
\]
which have a solution iff $\lambda=-1$, in which case $\mc{MC}(L,E)$ is a rational surface in $\AA^3$.

\begin{figure}[h]
\centering
\begin{tikzpicture}
\draw[thick] (-.1,.1) to [out=135,in=-135] (0,1) to [out=45,in=-180] (.5,1.2) to [out=0,in=125] (1.8,0) to [out=-55,in=-180] (3.5,-1.5) to [out=0,in=-90] (5,0) to [out=90,in=0] (3.5,1.5) to [out=180,in=55] (1.9,.1);
\draw[thick] (1.7,-.1) to [out=-125,in=0] (.5,-1.2) to [out=180,in=-45] (.1,-1.1);
\draw[thick] (-.1,-.9) to [out=135,in=-135] (0,0) to [out=45,in=-45] (.1,.9);
\draw[thick] (-.1,1.1) to [out=135,in=0] (-.5,1.2) to [out=180,in=55] (-1.7,.1);
\draw[thick] (-1.9,-.1) to [out=-125,in=0] (-3.5,-1.5) to [out=180,in=-90] (-5,0) to [out=90,in=180] (-3.5,1.5) to [out=0,in=125] (-1.8,0) to [out=-55,in=180] (-.5,-1.2) to [out=0,in=-135] (0,-1) to [out=45,in=-45] (.1,-.1);
\draw (-3.5,1.4) to (-3.5,1.6);
\draw (-3.5,-1.4) to (-3.5,-1.6);
\draw (-.5,1.1) to (-.5,1.3);
\draw (-.5,-1.1) to (-.5,-1.3);
\draw (3.5,1.4) to (3.5,1.6);
\draw (3.5,-1.4) to (3.5,-1.6);
\draw (.5,1.1) to (.5,1.3);
\draw (.5,-1.1) to (.5,-1.3);
\node[blue] at (.4,.25) {\scriptsize{0}};
\node[blue] at (2.25,1) {\scriptsize{1}};
\node[blue] at (2.25,-1) {\scriptsize{-1}};
\node[blue] at (-.4,.25) {\scriptsize{0}};
\node[blue] at (-2.25,1) {\scriptsize{-1}};
\node[blue] at (-2.25,-1) {\scriptsize{1}};
\node[blue,left] at (-5,0) {\scriptsize{0}};
\node[blue,right] at (5,0) {\scriptsize{0}};
\node[red,left] at (-1.8,0) {$u$};
\node[red,right] at (1.8,0) {$v$};
\node[red,above] at (0,1) {$x$};
\node[red,below] at (0,0) {$y$};
\node[red,below] at (0,-1) {$z$};
\end{tikzpicture}
\caption{Lagrangian projection of a Legendrian trefoil.}
\label{fig_trefoil}
\end{figure}
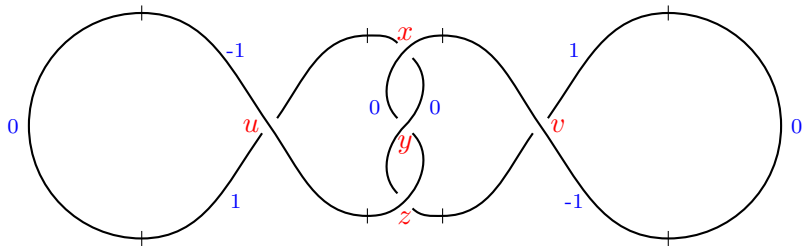

For given $\delta$ satisfying the Maurer--Cartan equation, $(L,E,\delta)$ represents the zero object in the Fukaya category.
Note however that the category $\mc C(L)$ is non-trivial and an interesting invariant of the knot.

\subsubsection{Relation to the Chekanov--Eliashberg DGA}
\label{subsubsec_ce}

We begin with a general algebraic construction.
Let $A$ be a curved $A_\infty$-algebra (curved $A_\infty$-category with a single object) such that $A=A_{>0}$.
Assume that $A$ is finite-dimensional, then $m_k=0$ for $k\gg 0$.
The bar resolution $BA$ of $A$ is the free coalgebra with underlying vector space
\[
BA:=\bigoplus_{k=0}^{\infty}\left(A[1]\right)^{\otimes k}
\]
and differential $d:BA\to BA$ of degree $1$ whose $A[1]\to (A[1])^{\otimes k}$-component is $m_k$.
The component-wise dual
\[
BA^{\vee}:=\bigoplus_{k=0}^{\infty}\left(A^{\vee}[-1]\right)^{\otimes k}
\]
is a differential graded algebra.

If we apply the above to $A:=\mathrm{Hom}((L,E),(L,E))_{>0}$ we obtain (an algebra isomorphic to) the Chekanov--Eliashberg algebra of $L$, more precisely the refined version of \cite{ees}, with formal parameters $t_i$ specialized to the monodromy of $E$. 
This is clear at least over $\ZZ/2$, since both definitions involve counting the same disks.
Presumably, the signs can be made to agree as well and the statement is true in arbitrary characteristic.

\subsection{Smoothing intersections}
\label{subsec_smoothing}

In this subsection we discuss the relation between smoothing of intersection points and Maurer--Cartan elements.
The case of an intersection point in the interior of $S$ is treated first, and the case of an intersection point on the boundary of $S$ second.

\subsubsection{Interior point}

Suppose $L=(I,\gamma,\tilde{\gamma})$ is a graded Legendrian curve where $\gamma:I\to S\times\RR$ is an embedding except for a transverse self-intersection point at $(x,z):=\gamma(t_0)=\gamma(t_1)$.
Also assume that the grading $\tilde{\gamma}$ satisfies
\[
i(L,t_0,L,t_1)=1
\]
then there are three ways of resolving the singularity  at $(x,z)$ by modifying $L$ in a neighborhood of that point:
\begin{enumerate}[1)]
\item
$L_+$ obtained by pushing the $t_0$-strand above the $t_1$-strand
\item
$L_-$ obtained by pushing the $t_0$-strand below the $t_1$-strand
\item
$L_s$ obtained by cutting both strands and reconnecting them with each other (only one of the two ways of doing this gives a \textit{graded} curve)
\end{enumerate} 
The graded Legendrian curves $L_+,L_-,L_s$ are well-defined up to isotopy in a neighborhood of $(x,z)$.
See Figure~\ref{fig_resolve}. 
The reader may wonder whether $L_s$ really lifts to a Legendrian curve.
When reconnecting the two strands in $S$ there could a difference in $z$-value of the endpoints lifts, however this difference must be small since $L$ has a self-intersection not just in $S$ but in $S\times\RR$.
This jump in $z$-value can be fixed by adding a bump in one direction or another to one of the strands.

\begin{figure}[h]
\centering
\begin{tikzpicture}
\draw[thick] (-5,1) to (-3,-1);
\draw[white,line width=1mm] (-5,-1) to (-3,1);
\draw[thick] (-5,-1) to [out=45,in=-135] (-4.7,-.7) to [out=45,in=-150] (-4,0) to [out=30,in=-135] (-3.3,.7) to [out=45,in=-135] (-3,1);

\node [below] at (-4,-1.2) {$L_+$};

\draw[thick] (-1,-1) to [out=45,in=-135] (-.7,-.7) to [out=45,in=-120] (0,0) to [out=60,in=-135] (.7,.7) to [out=45,in=-135] (1,1);
\draw[white,line width=1mm] (-1,1) to (1,-1);
\draw[thick] (-1,1) to (1,-1);

\node [below] at (0,-1.2) {$L_-$};

\draw[thick] (3,-1) to [out=45,in=-135] (3.3,-.7) to [out=45,in=-135] (3.7,-.4) to [out=45,in=-90] (3.9,0) to [out=90,in=-45] (3.7,.4) to [out=135,in=-45] (3.3,.7) to [out=135,in=-45] (3,1);
\draw[thick] (5,-1) to [out=135,in=-45] (4.7,-.7) to [out=135,in=-45] (4.3,-.4) to [out=135,in=-90] (4.1,0) to [out=90,in=-135] (4.3,.4) to [out=45,in=-135] (4.7,.7) to [out=45,in=-135] (5,1);
\node [below] at (4,-1.2) {$L_s$};

\end{tikzpicture}
\caption{Resolving a self-intersection.}
\label{fig_resolve}
\end{figure}
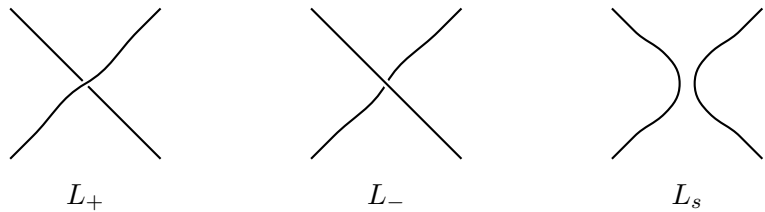

Suppose a local system of vector spaces $E$ is given on $L_-$ or equivalently $L_+$ and let $E_i$ be the fiber of $E$ over $t_i$, $i=1,2$.
For any choice of isomorphism $g:E_0\to E_1$ we also get a local system $E_g$ on $L_s$ by identifying $E_0$ and $E_1$ via $g$ for the left branch of $L_s$ and via $-g$ for the right branch of $L_s$.

\begin{prop}
\label{prop_resolve}
Let $L=(I,\gamma,\tilde{\gamma})\subset S\times\RR$ be a graded Legendrian curve satisfying the usual boundary conditions for $\mc F^{\vee}$ and which is embedded except for a single point $p=\gamma(t_0)=\gamma(t_1)$ of transverse self-intersection, and $E$ a local system of vector spaces on $I$.
Then for an invertible $g\in\Hom(E_0,E_1)$ there is a bijection
\[
\varphi:\mc{MC}(L_s,E_g)\longrightarrow \left\{ \delta\in\mc{MC}(L_-,E)\mid \delta_p=g\right\}
\]
where $\delta_p$ is the component of the Maurer--Cartan element $\delta\in\Hom^1((L_-,E),(L_-,E))_{>0}$ belonging to $p$.
Moreover, $\varphi$ preserves the isomorphism class in $\mc F^{\vee}$, i.e. $(L_s,E_g,\delta)$ is isomorphic to $(L_-,E,\varphi(\delta))$ in $\mc F^{\vee}$.
\end{prop}

\begin{proof}
Immersed disks with boundary on $L,L_s,L_-$ and their perturbed copies can be divided into \textit{small disks} whose area goes to zero as the perturbation is chosen smaller and smaller, and the remaining \textit{big disks} whose area goes to some positive constant.
In the definition of the filtered category we regard the perturbation as negligible, so small disks contribute zeroth order terms and big disks contribute terms of positive order to the structure constants of $\mc F^{\vee}$. 
The idea of the proof is to first show the statement modulo higher order terms, ignoring big disks, then use the general algebraic machinery of transporting Maurer--Cartan elements to lift it to all orders.

Assume, without loss of generality, that there are no components of $L$ disjoint from $p$, i.e. ignore any components which are the same for $L,L_s,L_-$.
There are several cases to consider depending on how the four paths along $L$ starting at $p$ eventually link up.
Consider first the simplest case where both components of $L$ meeting at $p$ are not closed.
In this case we can assume that the local system $E$ on $L_-$ is trivial and $g=\delta_p=1$ is the identity matrix.
We claim that the morphisms $\alpha_1+\alpha_2$ and $\beta_1-\beta_2$ (see Figure~\ref{fig_alphabeta}) are inverses isomorphisms (to zeroth order) between $(L_-,E,\delta)$ and $(L_s,E_g)$.
Indeed, the two quadrilaterals shown in Figure~\ref{fig_m2alphabeta} give
\[
\widetilde{m}_2(\alpha_1+\alpha_2,\beta_1-\beta_2)=\pi_1+\pi_2=1
\]
and the two triangles and two quadrilaterals shown in Figure~\ref{fig_m2betalpha} give
\[
\widetilde{m}_2(\beta_1-\beta_2,\alpha_1+\alpha_2)=\gamma-\gamma+\pi_1+\pi_2=1.
\]
Also, $\alpha_1$, $\alpha_2$, $\beta_1$, and $\beta_2$ are clearly closed morphisms (to zeroth order).

\begin{figure}[h]
\centering
\begin{minipage}{.45\textwidth}
\begin{center}
\begin{tikzpicture}
\draw[thick] (-1.5,-1.5) to (1.5,1.5);
\draw[thick,->] (1.5,-1.5) to (-1.5,1.5);
\draw[thick,red,->] (-1.8,-1.2) to (-.3,.3) to [out=45,in=-45] (-.3,.9) to (-1.2,1.8);
\draw[thick,red] (1.2,-1.8) to (.3,-.9) to [out=135,in=-135] (.3,-.3) to (1.8,1.2);

\draw[fill] (0,0) circle [radius=.05];
\draw[fill,teal] (-.3,.3) circle [radius=.05];
\node[left] at (-.3,.3) {$\scriptstyle{\alpha_1}$};
\draw[fill,teal] (.3,-.3) circle [radius=.05];
\node[right] at (.3,-.3) {$\scriptstyle{\alpha_2}$};
\end{tikzpicture}
\end{center}
\end{minipage}
\begin{minipage}{.45\textwidth}
\begin{center}
\begin{tikzpicture}
\draw[thick] (-1.5,-1.5) to (1.5,1.5);
\draw[thick,->] (1.5,-1.5) to (-1.5,1.5);
\draw[thick,red,->] (-1.2,-1.8) to (-.3,-.9) to [out=45,in=-45] (-.3,-.3) to (-1.8,1.2);
\draw[thick,red] (1.8,-1.2) to (.3,.3) to [out=135,in=-135] (.3,.9) to (1.2,1.8);

\draw[fill] (0,0) circle [radius=.05];
\draw[fill,teal] (-.3,-.3) circle [radius=.05];
\node[left] at (-.3,-.3) {$\scriptstyle{\beta_1}$};
\draw[fill,teal] (.3,.3) circle [radius=.05];
\node[right] at (.3,.3) {$\scriptstyle{\beta_2}$};
\end{tikzpicture}
\end{center}
\end{minipage}
\caption{Morphisms $\alpha_1,\alpha_2$ from $(L_-,E)$ to $(L_s,E_g)$ and $\beta_1,\beta_2$ from $(L_s,E_g)$ to $(L_-,E)$}
\label{fig_alphabeta}
\end{figure}
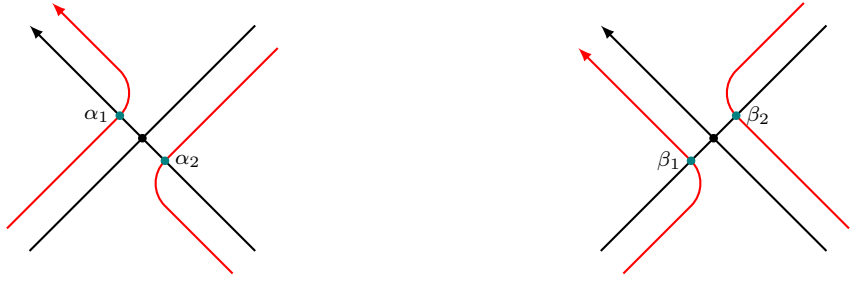

\begin{figure}[h]
\centering
\begin{tikzpicture}
\path[fill,pattern=north west lines,pattern color=gray] (0,0) to (-.3,.3) to (-.6,0) to (-.3,-.3);
\path[fill,pattern=north east lines,pattern color=gray] (0,0) to (.3,.3) to (.6,0) to (.3,-.3);

\draw[thick] (-1.5,-1.5) to (1.5,1.5);
\draw[thick,->] (1.5,-1.5) to (-1.5,1.5);
\draw[thick,red] (-1.2,-1.8) to (-.3,-.9) to [out=45,in=-45] (-.3,-.3) to (-1.8,1.2);
\draw[thick,red] (1.8,-1.2) to (.3,.3) to [out=135,in=-135] (.3,.9) to (1.2,1.8);
\draw[thick,magenta] (-1.8,-1.2) to (-.3,.3) to [out=45,in=-45] (-.3,.9) to (-1.2,1.8);
\draw[thick,magenta] (1.2,-1.8) to (.3,-.9) to [out=135,in=-135] (.3,-.3) to (1.8,1.2);

\draw[fill] (0,0) circle [radius=.05];
\draw[fill,teal] (-.3,.3) circle [radius=.05];
\draw[fill,teal] (-.3,-.3) circle [radius=.05];
\draw[fill,teal] (-.6,0) circle [radius=.05];
\draw[fill,teal] (.3,.3) circle [radius=.05];
\draw[fill,teal] (.3,-.3) circle [radius=.05];
\draw[fill,teal] (.6,0) circle [radius=.05];

\node[left] at (-.6,0) {$\scriptstyle{\pi_1}$};
\node[above] at (-.4,.4) {$\scriptstyle{\alpha_1}$};
\node[below] at (-.4,-.4) {$\scriptstyle{\beta_1}$};
\node[right] at (.6,0) {$\scriptstyle{\pi_2}$};
\node[above] at (.4,.4) {$\scriptstyle{\alpha_2}$};
\node[below] at (.4,-.4) {$\scriptstyle{\beta_2}$};
\end{tikzpicture}
\caption{$\widetilde{m}_2(\alpha_1,\beta_1)=\pi_1$ (left quadrilateral) and $\widetilde{m}_2(\alpha_2,\beta_2)=-\pi_2$ (right quadrilateral)}
\label{fig_m2alphabeta}
\end{figure}
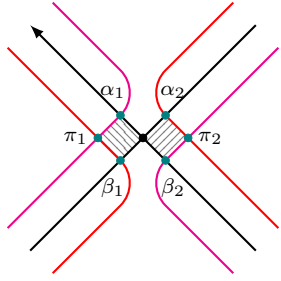

\begin{figure}[h]
\centering
\begin{minipage}{.45\textwidth}
\begin{center}
\begin{tikzpicture}
\path[fill,pattern=north west lines,pattern color=gray] (-.3,.3) to (-.6,.6) to [out=-135,in=90] (-.8,.1) to [out=0,in=-135] (-.3,.3);
\path[fill,pattern=north east lines,pattern color=gray] (-.3,.3) to (.3,.9) to [out=45,in=135] (.9,.9) to (1.2,.6) to [out=-135,in=45] (.6,-.6) to (-.3,.3);

\draw[thick] (-1.5,-1.5) to (1.5,1.5);
\draw[thick] (1.5,-1.5) to (-1.5,1.5);
\draw[thick,red] (-1.8,-1.2) to (-1,-.4) to [out=45,in=-90] (-.8,.1) to [out=90,in=-135] (-.6,.6) to [out=45,in=-45] (-.6,1.2) to (-1.2,1.8);
\draw[thick,red] (1.2,-1.8) to (.6,-1.2) to [out=135,in=-135] (.6,-.6) to [out=45,in=-135] (1.2,.6) to (1.8,1.2);
\draw[thick,gray] (-2.1,-.9) to (-1.3,-.1) to [out=45,in=180] (-.8,.1) to [out=0,in=-135] (-.3,.3) to (.3,.9) to [out=45,in=135] (.9,.9) to (1.2,.6) to [out=-45,in=-135] (1.8,.6) to (2.1,.9);
\draw[thick,gray] (.9,-2.1) to (0,-1.2) to [out=135,in=-135] (.3,-.3) to [out=45,in=-45] (.6,.6) to (-.9,2.1);

\draw[fill] (0,0) circle [radius=.05];
\draw[fill,gray] (.3,.9) circle [radius=.05];
\draw[fill,teal] (-.8,.1) circle [radius=.05];
\node[below right] at (-.9,.2) {$\scriptstyle{\beta_1}$};
\draw[fill,teal] (-.6,.6) circle [radius=.05];
\node[left] at (-.6,.6) {$\scriptstyle{\alpha_1}$};
\draw[fill,teal] (-.3,.3) circle [radius=.05];
\node[below] at (-.3,.3) {$\scriptstyle{\gamma}$};
\draw[fill,teal] (.3,-.3) circle [radius=.05];
\node[below] at (.3,-.3) {$\scriptstyle{\pi_1}$};
\draw[fill,teal] (.6,-.6) circle [radius=.05];
\node[right] at (.6,-.6) {$\scriptstyle{\alpha_2}$};
\draw[fill,teal] (.9,.9) circle [radius=.05];
\node[above] at (.9,.9) {$\scriptstyle{\pi_2}$};
\draw[fill,teal] (1.2,.6) circle [radius=.05];
\node[below] at (1.2,.6) {$\scriptstyle{\beta_2}$};
\end{tikzpicture}
\end{center}
\end{minipage}
\begin{minipage}{.45\textwidth}
\begin{center}
\begin{tikzpicture}
\path[fill,pattern=north east lines,pattern color=gray] (.3,-.3) to [out=45,in=-45] (.6,.6) to (.3,.9) to [out=45,in=135] (.9,.9) to (1.2,.6) to [out=-135,in=45] (.6,-.6) to (.3,-.3);
\path[fill,pattern=north west lines,pattern color=gray] (0,0) to (.9,.9) to (1.2,.6) to [out=-135,in=45] (.6,-.6) to (0,0);

\draw[thick] (-1.5,-1.5) to (1.5,1.5);
\draw[thick,->] (1.5,-1.5) to (-1.5,1.5);
\draw[thick,red] (-1.8,-1.2) to (-1,-.4) to [out=45,in=-90] (-.8,.1) to [out=90,in=-135] (-.6,.6) to [out=45,in=-45] (-.6,1.2) to (-1.2,1.8);
\draw[thick,red] (1.2,-1.8) to (.6,-1.2) to [out=135,in=-135] (.6,-.6) to [out=45,in=-135] (1.2,.6) to (1.8,1.2);
\draw[thick,gray] (-2.1,-.9) to (-1.3,-.1) to [out=45,in=180] (-.8,.1) to [out=0,in=-135] (-.3,.3) to (.3,.9) to [out=45,in=135] (.9,.9) to (1.2,.6) to [out=-45,in=-135] (1.8,.6) to (2.1,.9);
\draw[thick,gray,->] (.9,-2.1) to (0,-1.2) to [out=135,in=-135] (.3,-.3) to [out=45,in=-45] (.6,.6) to (-.9,2.1);

\draw[fill] (0,0) circle [radius=.05];
\draw[fill,gray] (.3,.9) circle [radius=.05];
\draw[fill,teal] (-.8,.1) circle [radius=.05];
\node[below right] at (-.9,.2) {$\scriptstyle{\beta_1}$};
\draw[fill,teal] (-.6,.6) circle [radius=.05];
\node[left] at (-.6,.6) {$\scriptstyle{\alpha_1}$};
\draw[fill,teal] (-.3,.3) circle [radius=.05];
\node[below] at (-.3,.3) {$\scriptstyle{\gamma}$};
\draw[fill,teal] (.3,-.3) circle [radius=.05];
\node[below] at (.3,-.3) {$\scriptstyle{\pi_1}$};
\draw[fill,teal] (.6,-.6) circle [radius=.05];
\node[right] at (.6,-.6) {$\scriptstyle{\alpha_2}$};
\draw[fill,teal] (.9,.9) circle [radius=.05];
\node[above] at (.9,.9) {$\scriptstyle{\pi_2}$};
\draw[fill,teal] (1.2,.6) circle [radius=.05];
\node[below] at (1.2,.6) {$\scriptstyle{\beta_2}$};
\end{tikzpicture}
\end{center}
\end{minipage}
\caption{$\widetilde{m}_2(\beta_1,\alpha_1)=\gamma$ (left triangle in left picture) and $\widetilde{m}_2(\beta_2,\alpha_2)=\gamma-\pi_1-\pi_2$ (right triangle in left picture and two quadrilaterals in right picture)}
\label{fig_m2betalpha}
\end{figure}
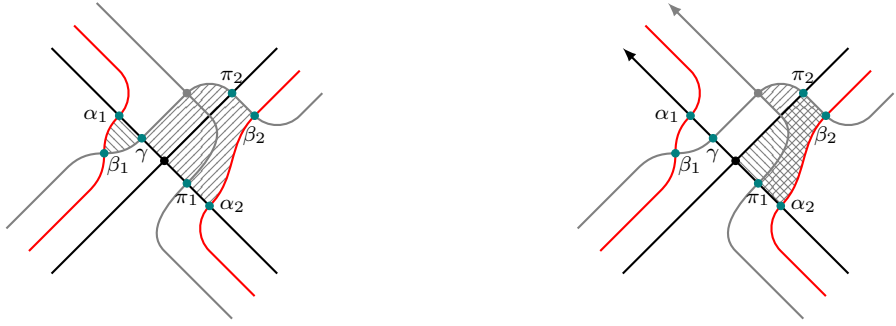

If two of the four paths starting at $p$ eventually meet, then there is one additional crossing of the projections of $L_-$ and $L_s$ to $S$. 
There are six different cases to consider, see the top two rows in Figure~\ref{fig_9cases1}.
If all four of the paths starting at $p$ eventually meet, then there are two additional crossings of the projections of $L_-$ and $L_s$ to $S$. 
There are three different cases to consider, see the bottom row in Figure~\ref{fig_9cases1}.
In all cases trivialize $E$ near $p$ so that $g=1$.
When constructing $E_g$ use the same trivialization and place the additional monodromy sign $-1$ far from $p$ along the path going towards the upper right corner in the figures.
A tedious but straightforward checking of all the nine cases shows that $\alpha_1+\alpha_2$ and $\beta_1-\beta_2$ are closed morphisms (to zeroth order).
In each case there is a pair of a bigon/triangle with vertex $p$ giving contributions to $\widetilde{m}_1$ which are negatives of each other.
The argument that the two morphisms are inverses of each other stays the same, since the additional intersection points give morphisms of degree $1$ as vertices of small disks and thus do not affect the value of $\widetilde{m}_2$ for degree reasons.

\begin{figure}[h]
\begin{center}
\begin{minipage}{.3\textwidth}
\begin{tikzpicture}
\draw[thick] (-1,-1) to (1,1) to [out=45,in=0] (0,2) to [out=180,in=135] (-1,1) to (1,-1);
\draw[thick,red] (-1.15,-.85) to (-.15,.15) to [out=45,in=-45] (-.15,.45) to (-.85,1.15) to [out=135,in=180] (-.15,1.85) to [out=0,in=180] (.15,2.15) to [out=0,in=45] (1.15,.85) to (.15,-.15) to [out=-135,in=135] (.15,-.45) to (.85,-1.15);

\draw[fill] (0,0) circle [radius=.05];
\draw[fill,teal] (-.15,.15) circle [radius=.05];
\node[left] at (-.15,.15) {$\scriptstyle{\alpha_1}$};
\draw[fill,teal] (.15,-.15) circle [radius=.05];
\node[right] at (.15,-.15) {$\scriptstyle{\alpha_2}$};
\draw[fill,magenta] (0,2) circle [radius=.05];
\end{tikzpicture}
\end{minipage}
\begin{minipage}{.3\textwidth}
\begin{tikzpicture}
\draw[thick] (-1,-1) to (1,1) to [out=45,in=90] (2,0) to [out=-90,in=-45] (1,-1) to (-1,1);
\draw[thick,red] (-1.15,-.85) to (-.15,.15) to [out=45,in=-45] (-.15,.45) to (-.85,1.15);
\draw[thick,red] (.85,-1.15) to (.15,-.45) to [out=135,in=-135] (.15,-.15) to (1.15,.85) to [out=45,in=90] (1.85,.15) to [out=-90,in=90] (2.15,-.15) to [out=-90,in=-45] (.85,-1.15);

\draw[fill] (0,0) circle [radius=.05];
\draw[fill,teal] (-.15,.15) circle [radius=.05];
\node[left] at (-.15,.15) {$\scriptstyle{\alpha_1}$};
\draw[fill,teal] (.15,-.15) circle [radius=.05];
\node[right] at (.15,-.15) {$\scriptstyle{\alpha_2}$};\\
\draw[fill,magenta] (2,0) circle [radius=.05];
\end{tikzpicture}
\end{minipage}
\begin{minipage}{.3\textwidth}
\begin{tikzpicture}
\draw[thick] (-.85,-.85) to (1,1) to [out=45,in=45] (1.5,-1.5) to [out=-135,in=-135] (-.85,-.85);
\draw[thick] (1,-1) to (-1,1);
\draw[thick,red] (.85,-1.15) to (.15,-.45) to [out=135,in=-135] (.15,-.15) to (1,.7) to [out=45,in=45] (1.5,-1.3) to [out=-135,in=45] (1.5,-1.7) to [out=-135,in=-135] (-1.15,-.85) to (-.15,.15) to [out=45,in=-45] (-.15,.45) to (-.85,1.15);

\draw[fill] (0,0) circle [radius=.05];
\draw[fill,teal] (-.15,.15) circle [radius=.05];
\node[left] at (-.15,.15) {$\scriptstyle{\alpha_1}$};
\draw[fill,teal] (.15,-.15) circle [radius=.05];
\node[right] at (.15,-.15) {$\scriptstyle{\alpha_2}$};
\draw[fill,magenta] (1.5,-1.5) circle [radius=.05];
\end{tikzpicture}
\end{minipage}

\begin{minipage}{.3\textwidth}
\begin{tikzpicture}
\draw[thick] (-1,1) to (1,-1) to [out=-45,in=0] (0,-2) to [out=180,in=-135] (-1,-1) to (1,1);
\draw[thick,red] (-.85,1.15) to (-.15,.45) to [out=-45,in=45] (-.15,.15) to (-1.15,-.85) to [out=-135,in=180] (-.15,-2.15) to [out=0,in=180] (.15,-1.85) to [out=0,in=-45] (.85,-1.15) to (.15,-.45) to [out=135,in=-135] (.15,-.15) to (1.15,.85);

\draw[fill] (0,0) circle [radius=.05];
\draw[fill,teal] (-.15,.15) circle [radius=.05];
\node[left] at (-.15,.15) {$\scriptstyle{\alpha_1}$};
\draw[fill,teal] (.15,-.15) circle [radius=.05];
\node[right] at (.15,-.15) {$\scriptstyle{\alpha_2}$};
\draw[fill,magenta] (0,-2) circle [radius=.05];
\end{tikzpicture}
\end{minipage}
\begin{minipage}{.3\textwidth}
\begin{tikzpicture}
\draw[thick] (-1,-1) to (1,1);
\draw[thick] (1,-1) to (-1,1) to [out=135,in=90] (-2,0) to [out=-90,in=-135] (-1,-1) to (1,1);;
\draw[thick,red] (-1.15,-.85) to (-.15,.15) to [out=45,in=-45] (-.15,.45) to (-.85,1.15) to [out=135,in=90] (-2.15,.15) to [out=-90,in=90] (-1.85,-.15) to [out=-90,in=-135] (-1.15,-.85);
\draw[thick,red] (.85,-1.15) to (.15,-.45) to [out=135,in=-135] (.15,-.15) to (1.15,.85);

\draw[fill] (0,0) circle [radius=.05];
\draw[fill,teal] (-.15,.15) circle [radius=.05];
\node[left] at (-.15,.15) {$\scriptstyle{\alpha_1}$};
\draw[fill,teal] (.15,-.15) circle [radius=.05];
\node[right] at (.15,-.15) {$\scriptstyle{\alpha_2}$};
\draw[fill,magenta] (-2,0) circle [radius=.05];
\end{tikzpicture}
\end{minipage}
\begin{minipage}{.3\textwidth}
\begin{tikzpicture}
\draw[thick] (-1,-1) to (1,1);
\draw[thick] (.85,-.85) to (-1,1) to [out=135,in=135] (1.5,1.5) to [out=-45,in=-45] (.85,-.85);
\draw[thick,red] (-1.15,-.85) to (-.15,.15) to [out=45,in=-45] (-.15,.45) to (-.7,1) to [out=135,in=135] (1.3,1.5) to [out=-45,in=135] (1.7,1.5) to [out=-45,in=-45] (.85,-1.15) to (.15,-.45) to [out=135,in=-135] (.15,-.15) to (1.15,.85);

\draw[fill] (0,0) circle [radius=.05];
\draw[fill,teal] (-.15,.15) circle [radius=.05];
\node[left] at (-.15,.15) {$\scriptstyle{\alpha_1}$};
\draw[fill,teal] (.15,-.15) circle [radius=.05];
\node[right] at (.15,-.15) {$\scriptstyle{\alpha_2}$};
\draw[fill,magenta] (1.5,1.5) circle [radius=.05];
\end{tikzpicture}
\end{minipage}

\begin{minipage}{.3\textwidth}
\begin{tikzpicture}
\draw[thick] (-1,1) to (1,-1) to [out=-45,in=0] (0,-2) to [out=180,in=-135] (-1,-1) to (1,1) to [out=45,in=0] (0,2) to [out=180,in=135] (-1,1);
\draw[thick,red] (-.85,1.15) to (-.15,.45) to [out=-45,in=45] (-.15,.15) to (-1.15,-.85) to [out=-135,in=180] (-.15,-2.15) to [out=0,in=180] (.15,-1.85) to [out=0,in=-45] (.85,-1.15) to (.15,-.45) to [out=135,in=-135] (.15,-.15) to (1.15,.85) to [out=45,in=0] (.15,2.15) to [out=180,in=0] (-.15,1.85) to [out=180,in=135] (-.85,1.15);

\draw[fill] (0,0) circle [radius=.05];
\draw[fill,teal] (-.15,.15) circle [radius=.05];
\node[left] at (-.15,.15) {$\scriptstyle{\alpha_1}$};
\draw[fill,teal] (.15,-.15) circle [radius=.05];
\node[right] at (.15,-.15) {$\scriptstyle{\alpha_2}$};
\draw[fill,magenta] (0,-2) circle [radius=.05];
\draw[fill,magenta] (0,2) circle [radius=.05];
\end{tikzpicture}
\end{minipage}
\begin{minipage}{.3\textwidth}
\begin{tikzpicture}
\draw[thick] (-1,-1) to (1,1) to [out=45,in=90] (2,0) to [out=-90,in=-45] (1,-1) to (-1,1) to [out=135,in=90] (-2,0) to [out=-90,in=-135] (-1,-1);
\draw[thick,red] (-1.15,-.85) to (-.15,.15) to [out=45,in=-45] (-.15,.45) to (-.85,1.15) to [out=135,in=90] (-2.15,.15) to [out=-90,in=90] (-1.85,-.15) to [out=-90,in=-135] (-1.15,-.85);
\draw[thick,red] (.85,-1.15) to (.15,-.45) to [out=135,in=-135] (.15,-.15) to (1.15,.85) to [out=45,in=90] (1.85,.15) to [out=-90,in=90] (2.15,-.15) to [out=-90,in=-45] (.85,-1.15);

\draw[fill] (0,0) circle [radius=.05];
\draw[fill,teal] (-.15,.15) circle [radius=.05];
\node[left] at (-.15,.15) {$\scriptstyle{\alpha_1}$};
\draw[fill,teal] (.15,-.15) circle [radius=.05];
\node[right] at (.15,-.15) {$\scriptstyle{\alpha_2}$};
\draw[fill,magenta] (-2,0) circle [radius=.05];
\draw[fill,magenta] (2,0) circle [radius=.05];
\end{tikzpicture}
\end{minipage}
\begin{minipage}{.3\textwidth}
\begin{tikzpicture}
\draw[thick] (-.85,-.85) to (1,1) to [out=45,in=45] (1.5,-1.5) to [out=-135,in=-135] (-.85,-.85);
\draw[thick] (.85,-.85) to (-1,1) to [out=135,in=135] (1.5,1.5) to [out=-45,in=-45] (.85,-.85);
\draw[thick,red] (-1.15,-.85) to (-.15,.15) to [out=45,in=-45] (-.15,.45) to (-.7,1) to [out=135,in=135] (1.3,1.5) to [out=-45,in=135] (1.7,1.5) to [out=-45,in=-45] (.85,-1.15) to (.15,-.45) to [out=135,in=-135] (.15,-.15) to (1,.7) to [out=45,in=45] (1.5,-1.3) to [out=-135,in=45] (1.5,-1.7) to [out=-135,in=-135] (-1.15,-.85);

\draw[fill] (0,0) circle [radius=.05];
\draw[fill,teal] (-.15,.15) circle [radius=.05];
\node[left] at (-.15,.15) {$\scriptstyle{\alpha_1}$};
\draw[fill,teal] (.15,-.15) circle [radius=.05];
\node[right] at (.15,-.15) {$\scriptstyle{\alpha_2}$};
\draw[fill,magenta] (1.5,1.5) circle [radius=.05];
\draw[fill,magenta] (1.5,-1.5) circle [radius=.05];
\end{tikzpicture}
\end{minipage}
\end{center}
\caption{Showing that $\alpha_1+\alpha_2$ is closed. There are similar pictures for $\beta_1-\beta_2$.}
\label{fig_9cases1}
\end{figure}
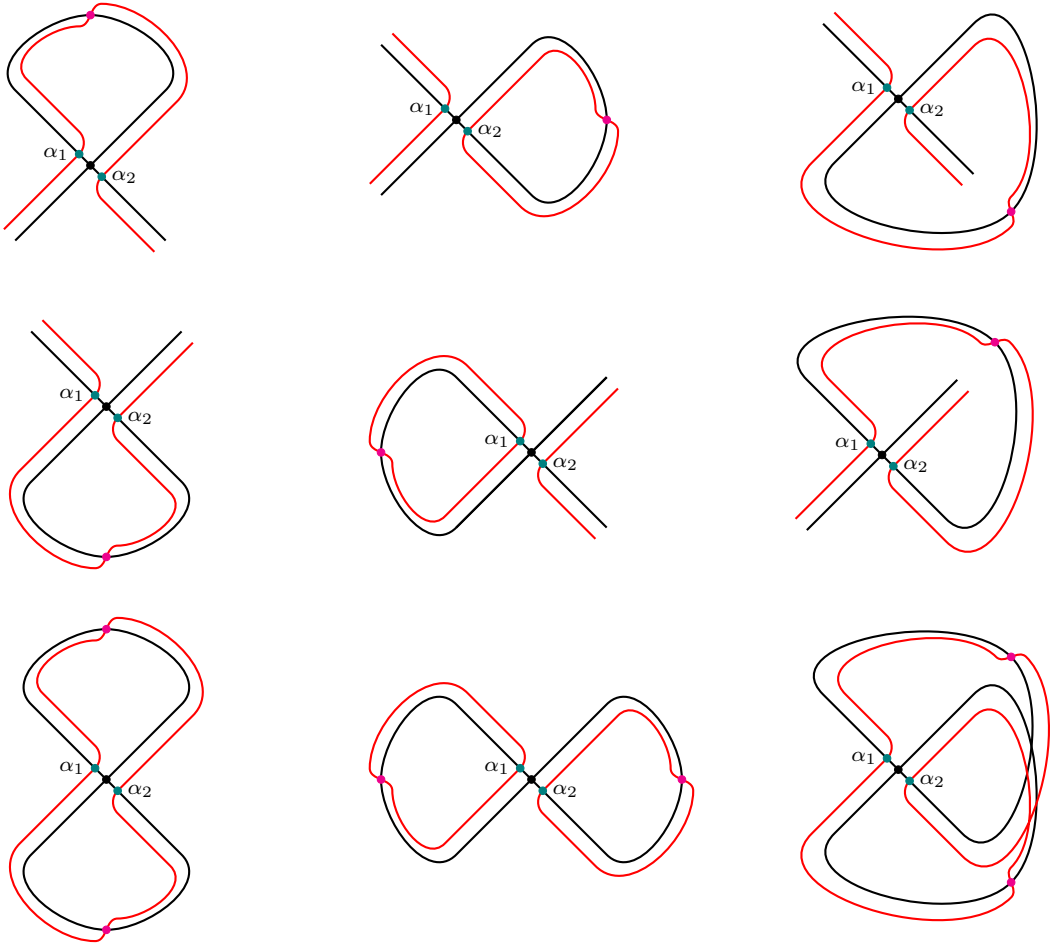

To complete the proof we apply Proposition~\ref{prop_isolift} with $X=(L_-,E,\delta_p)$, $Y=(L_s,E_g)$, and $f_0=\alpha_1+\alpha_2$. 
In general the proposition would not give a unique $\gamma\in\mc{MC}(Y)$, however in our case we have $\Hom(X,X)_{>0}=\Hom(Y,Y)_{>0}$ and the map $\phi$ constructed in the proof of the proposition is an isomorphism, not just a homotopy equivalence.
\end{proof}

\subsubsection{Boundary point}

Suppose $L=(I,\gamma,\tilde{\gamma})$ is a graded Legendrian curve where $\gamma:I\to S\times\RR$ is an embedding except for a transverse self-intersection point at $(x,z):=\gamma(t_0)=\gamma(t_1)$ with $x\in N\subset\partial S$ on the boundary.
We order $t_0,t_1$ so that in the clockwise order the strand of $L$ belonging to $t_0$ comes before the strand belonging to $t_1$, i.e. the $t_0$ strand is the upper one in Figure~\ref{fig_resolve_b}.
Assume furthermore that the grading $\tilde{\gamma}$ is such that
\[
i(L,t_0,L,t_1)=1
\]
then there are three ways of resolving the singularity at $(x,z)$ by modifying $L$ in a neighborhood of that point:
\begin{enumerate}[1)]
\item
$L_+$ obtained by pushing the $t_0$-strand above the $t_1$-strand
\item
$L_-$ obtained by pushing the $t_0$-strand below the $t_1$-strand
\item
$L_s$ obtained by detaching both strands from the boundary and connecting them with each other
\end{enumerate} 
See Figure~\ref{fig_resolve}. 
The graded Legendrian curves $L_+,L_-,L_s$ are well-defined up to isotopy in a neighborhood of $(x,z)$.

\begin{figure}[h]
\centering
\begin{tikzpicture}
\draw[thick] (-4,0) to (-3,-1);
\draw[white,line width=1.5mm] (-5,-1) to (-3,1);
\draw[thick] (-4,0) to [out=30,in=-135] (-3.3,.7) to [out=45,in=-135] (-3,1);
\draw[densely dotted] (-4,-1) to (-4,1);
\node [below] at (-4,-1.2) {$L_+$};

\draw[thick] (0,0) to [out=60,in=-135] (.7,.7) to [out=45,in=-135] (1,1);
\draw[white,line width=1.5mm] (-1,1) to (1,-1);
\draw[thick] (0,0) to (1,-1);
\draw[densely dotted] (0,-1) to (0,1);
\node [below] at (0,-1.2) {$L_-$};

\draw[thick] (5,-1) to [out=135,in=-45] (4.7,-.7) to [out=135,in=-45] (4.3,-.4) to [out=135,in=-90] (4.1,0) to [out=90,in=-135] (4.3,.4) to [out=45,in=-135] (4.7,.7) to [out=45,in=-135] (5,1);
\draw[densely dotted] (4,-1) to (4,1);
\node [below] at (4,-1.2) {$L_s$};
\end{tikzpicture}
\caption{Resolving a self-intersection at the boundary.}
\label{fig_resolve_b}
\end{figure}
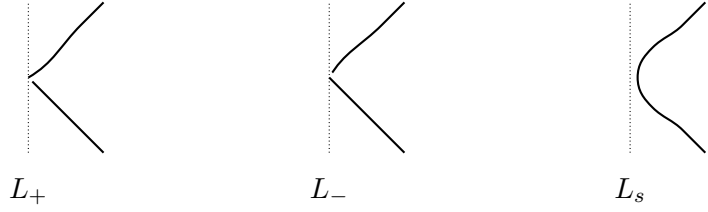

Suppose a local system of vector spaces $E$ is given on $L_-$ or equivalently $L_+$ and let $E_i$ be the fiber of $E$ over $t_i$, $i=1,2$.
For any choice of isomorphism $g:E_0\to E_1$ we also get a local system $E_g$ on $L_s$ by identifying $E_0$ and $E_1$ via $-g$.
In this situation we have the following boundary version of Proposition~\ref{prop_resolve}.
The proof is completely analogous. The corresponding figures are as in the proof of Proposition~\ref{prop_resolve} but with the left half removed.

\begin{prop}
\label{prop_resolve_b}
Let $L=(I,\gamma,\tilde{\gamma})\subset S\times\RR$ be a graded Legendrian curve as above, and $E$ a local system of vector spaces on $I$.
Then for an invertible $g\in\Hom(E_0,E_1)$ there is an isomorphism
\[
\varphi:\mc{MC}(L_s,E_g)\longrightarrow \left\{ \delta\in\mc{MC}(L_-,E)\mid \delta_p=g\right\}
\]
where $\delta_p$ is the component of the Maurer--Cartan element $\delta\in\Hom^1((L_-,E),(L_-,E))_{>0}$ belonging to $p$.
Moreover, $\varphi$ preserves the isomorphism class in $\mc F^{\vee}$, i.e. $(L_s,E_g,\delta)$ is isomorphic to $(L_-,E,\varphi(\delta))$ in $\mc F^{\vee}$.
\end{prop}

\section{Legendrian skein algebras}
\label{sec_skein}

This section contains the heart of the paper.
We introduce the skein algebra in Subsection~\ref{subsec_skeinrel} using the relation given in the introduction.
Some comments about the front projection are also found here.
In Subsection~\ref{subsec_skeinhall} we define the homomorphism $\Phi$ from the skein algebra to the Hall algebra of the Fukaya category.
The main point is to show that the relation \eqref{gs_1} holds, which uses the result of the previous section.
Subsection~\ref{subsec_tangles} discusses Legendrian tangles in preparation for the final two subsections, where we specialize to the case of a disk and an annulus, respectively, in which case we can say more about $\Phi$.

\subsection{Skein relations}
\label{subsec_skeinrel}

Fix $S,N,\theta,\eta$ as before.
Thus, $S$ is a compact surface with boundary, $N\subset \partial S$ a finite set, $\theta$ a Liouville form on $S$ and $\eta$ a grading on $S$.
From this data we get a threefold $S\times\RR$ with contact 1-form $dz+\theta$.
A \textbf{graded Legendrian link} is a graded Legendrian curve $L$ embedded in $S\times\RR$ with $\partial L\subset N\times\RR$.
The skein module $\mathrm{Skein}(S,N,\theta,\eta)$ is the $R:=\ZZ[q^{\pm},(q-1)^{-1}]$-module generated by isotopy classes of graded Legendrian links modulo the submodule generated by the skein relations \eqref{gs_1}, \eqref{gs_2}, and \eqref{gs_3}.
These are linear relations between links which are identical except in a small ball in $S\times\RR$ where they differ as shown.
Near $N\times\RR$ we impose in addition the \textit{boundary skein relations} \eqref{gs_1b} and \eqref{gs_2b}. 

As pointed out in the introduction, the Skein \textit{module} can be defined in the same way for any graded contact threefold, at least in the case $N=\emptyset$ without the boundary skein relations.
Since our threefold has the product form $S\times\RR$, there is an associative unital product on $\mathrm{Skein}(M,N,\theta,\eta)$ defined on links by
\[
L_1L_2:=\text{stack }L_2\text{ on top of }L_1
\]
i.e. by translating $L_2$ in the positive $z$-direction until the maximum value of $z$ on $L_1$ is less than the minimum value of $z$ on $L_2$, then taking the disjoint union. 
The unit is represented by the empty link.

\begin{remark}
While Legendrian skein modules have not previously been explicitly considered in the literature, the defining relations are well-known to experts in Legendrian knot theory and appear (in slightly different but equivalent form) for example in work of Rutherford~\cite{rutherford06}.
Non-Legendrian skein modules, on the other hand, have been studied extensively.
Let us mention here just one particularly intriguing results due to Turaev~\cite{turaev}, that the HOMFLY-PT skein algebra of a product threefold $S\times [0,1]$ is a quantization of the Goldman Lie algebra, which acts on moduli spaces of local systems. 
\end{remark}

We explain briefly how to relate our skein relations to the ones in \cite{rutherford06} which give the graded ruling polynomial as it is usually defined in the Legendrian knot theory literature. 
Given a graded Legendrian link $L$ in $S\times \RR$ without boundary define its \textit{writhe} as the following signed number of crossings of the projection to $S$:
\begin{equation}
w(L):=
\#\left(
\begin{tikzpicture}[baseline=-\dimexpr\fontdimen22\textfont2\relax,scale=.6]
\draw[thick] (-.5,.5) to (.5,-.5);
\draw[line width=1.5mm,white] (-.5,-.5) to (.5,.5);
\draw[thick] (-.5,-.5) to (.5,.5);
\node[right,blue] at (.4,.5) {\scriptsize{n+2k}};
\node[right,blue] at (.4,-.5) {\scriptsize{n}};
\end{tikzpicture}
\right)+\#\left(
\begin{tikzpicture}[baseline=-\dimexpr\fontdimen22\textfont2\relax,scale=.6]
\draw[thick] (-.5,-.5) to (.5,.5);
\draw[line width=1.5mm,white] (-.5,.5) to (.5,-.5);
\draw[thick] (-.5,.5) to (.5,-.5);
\node[right,blue] at (.4,.5) {\scriptsize{n+2k}};
\node[right,blue] at (.4,-.5) {\scriptsize{n+1}};
\end{tikzpicture}
\right)-\#\left(
\begin{tikzpicture}[baseline=-\dimexpr\fontdimen22\textfont2\relax,scale=.6]
\draw[thick] (-.5,.5) to (.5,-.5);
\draw[line width=1.5mm,white] (-.5,-.5) to (.5,.5);
\draw[thick] (-.5,-.5) to (.5,.5);
\node[right,blue] at (.4,.5) {\scriptsize{n+2k}};
\node[right,blue] at (.4,-.5) {\scriptsize{n+1}};
\end{tikzpicture}
\right)-\#\left(
\begin{tikzpicture}[baseline=-\dimexpr\fontdimen22\textfont2\relax,scale=.6]
\draw[thick] (-.5,-.5) to (.5,.5);
\draw[line width=1.5mm,white] (-.5,.5) to (.5,-.5);
\draw[thick] (-.5,.5) to (.5,-.5);
\node[right,blue] at (.4,.5) {\scriptsize{n+2k}};
\node[right,blue] at (.4,-.5) {\scriptsize{n}};
\end{tikzpicture}
\right)
\end{equation}
This is easily seen to be an isotopy invariant by checking Reidemeister moves.
If we replace $L$ by $q^{\frac{1}{2}w(L)}L$ then the skein relation \eqref{gs_1} becomes, after multiplication with $q^{\pm\frac{1}{2}}$,
\begin{equation}\label{gs_1alt}
\begin{tikzpicture}[baseline=-\dimexpr\fontdimen22\textfont2\relax,scale=.6]
\draw[thick] (-1,1) to (1,-1);
\draw[white,line width=2mm] (-1,-1) to (1,1);
\draw[thick] (-1,-1) to (1,1);
\node[blue,rotate=45,below] at (-.7,-.6) {\scriptsize{m}};
\node[blue,rotate=-45,above] at (-.7,.6) {\scriptsize{n}};
\draw[dashed] (0,0) circle [radius=1.414];
\end{tikzpicture}%
-
\begin{tikzpicture}[baseline=-\dimexpr\fontdimen22\textfont2\relax,scale=.6]
\draw[thick] (-1,-1) to (1,1);
\draw[white,line width=2mm] (-1,1) to (1,-1);
\draw[thick] (-1,1) to (1,-1);
\node[blue,rotate=45,below] at (-.7,-.6) {\scriptsize{m}};
\node[blue,rotate=-45,above] at (-.7,.6) {\scriptsize{n}};
\draw[dashed] (0,0) circle [radius=1.414];
\end{tikzpicture}%
=z\left(\delta_{m,n}\;
\begin{tikzpicture}[baseline=-\dimexpr\fontdimen22\textfont2\relax,scale=.6]
\draw[thick] (-1,-1) to (-.62,-.62) to [out=45,in=-45] (-.62,.62) to (-1,1);
\draw[thick] (1,-1) to (.62,-.62) to [out=135,in=-135] (.62,.62) to (1,1);
\node[blue,rotate=45,below] at (-.7,-.6) {\scriptsize{n}};
\node[blue,rotate=45,above] at (.7,.6) {\scriptsize{n}};
\draw[dashed] (0,0) circle [radius=1.414];
\end{tikzpicture}%
-\delta_{m,n+1}\;
\begin{tikzpicture}[baseline=-\dimexpr\fontdimen22\textfont2\relax,scale=.6]
\draw[thick] (-1,-1) to (-.62,-.62) to [out=45,in=135] (.62,-.62) to (1,-1);
\draw[thick] (-1,1) to (-.62,.62) to [out=-45,in=-135] (.62,.62) to (1,1);
\draw (0,.25) to (0,.5);
\draw (0,-.25) to (0,-.5);
\node[blue,rotate=-45,below] at (.7,-.6) {\scriptsize{n}};
\node[blue,rotate=-45,above] at (-.7,.6) {\scriptsize{n}};
\draw[dashed] (0,0) circle [radius=1.414];
\end{tikzpicture}%
\right)
\tag{S1'}
\end{equation}
where $z=q^{\frac{1}{2}}-q^{-\frac{1}{2}}$, and \eqref{gs_2} becomes
\begin{equation}\label{gs_2alt}
\begin{tikzpicture}[baseline=-\dimexpr\fontdimen22\textfont2\relax,scale=.6]
\draw[thick] (.1,.1) to [out=45,in=90] (1,0) to [out=-90,in=-45] (0,0) to [out=135,in=90] (-1,0) to [out=-90,in=-135] (-.1,-.1);
\draw[dashed] (0,0) circle [radius=1.414];
\end{tikzpicture}
=z^{-1}\;
\begin{tikzpicture}[baseline=-\dimexpr\fontdimen22\textfont2\relax,scale=.6]
\draw[dashed] (0,0) circle [radius=1.414];
\end{tikzpicture}
\tag{S2'}
\end{equation}
which are the skein relations for $z^{-1}R(L)$ where $R(L)$ is the graded ruling polynomial as defined e.g. in \cite{rutherford06}.

\subsubsection{Front projection}

So far we have been depicting Legendrian links using the \textbf{Lagrangian projection}, i.e. the projection to the $(x,y)$-plane, from which the Legendrian curve can be recovered by integration.
Generally, the \textbf{front projection}, which is the projection to the $(x,z)$-plane, turns out to be more useful since the Legendrian curve can be recovered simply by looking at the slope: $y=dz/dx$, at least where $dx\neq 0$ on the curve.
For a generic Legendrian curve the front projection has three types of singularities, left cusps ($\prec$), right cusps ($\succ$), and transverse crossings ($\times$), away from which the projection is the graph of a smooth function $z=f(x)$.
In particular, unlike for non-Legendrian knots, it is not necessary to indicate which strand passes over the other, though we will often do so for convenience.

Here is what the skein relations look like under front projection.
\begin{equation}\label{gs_1f}
\begin{tikzpicture}[baseline=-\dimexpr\fontdimen22\textfont2\relax,scale=.6]
\draw[thick] (-1,-1) to [out=60,in=180] (0,-.25) to [out=0,in=120] (1,-1);
\draw[thick] (-1,1) to [out=-60,in=180] (0,.25) to [out=0,in=-120] (1,1);
\node[blue,below] at (0,-.2) {\scriptsize{n+1}};
\node[blue,above] at (0,.2) {\scriptsize{m}};
\draw[dashed] (0,0) circle [radius=1.414];
\end{tikzpicture}%
-q^{(-1)^{m-n}}\;
\begin{tikzpicture}[baseline=-\dimexpr\fontdimen22\textfont2\relax,scale=.6]
\draw[thick] (-1,-1) to [out=60,in=-120] (-.6,-.15);
\draw[thick] (-.35,.15) to [out=45,in=180] (0,.25) to [out=0,in=120] (1,-1);
\draw[thick] (-1,1) to [out=-60,in=180] (0,-.25) to [out=0,in=-135] (.35,-.15);
\draw[thick] (.6,.15) to [out=60,in=-120] (1,1);
\node[blue,below] at (0,-.2) {\scriptsize{m}};
\node[blue,above] at (0,.2) {\scriptsize{n+1}};
\draw[dashed] (0,0) circle [radius=1.414];
\end{tikzpicture}%
=\delta_{m,n}(q-1)\;
\begin{tikzpicture}[baseline=-\dimexpr\fontdimen22\textfont2\relax,scale=.6]
\draw[thick] (-1,-1) to [out=60,in=180] (-.3,0) to [out=180,in=-60] (-1,1);
\draw[thick] (1,-1) to [out=-120,in=0] (.3,0) to [out=0,in=-120] (1,1);
\node[blue,rotate=-25,above] at (.5,.6) {\scriptsize{m}};
\node[blue,rotate=25,above] at (-.5,.6) {\scriptsize{m}};
\draw[dashed] (0,0) circle [radius=1.414];
\end{tikzpicture}%
-\delta_{m,n+1}(1-q^{-1})\;
\begin{tikzpicture}[baseline=-\dimexpr\fontdimen22\textfont2\relax,scale=.6]
\draw[thick] (-1,-1) to [out=60,in=-150] (-.2,-.1);
\draw[thick] (.2,.1) to [out=30,in=-120] (1,1);
\draw[thick] (-1,1) to [out=-60,in=150] (0,0) to [out=-30,in=120] (1,-1);
\node[blue,rotate=-25,below] at (-.5,-.6) {\scriptsize{m}};
\node[blue,rotate=25,above] at (-.5,.6) {\scriptsize{m}};
\draw[dashed] (0,0) circle [radius=1.414];
\end{tikzpicture}%
\tag{S1}
\end{equation}
An alternative way of presenting the same relation is as follows.
\begin{equation}\label{gs_1fa}
\begin{tikzpicture}[baseline=-\dimexpr\fontdimen22\textfont2\relax,scale=.6]
\draw[thick] (-1.414,0) to [out=0,in=180] (0,-.5) to [out=0,in=180] (1.414,0);
\draw[line width=1mm,white] (1,1) to [out=-135,in=0] (0,0) to [out=0,in=135] (1,-1);
\draw[thick] (1,1) to [out=-135,in=0] (0,0) to [out=0,in=135] (1,-1);
\node[blue,above] at (-.8,-.2) {\scriptsize{n}};
\node[blue,left] at (.95,.85) {\scriptsize{m-1}};
\node[blue,left] at (.95,-.85) {\scriptsize{m}};
\draw[dashed] (0,0) circle [radius=1.414];
\end{tikzpicture}%
-q^{(-1)^{m-n}}\;
\begin{tikzpicture}[baseline=-\dimexpr\fontdimen22\textfont2\relax,scale=.6]
\draw[thick] (1,1) to [out=-135,in=0] (0,0) to [out=0,in=135] (1,-1);
\draw[line width=1mm,white] (-1.414,0) to [out=0,in=180] (0,.5) to [out=0,in=180] (1.414,0);
\draw[thick] (-1.414,0) to [out=0,in=180] (0,.5) to [out=0,in=180] (1.414,0);
\node[blue,below] at (-.8,.2) {\scriptsize{n}};
\node[blue,left] at (.95,.85) {\scriptsize{m-1}};
\node[blue,left] at (.95,-.85) {\scriptsize{m}};
\draw[dashed] (0,0) circle [radius=1.414];
\end{tikzpicture}%
=\delta_{m,n}(q-1)\;
\begin{tikzpicture}[baseline=-\dimexpr\fontdimen22\textfont2\relax,scale=.6]
\draw[thick] (-1.414,0) to [out=0,in=135] (1,-1);
\draw[thick] (1,1) to [out=-135,in=0] (0,.3) to [out=0,in=180] (1.414,0);
\node[blue,below] at (-.5,0) {\scriptsize{n}};
\node[blue,left] at (.95,.85) {\scriptsize{m-1}};
\node[blue,left] at (.95,-.05) {\scriptsize{m}};
\draw[dashed] (0,0) circle [radius=1.414];
\end{tikzpicture}%
-\delta_{m,n+1}(1-q^{-1})\;
\begin{tikzpicture}[baseline=-\dimexpr\fontdimen22\textfont2\relax,scale=.6]
\draw[thick] (-1.414,0) to [out=0,in=-135] (1,1);
\draw[thick] (1,-1) to [out=135,in=0] (0,-.3) to [out=0,in=180] (1.414,0);
\node[blue,below] at (-.7,.2) {\scriptsize{n}};
\node[blue,above] at (.8,-.2) {\scriptsize{m-1}};
\node[blue,left] at (.95,-.85) {\scriptsize{m}};
\draw[dashed] (0,0) circle [radius=1.414];
\end{tikzpicture}%
\tag{S1}
\end{equation}
\begin{equation}\label{gs_2f}
\begin{tikzpicture}[baseline=-\dimexpr\fontdimen22\textfont2\relax,scale=.6]
\draw[thick] (-1,0) to [out=0,in=180] (0,.6) to [out=0,in=180] (1,0) to [out=180,in=0] (0,-.6) to [out=180,in=0] (-1,0);
\draw[dashed] (0,0) circle [radius=1.414];
\end{tikzpicture}
=(q-1)^{-1}\;
\begin{tikzpicture}[baseline=-\dimexpr\fontdimen22\textfont2\relax,scale=.6]
\draw[dashed] (0,0) circle [radius=1.414];
\end{tikzpicture}
\tag{S2}
\end{equation}
\begin{equation}\label{gs_3f}
\begin{tikzpicture}[baseline=-\dimexpr\fontdimen22\textfont2\relax,scale=.6]
\draw[thick] (-1.414,0) to [out=0,in=-135] (.5,.5) to [out=-135,in=45] (-.5,-.5) to [out=45,in=180] (1.414,0);
\draw[dashed] (0,0) circle [radius=1.414];
\end{tikzpicture}
=0
\tag{S3}
\end{equation}
Finally, we draw the boundary skein relations under front projection, where the dotted line is a component of $N\times \RR$.
\begin{equation}\label{gs_1fb}
\begin{tikzpicture}[baseline=-\dimexpr\fontdimen22\textfont2\relax,scale=.6]
\draw[thick] (0,-.25) to [out=0,in=120] (1,-1);
\draw[thick] (0,.25) to [out=0,in=-120] (1,1);
\node[blue,left] at (.1,.25) {\scriptsize{m}};
\node[blue,left] at (.1,-.25) {\scriptsize{n+1}};
\draw[densely dotted] (0,-1.414) to (0,1.414);
\draw[dashed] (0,0) circle [radius=1.414];
\end{tikzpicture}%
-q^{s(m-n)}\;
\begin{tikzpicture}[baseline=-\dimexpr\fontdimen22\textfont2\relax,scale=.6]
\draw[thick] (0,.25) to [out=0,in=120] (1,-1);
\draw[thick] (0,-.25) to [out=0,in=-135] (.35,-.15);
\draw[thick] (.6,.15) to [out=60,in=-120] (1,1);
\node[blue,left] at (.1,-.25) {\scriptsize{m}};
\node[blue,left] at (.1,.25) {\scriptsize{n+1}};
\draw[densely dotted] (0,-1.414) to (0,1.414);
\draw[dashed] (0,0) circle [radius=1.414];
\end{tikzpicture}%
=\delta_{m,n}(q-1)\;
\begin{tikzpicture}[baseline=-\dimexpr\fontdimen22\textfont2\relax,scale=.6]
\draw[thick] (1,-1) to [out=-120,in=0] (.3,0) to [out=0,in=-120] (1,1);
\node[blue,rotate=-25,above] at (.5,.6) {\scriptsize{m}};
\draw[densely dotted] (0,-1.414) to (0,1.414);
\draw[dashed] (0,0) circle [radius=1.414];
\end {tikzpicture}%
\tag{S1b}
\end{equation}
\begin{equation}\label{gs_2fb}
\begin{tikzpicture}[baseline=-\dimexpr\fontdimen22\textfont2\relax,scale=.6]
\draw[thick] (0,.6) to [out=0,in=180] (1,0) to [out=180,in=0] (0,-.6);
\draw[densely dotted] (0,-1.414) to (0,1.414);
\draw[dashed] (0,0) circle [radius=1.414];
\end{tikzpicture}
=(q-1)^{-1}\;
\begin{tikzpicture}[baseline=-\dimexpr\fontdimen22\textfont2\relax,scale=.6]
\draw[densely dotted] (0,-1.414) to (0,1.414);
\draw[dashed] (0,0) circle [radius=1.414];
\end{tikzpicture}
\tag{S2b}
\end{equation}
Here, $s(m-n):=(-1)^{m-n}$ if $m>n$ and $s(m-n):=0$ if $m\leq n$ as in the introduction.

\subsection{From Skein to Hall}
\label{subsec_skeinhall}

Fix data $S$, $N$, $\theta$, $\eta$, and $\KK:=\FF_q$ as usual.
The goal of this subsection is to define a homomorphism of $\QQ$-algebras
\[
\Phi:\mathrm{Skein}(S,N,\theta,\eta)\otimes_{\ZZ[t^{\pm},(1-t)^{-1}]}\QQ\longrightarrow\mathrm{Hall}(\mc F(S,N,\theta,\eta,\KK))
\]
where $\ZZ[t^{\pm},(1-t)^{-1}]\to \QQ$ by mapping $t\mapsto q$.

We begin by defining the image of a graded Legendrian link $L\subset S\times \RR$ in the Hall algebra.
Recall that to $L$ we attach the $A_\infty$-category $\mc C(L)_1$ of rank one local systems on $L$ together with Maurer--Cartan element.
This category has finitely many objects and a weighted counting measure 
\[
\mu_{\mc C(L)_1}(X):=\sum_{X\in\mathrm{Iso}(\mc C(L)_1)}|\mathrm{Aut}(X)|^{-1}\prod_{k=1}^{\infty}|\mathrm{Ext}^{-k}(X,X)|^{(-1)^{k+1}}
\]
for which we give a more explicit formula below.
We define $\Phi(L)$ by pushing $\mu_{\mc C(L)_1}$ forward along the functor
\[
F_L:\mc C(L)_1\longrightarrow \mc F^{\vee}(S,N,\theta,\eta,\KK)
\]
to the Fukaya category of $S$, i.e.
\[
\Phi(L):=\left(F_L\right)_*\left(\mu_{\mc C(L)_1}\right).
\]

Before stating the following lemma, we assign an integer $e(L)$ to a generic graded Lagrangian link $L=(I,\gamma,\tilde{\gamma})$.
For each self-crossing $x\in\mathrm{Cr}(L)$ of $p_1(L)$ we have an intersection index $i_x:=i(L,t_0,L,t_1)\in\ZZ$ where $p_1(\gamma(t_0))=p_1(\gamma(t_1))=x$ and $p_2(\gamma(t_0))>p_2(\gamma(t_1))$.
Define $e(L)$ as the number of $x\in\mathrm{Cr}(L)$ with $i_x(L)\leq 0$ and even, minus the number of $x\in\mathrm{Cr}(L)$ with $i_x(L)\leq 0$ and odd.

\begin{lemma}
Let $L$ be a graded Legendrian link, then the formula
\begin{equation}
\label{Phi_formula}
\mu_{\mc C(L)_1}=(q-1)^{-|\pi_0(L)|}q^{-e(L)}\sum_{E}\sum_{\delta\in\mc{MC}(L,E)}[(L,E,\delta)]
\end{equation}
holds, where $E$ ranges over all isomorphism classes of rank one local systems on $L$.
\end{lemma}

\begin{proof}
The weighted counting measure $\mu_{\mc C}$ on an $A_\infty$-category $\mc C$ can be written as $\mu_{\mc C}=A^!(1)$ where $A:\mc C\to 0$ is the functor to the final category with only the zero object.
If we factor the functor $\mc C_1(L)\to 0$ through $G:\mc C(L)_1\to \mc C(L)_{1,0}$, where $\mc C(L)_{1,0}$ has morphisms $\Hom_{\geq 0}/\Hom_{>0}$ as in Subsection~\ref{subsec_curved}, we get
\[
\mu_{\mc C(L)_1}=G^!\left(\mu_{\mc C(L)_{1,0}}\right).
\]
Isomorphism classes of objects in $\mc C(L)_{1,0}$ correspond to isomorphism classes of rank 1 local systems on $L$, so the weighted counting measure has the simple form 
\[
\mu_{\mc C(L)_{1,0}}=(q-1)^{-|\pi_0(L)|}\sum_{E}[(L,E)].
\]
Combining this with the formula for $G^!$ from Proposition~\ref{upper_shriek_formula} and noting that
\[
e(L)=\sum_{i=0}^\infty (-1)^i\dim\Hom^{-i}((L,E),(L,E))_{>0}
\]
essentially by definition, we obtain \eqref{Phi_formula}.
\end{proof}

Our main result is the following.

\begin{theorem}\label{thm_homomorphism}
The assignment $L\mapsto\Phi(L)$ induces a well-defined homomorphism 
\[
\Phi:\mathrm{Skein}(S,N,\theta,\eta)\otimes_{\ZZ[t^{\pm},(1-t)^{-1}]}\QQ\longrightarrow\mathrm{Hall}(\mc F(S,N,\theta,\eta,\KK))
\]
of $\QQ$-algebras.
\end{theorem}

The proof will be completed in the remainder of this subsection.

\subsubsection{Skein relations in the Hall algebra}

\fbox{\eqref{gs_1}.} 
Possibly performing a rotation by a right angle, we may assume that $m\leq n$.
Let us first look at the case where $m=n$.
In the notation of Subsection~\ref{subsec_smoothing}, we have an immersed graded Legendrian curve $L$ with single self-intersection over $p\in S$ and its three resolutions $L_+$, $L_-$, and $L_s$, which are, from left to right, the three links which appear in the relation \eqref{gs_1}.
Given a rank one local system $E$ on $L_+$ let $E_0$ (resp. $E_1$) be the fiber of $E$ over the upper branch (resp. lower branch) crossing over $p$.
Also, for a Maurer--Cartan element $\delta\in\mc{MC}(L_+,E)$ let $\delta_p\in\Hom(E_0,E_1)$ be the component of $\delta$ belonging to $p$.
The main idea is to distinguish the case $\delta_p=0$ and $\delta_p\neq 0$.
We have a corresponding sum
\begin{align*}
\Phi(L_+)=&(q-1)^{-|\pi_0(L_+)|}q^{-e(L_+)}\sum_{E}\sum_{\substack{\delta\in\mc{MC}(L_+,E) \\ \delta_p=0}}[(L_+,E,\delta)] \\
&+(q-1)^{-|\pi_0(L_+)|}q^{-e(L_+)}\sum_{E}\sum_{\substack{\delta\in\mc{MC}(L_+,E) \\ \delta_p\neq 0}}[(L_+,E,\delta)].
\end{align*}
The first summand is $q\Phi(L_-)$, since the set of $\delta\in\mc{MC}(L_+,E)$ with $\delta_p=0$ is naturally identified with $\mc{MC}(L_-,E)$ and $e(L_-)=e(L_+)+1$ due to the extra basis element of $\Hom((L_-,E),(L_-,E))_{>0}$ in degree zero.
We need to show that the second summand is $(q-1)\Phi(L_s)$.

Rank one local systems over $\KK$ are classified by first cohomology with coefficients in $\KK^{\times}$.
There are surjective pullback maps $g_i:H^1(L;\KK^{\times})\to H^1(L_i;\KK^{\times})$ where $i\in\{+,-,s\}$.
The map $g_i$ is $(q-1)^{\epsilon_+}$-to-1, where $\epsilon_i=1$ if the two branches of $L_i$ near the intersection point belong to the same component, and $\epsilon_i=0$ if they belong to distinct components.
Given $E\in H^1(L_+,K^{\times})$ and $\delta\in\mc{MC}(L_+,E)$ with $\delta_p\neq 0$ we can use $\delta_p$ to identify the fibers of $E$ over $p$ to get a local system over $L$, which in turn pulls back to a local system over $L_s$ and a corresponding $\delta\in\mc{MC}(L_s,E)$ by Proposition~\ref{prop_resolve}.
Taking into account the different sizes of the $H^1(L_i;\KK^{\times})$ we get
\begin{align*}
\sum_{E\in H^1(L_+;\KK^{\times})}\sum_{\substack{\delta\in\mc{MC}(L_+,E) \\ \delta_p\neq 0}}[(L_+,E,\delta)]&=(q-1)^{1-\epsilon_+}\sum_{E\in H^1(L;\KK^{\times})}\sum_{\substack{\delta\in\mc{MC}(L_+,E) \\ \delta_p=1}}[(L_+,E,\delta)] \\
&=(q-1)^{1+\epsilon_s-\epsilon_+}\sum_{E\in H^1(L_s;\KK^{\times})}\sum_{\delta\in\mc{MC}(L_s,E)}[(L_s,E,\delta)]
\end{align*}
which proves the claim since $e(L_+)=e(L_s)$ and
\[
\epsilon_s-\epsilon_+=\mathrm{rk} H_1(L_+)-\mathrm{rk} H_1(L_s)=|\pi_0(L_+)|-|\pi_0(L_s)|
\]
as $L_\pm$ and $L_s$ have the same Euler characteristic.

The case $m<n$ is much simpler, since $\mc{MC}(E,L_-)=\mc{MC}(E,L_+)$.
We just need to note that $\Hom((L_-,E),(L_-,E))_{>0}$ has an extra basis element in degree $m-n<0$ compared to $\Hom((L_+,E),(L_+,E))_{>0}$ and thus $e(L_-)=e(L_+)+(-1)^{m-n}$, so $\Phi(L_+)=q^{(-1)^{m-n}}\Phi(L_-)$.

\fbox{\eqref{gs_2}.} 
Let $L$ be a graded Legendrian link with component $L_\infty$ which is an unknot with projection to $S$ as displayed on the left hand side of \eqref{gs_2}.
Suppose we put a local system $E$ with monodromy $\lambda\in\KK^{\times}$ on $L_\infty$ then $\Hom((L_\infty,E),(L_\infty,E))_{>0}$ is of rank 1, generated by a degree 2 element $c$ which is the linear map from the fiber of $E$ over the top branch at the self-intersection to the fiber of $E$ over the bottom branch given by parallel transport along the left 1-gon of $L_\infty$.
With this convention we have
\[
m_0=(1+\lambda)c
\]
and no other terms in the Maurer--Cartan equation for $\Hom((L_\infty,E),(L_\infty,E))_{>0}$, so necessarily $\lambda=-1$ for $\mc{MC}(L)\neq\emptyset$.
It follows that if we denote $L=L_\infty\sqcup L'$ then the map which takes a pair $(E,\delta)$ with $E\in H^1(L;\KK^{\times})$ and $\delta\in\mc{MC}(L,E)$ and restrict it to $L'$ is a bijection, which moreover preserves the isomorphism class in the Fukaya category, since $L_\infty$ is a zero object.
Also $e(L)=e(L')$ and $|\pi_0(L)|=|\pi_0(L')|+1$ so $\Phi(L)=(q-1)^{-1}\Phi(L')$.

\fbox{\eqref{gs_3}.} 
We have a graded Legendrian link $L$ bounding a 1-gon, whose area $\rho$ we assume so be smaller than the areas of any other polygons cut out by the projection of $L$ to $S$.
The lowest order term appearing in the Maurer--Cartan equation for $L$ is $m_0$ coming from this 1-gon, so it cannot cancel with terms coming from other 1-gons or $\delta$, hence $\mc{MC}(L,E)=\emptyset$ for any $E$ and $\Phi(L)=0$.

The proofs of \eqref{gs_1b} and \eqref{gs_2b} are similar to \eqref{gs_1} and \eqref{gs_2}, respectively.

\subsubsection{Compatibility with product}

Suppose $L_1$ and $L_2$ are graded Legendrian links with $L_1$ entirely below $L_2$, i.e. such that the maximum value of $p_2:S\times\RR \to \RR$ on $L_1$ is less than the minimum value of $p_2$ on $L_2$.
We need to show that $\Phi(L_1\cup L_2)=\Phi(L_1)\Phi(L_2)$.
Let $E_i$ be a rank 1 local system on $L_i$, $i=1,2$ and let $E$ be their union on $L:=L_1\cup L_2$.
The main point is that
\begin{align*}
\Hom((L,E),(L,E))_{\geq 0}= & \Hom((L_1,E_1),(L_1,E_1))_{\geq 0}\oplus\Hom((L_2,E_2),(L_1,E_1)) \\
& \oplus \Hom((L_2,E_2),(L_2,E_2))_{\geq 0}
\end{align*}
and $\delta=(\delta_{11},\delta_{21},\delta_{22})\in\Hom^1((L,E),(L,E))_{>0}$ satisfies the Maurer--Cartan equation if and only if $\delta_{11}\in\mc{MC}(L_1,E_1)$, $\delta_{22}\in\mc{MC}(L_2,E_2)$, and $\delta_{21}$ is a closed degree 1 morphism from $(L_2,E_2,\delta_{22})$ to $(L_1,E_1,\delta_{11})$.
Thus
\[
\left(\sum_{\delta_{11}\in\mc{MC}(L_1,E_1)}[(L_1,E_1,\delta_{11})]\right)\left(\sum_{\delta_{22}\in\mc{MC}(L_2,E_2)}[(L_2,E_2,\delta_{22})]\right) = q^{-e(L_2,L_1)}\sum_{\delta\in\mc{MC}(L,E)}[(L,E,\delta)]
\]
where 
\[
e(L_2,L_1)=\sum_{i=0}^\infty (-1)^i\dim\Hom^{-i}((L_2,E_2),(L_1,E_1)).
\]
Since $e(L)=e(L_1)+e(L_2)+e(L_2,L_1)$ and $\pi_0(L)=\pi_0(L_1)\sqcup \pi_0(L_2)$, the claim follows using \eqref{Phi_formula}.

\subsection{Tangles}
\label{subsec_tangles}

Formally, a \textbf{graded Legendrian tangle} is a graded Legendrian link $L$ in $J^1[0,1]\cong [0,1]\times\RR^2$ with boundary in $\{0,1\}\times\RR\times\{0\}$.
We will refer to them simply as \textit{tangles}, and generally use the front projection to depict them.
The front projection is also implicit in the language used below: The positive $x$-axis points to the right, the positive $z$-axis points up.
Denote the left (resp. right) boundary of a tangle $L$ by $\partial_0L\subset\RR$ (resp. $\partial_1L\subset\RR$).
These are subsets of $\RR$ with grading, i.e. a labeling of the points by integers.
Besides vertical composition --- stacking --- of tangles, there is a horizontal composition, well-defined up to isotopy, for tangles $L_0$, $L_1$ with matching ends, i.e $\partial_1L_0$ isotopic to $\partial_0L_1$. 

\begin{remark}
One way of summarizing the situation is to say that there is a braided monoidal category whose objects are finite graded subsets of $\RR$ up to isotopy and morphisms from $X$ to $Y$ are graded Legendrian tangles $L$ up to isotopy with $\partial_0L=X$ and $\partial_1L=Y$. 
Composition in this category is horizontal composition of tangles, while the monoidal product is given by vertical composition.
This monoidal category has an equivalent description in terms of an object of $D^b(\mathbf k)$ with a pair of complete flags in the derived sense.
For more details see~\cite{haiden_flagstangles}.
\end{remark}

All tangles are obtained from the following \textbf{elementary tangles} under vertical and horizontal composition. 
\begin{equation*}
\begin{tikzpicture}[baseline=-\dimexpr\fontdimen22\textfont2\relax,scale=.6]
\draw[thick] (-1,0) to (1,0);
\node[above,blue] at (0,-.1) {\scriptsize{n}};
\node at (0,-1.5) {$1_n$};
\end{tikzpicture}
\qquad\qquad
\begin{tikzpicture}[baseline=-\dimexpr\fontdimen22\textfont2\relax,scale=.6]
\draw[thick] (1,.5) to [out=180,in=0] (-.5,0) to [out=0,in=180] (1,-.5);
\node[right,blue] at (.9,.5) {\scriptsize{n}};
\node[right,blue] at (.9,-.5) {\scriptsize{n+1}};
\node at (.25,-1.5) {$\lambda_n$};
\end{tikzpicture}
\qquad\qquad
\begin{tikzpicture}[baseline=-\dimexpr\fontdimen22\textfont2\relax,scale=.6]
\draw[thick] (-1,.5) to [out=0,in=180] (.5,0) to [out=180,in=0] (-1,-.5);
\node[left,blue] at (-.9,.5) {\scriptsize{n}};
\node[left,blue] at (-.9,-.5) {\scriptsize{n+1}};
\node at (-.25,-1.5) {$\rho_n$};
\end{tikzpicture}
\qquad\qquad
\begin{tikzpicture}[baseline=-\dimexpr\fontdimen22\textfont2\relax,scale=.6]
\draw[thick] (-1,-.5) to [out=0,in=180] (1,.5);
\draw[line width=1mm,white] (-1,.5) to [out=0,in=180] (1,-.5);
\draw[thick] (-1,.5) to [out=0,in=180] (1,-.5);
\node[right,blue] at (.9,.5) {\scriptsize{m}};
\node[right,blue] at (.9,-.5) {\scriptsize{n}};
\node at (0,-1.5) {$\sigma_{m,n}$};
\end{tikzpicture}
\end{equation*}
A \textbf{basic tangle} is a vertical composition of elementary tangles of the form
\begin{equation}
1_{m_1}\otimes \cdots \otimes 1_{m_i}\otimes \epsilon\otimes 1_{n_1}\otimes\cdots\otimes 1_{n_j}
\end{equation}
where $\epsilon$ is a non-identity elementary tangle and $A\otimes B$ denotes the tangle obtained by stacking $B$ on top of $A$.
Any generic tangle is, up to planar isotopy, a horizontal composition of basic tangles.
Isotopies of tangles are concatenations of basic moves: sliding an elementary tangle past another one and the three Reidemeister moves, which correspond to certain algebraic identities for composition of tangles and allow a completely combinatorial approach.

When constructing the \textbf{skein algebra of tangles},  $\mathfrak{T}$, it is useful to \textit{not} impose the boundary skein relations \eqref{gs_1b}, \eqref{gs_2b}, but only \eqref{gs_1f}, \eqref{gs_2f}, \eqref{gs_3f}, so that horizontal composition is still well-defined.
More generally, we can consider a variant of the skein algebra with ``frozen'' boundary marked points. 
Given a subset $N_f\subset N$ of \textit{frozen} marked points, define $\mathrm{Skein}(S,N,N_f,\theta,\eta)$ like $\mathrm{Skein}(S,N,\theta,\eta)$ but without imposing boundary skein relations near $N_f\times\RR$.
An isotopy is still allowed to slide the endpoints of a Legendrian link along $N\times\RR$, but not past each other.
The case of tangles in $J^1[0,1]$ is essentially the case of a disk with two frozen marked points on the boundary.

The main purpose of this section is to describe a relatively smaller class of tangles which generate the skein as a module.
In order to state our result we introduce two special types of tangles.
First, a \textbf{right cusp tangle} is a tangle with a single right cusp and no left cusps, all crossings involving one of the two strands starting at the cusp, which may not cross each other, and no pair of strands crossing twice (see Figure~\ref{fig_cusptangle}).
Up to specifying the grading, a right cusp tangle is determined up to isotopy by the number of strands above and below the right cusp and the number of strands crossed by the bottom and top strands starting at the cusp.

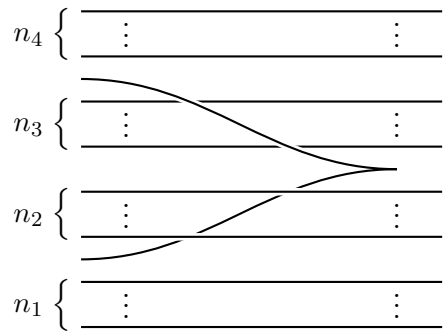
\begin{figure}[ht]
\centering
\begin{tikzpicture}[scale=.6]
\draw[thick] (-4,3.5) to (4,3.5);
\node at (-3,3.15) {\vdots};
\node at (3,3.15) {\vdots};
\draw[thick] (-4,2.5) to (4,2.5);
\draw[thick] (-4,1.5) to (4,1.5);
\node at (-3,1.15) {\vdots};
\node at (3,1.15) {\vdots};
\draw[thick] (-4,.5) to (4,.5);

\draw[white,line width=1mm] (-4,2) to [out=0,in=180] (3,0) to [out=180,in=0] (-4,-2);
\draw[thick] (-4,2) to [out=0,in=180] (3,0) to [out=180,in=0] (-4,-2);

\draw[thick] (-4,-3.5) to (4,-3.5);
\node at (-3,-2.85) {\vdots};
\node at (3,-2.85) {\vdots};
\draw[thick] (-4,-2.5) to (4,-2.5);
\draw[white,line width=1mm] (-4,-1.5) to (4,-1.5);
\draw[thick] (-4,-1.5) to (4,-1.5);
\node at (-3,-.85) {\vdots};
\node at (3,-.85) {\vdots};
\draw[white,line width=1mm] (-4,-.5) to (4,-.5);
\draw[thick] (-4,-.5) to (4,-.5);

\node[left] at (-4,3) {$n_4\;\Big{\{}$};
\node[left] at (-4,1) {$n_3\;\Big{\{}$};
\node[left] at (-4,-1) {$n_2\;\Big{\{}$};
\node[left] at (-4,-3) {$n_1\;\Big{\{}$};

\end{tikzpicture}
\caption{A right cusp tangle.}
\label{fig_cusptangle}
\end{figure}

Second, a \textbf{permutation braid} is a tangle without cusps and with any pair of strands crossing at most once.
The first condition ensures that the tangle defines a permutation $\sigma:\partial_0L\to\partial_1L$, and the ungraded braid is determined up to Legendrian isotopy by $\sigma$.
One way to construct the braid corresponding to $\sigma\in S_n$ is to draw straight lines in the front projection. 
From this it is clear that the number of crossings of the braid is equal to $\mathrm{inv}(\sigma)$, the number of inversions of the permutation.

\begin{prop}\label{prop_tangle_gen}
The skein of tangles is generated as a module by tangles which are a composition, from left to right, of
\begin{enumerate}[1)]
\item
a sequence of right cusp tangles,
\item
a permutation braid,
\item
a sequence of left cusp tangles.
\end{enumerate}
\end{prop}

\begin{proof}
If $L$ is an arbitrary tangle then we need to show that it may be written, as an element of the skein, as a linear combination of tangles of the form described in the statement of the proposition.
We show this by induction on the number $|\partial L|$ of endpoints of $L$.

As a first step we show that $L$ is a linear combination of tangles which have, from left to right, a sequence of right cusp tangles, a braid (tangle without cusps), and a sequence of left cusp tangles.
Let us take care of the left cusps first. 
The strategy, following Rutherford~\cite{rutherford06}, is to focus on the right-most left cusp, which is part of some maximal left cusp tangle $C$.
If there are no more right cusps or crossings to the right of $C$ we can continue by induction with the part of the tangle to the left of $C$, which has less boundary points.
Otherwise, look at the basic tangle immediately to the right of $C$.
A case-by-case analysis shows that using Reidemeister and skein moves we can always reduce to tangles with fewer crossings to the right of the right-most left cusp.
We refer to \cite{rutherford06} for details.
Apply the same strategy to right cusps.

The second step is to reduce from braids to permutation braids.
Use induction on the number of crossings of the braid, so we need to analyze the following situation.
Suppose $P$ is a permutation braid and we compose it on the left with a basic tangle $B$ with a single crossing of the $i$-th and $(i+1)$-st strands.
There are two cases depending on $B$. 
If the strands of $P$ starting at the $i$-th and and $(i+1)$-st boundary point on the left of $P$ do not cross in $P$, then the composition $BP$ is a permutation braid.
Otherwise, if the strands do cross in $P$, we can apply third Reidemeister moves to $P$ to move this crossing to the left, so $P=BP'$ and we want to simplify $BBP'$.
We apply the skein relation \eqref{gs_1} to $BB$ which allows us to write $BBP'$ as a linear combination of $BP'$, $P'$, and a tangle $CC'P'$ with a right cusp in $C$ and a left cusp in $C'$, where some terms may not be present depending on the grading.
The tangles $BP'=P$ and $P'$ are permutation braids and induction takes care of $C'P'$ which has two boundary points less.
\end{proof}

\subsection{Disk}
\label{subsec_disk}

In this subsection we consider the special case where $S$ is a disk with $n+1=|N|$ marked points on the boundary.
For concreteness we take $S$ to be the closed unit disk in $\CC$ and $N=\{p_0,\ldots,p_n\}$ the set of $(n+1)$-st roots of unity, where $p_k:=\exp(2\pi ik/(n+1))$.
Since $S$ is simply connected, all grading structures are equivalent.
For brevity write $\mathrm{Skein}:=\mathrm{Skein}(S,N,\theta,\eta)$.

The Fukaya category $\mc{F}=\mc{F}(S,N,\theta,\eta,\KK)\cong \mc{F}^{\vee}(S,N,\theta,\eta,\KK)$ is equivalent to the bounded derived category $D^b(\mathrm{Rep}(A_n))$ of representations of any $A_n$ quiver, see for example \cite{hkk}.
We will not use this equivalence directly here, but for computing $\mathrm{Hall}(\mc F)$ we use two facts about $\mc F$.
The first is the classification of indecomposable objects and the second is a particular slicing.
Isomorphism classes of indecomposable objects in $\mc F$ are indexed by elements of the set
\[
I:=\{(i,j)\in\ZZ^2\mid 0<j-i\leq n\}
\] 
where the object $E_{i,j}\in\Ob(\mc F)$ corresponding to $(i,j)\in I$ has an underlying curve whose projection to $S$ is an arc (say, straight line) connecting $p_i$ and $p_j$ with grading such that
\begin{enumerate}[1)]
\item
$E_{i,j}[1]=E_{j,i+n+1}$,
\item
$\phi\in [0,1)$ if and only if $-(n+1)/2\leq i+j<(n+1)/2$.
\end{enumerate}
Here $\phi\in [0,1)$ corresponds to the grading where the arc is labeled by ``0'' with our usual conventions.
The precise Legendrian lift of the straight line in $S$ to $S\times \RR$ will be irrelevant.
Also, all $\KK$-linear rank one local systems on $E_{i,j}$ are of course isomorphic.

The category $\mc F$ admits a slicing $\left(\mc F_\phi\right)_{\phi\in\RR}$ where $\mc F_\phi$ has indecomposable objects $E_{i,j}$ with $\frac{i+j}{n+1}=\phi$.
Each $\mc F_\phi$ is semisimple category with simple objects represented by parallel disjoint arcs.
This slicing comes, up to reparametrization of $\RR$, from the stability condition for the quadratic differential $\exp(z^{n+1})dz^2$, see \cite{hkk}.
The stability condition is also characterized, up to a $\CC^\times$ factor, by the fact that it is preserved by $\Aut(\mc F)$ up to the action of $\CC^\times$, i.e. lies at the \textit{orbifold point} in the space of stability conditions.

\begin{theorem}\label{thm_phi_disk}
For the disk with $n+1$ marked points on the boundary, $\Phi$ is an isomorphism.
Thus, the graded Legendrian skein algebra of the disk with $n+1$ marked points on the boundary, specialized at a prime power $q$, is isomorphic to the Hall algebra of the bounded derived category of representations of the $A_n$ quiver over $\FF_q$.
\end{theorem}

\begin{lemma}
$\mathrm{Skein}$ is generated, as an algebra, by the straight the line segments $E_{i,j}$, $(i,j)\in I$.
\end{lemma}

\begin{proof}
When using the front projection it will be convenient to adopt a half-plane model where $S$ is the closed right half-plane $\{\mathrm{Re}(z)\geq 0\}\subset\CC$ and $p_k=-\sqrt{-1}k$, $k=0,\ldots,n$.
Thus, under front projection, a strand ends at $p_k$ if the slope at the boundary is $-p_k$.
In depictions of Legendrian links one can simply draw the strands as horizontal near the boundary (omitting a tiny ``bend'' at the end) and label them by $k$ if they end at $p_k$.
From this discussion it is also clear that, after breaking the cyclic symmetry of the disk with marked points, there is a map from the skein of tangles $L$ with $\partial_1L=\emptyset$ and endpoints in $\partial_0L$ labeled by $\{0,\ldots,n\}$ to the skein of the disk.

We will need to modify links by moving endpoints past each other in the front projection. 
If they are labeled by the same integer, that is end at the same $p_k$, then the boundary skein relation \eqref{gs_1b} holds.
If $i>j$ then
\begin{equation}\label{disk_b2}
\begin{tikzpicture}[baseline=-\dimexpr\fontdimen22\textfont2\relax,scale=.8]
\draw[thick] (0,-.25) to [out=0,in=180] (1,.25);
\draw[white,line width=1mm] (0,.25) to [out=0,in=180] (1,-.25);
\draw[thick] (0,.25) to [out=0,in=180] (1,-.25);
\node[left] at (.1,.25) {$i$};
\node[left] at (.1,-.25) {$j$};
\draw[densely dotted] (0,-1) to (0,1);
\end{tikzpicture}%
=
\begin{tikzpicture}[baseline=-\dimexpr\fontdimen22\textfont2\relax,scale=.8]
\draw[thick] (0,-.25) to [out=0,in=180] (1,-.25);
\draw[thick] (0,.25) to [out=0,in=180] (1,.25);
\node[left] at (.1,.25) {$j$};
\node[left] at (.1,-.25) {$i$};
\draw[densely dotted] (0,-1) to (0,1);
\end{tikzpicture}%
\end{equation}
which is just Legendrian isotopy and does not use any skein relations, and if $i<j$ then
\begin{equation}\label{disk_b1}
\begin{tikzpicture}[baseline=-\dimexpr\fontdimen22\textfont2\relax,scale=.8]
\draw[thick] (0,-.25) to [out=0,in=180] (1,.25);
\draw[white,line width=1mm] (0,.25) to [out=0,in=180] (1,-.25);
\draw[thick] (0,.25) to [out=0,in=180] (1,-.25);
\node[left] at (.1,.25) {$i$};
\node[left] at (.1,-.25) {$j$};
\draw[densely dotted] (0,-1) to (0,1);
\end{tikzpicture}%
=
\begin{tikzpicture}[baseline=-\dimexpr\fontdimen22\textfont2\relax,scale=.8]
\draw[thick] (0,-.25) to [out=0,in=180] (1,.25);
\draw[white,line width=1mm] (0,.25) to [out=0,in=180] (1,-.25);
\draw[thick] (0,.25) to [out=0,in=180] (1,-.25);
\draw[thick] (1,-.25) to [out=0,in=180] (2,.25);
\draw[white,line width=1mm] (1,.25) to [out=0,in=180] (2,-.25);
\draw[thick] (1,.25) to [out=0,in=180] (2,-.25);
\node[left] at (.1,.25) {$j$};
\node[left] at (.1,-.25) {$i$};
\node[blue,above] at (1,.2) {\scriptsize{n+1}};
\node[blue,below] at (1,-.2) {\scriptsize{m}};
\draw[densely dotted] (0,-1) to (0,1);
\end{tikzpicture}%
=q^{-(-1)^{m-n}}\left((1+\delta_{m,n+1}(1-q^{-1}))
\begin{tikzpicture}[baseline=-\dimexpr\fontdimen22\textfont2\relax,scale=.8]
\draw[thick] (0,-.25) to [out=0,in=180] (1,-.25);
\draw[thick] (0,.25) to [out=0,in=180] (1,.25);
\node[left] at (.1,.25) {$j$};
\node[left] at (.1,-.25) {$i$};
\draw[densely dotted] (0,-1) to (0,1);
\end{tikzpicture}%
-\delta_{m,n}(q-1)
\begin{tikzpicture}[baseline=-\dimexpr\fontdimen22\textfont2\relax,scale=.8]
\draw[thick] (0,-.25) to [out=0,in=180] (.5,0) to [out=180,in=0] (0,.25);
\draw[thick] (1.5,.25) to [out=180,in=0] (1,0) to [out=0,in=180] (1.5,-.25);
\node[left] at (.1,.25) {$j$};
\node[left] at (.1,-.25) {$i$};
\draw[densely dotted] (0,-1) to (0,1);
\end{tikzpicture}%
\right)
\end{equation}
by \eqref{disk_b1} and \eqref{gs_1}.
(Warning: The integer labels $i,j$ here refer to the endpoints $p_i,p_j$ and not the grading as in \eqref{gs_1b}.)

Let $A \subseteq \mathrm{Skein}$ be the subalgebra generated by straight line segments and let $B_k\subset\mathrm{Skein}$ be the submodule generated by links with $\leq k$ endpoints.
We will prove by induction that $B_k\subset A$ for all $k$, which implies the statement of the lemma since $\bigcup_k B_k=\mathrm{Skein}$.

By Proposition~\ref{prop_tangle_gen} the submodule $B_k$ is generated by links $L$ with $|\partial L|\leq k$ and no left cusps under front projection.
Suppose $L$ is such a link and decompose it into basic tangles.
If left-most basic tangle of $L$ has a crossing, we can use \eqref{disk_b1}, \eqref{disk_b2}, or the boundary skein relation \eqref{gs_1b} to write $L$ in terms of links with fewer crossings and possibly less endpoints.
Thus, by induction, it remains to deal with the case where the left-most basic tangle of $L$ has a right cusp.
Let $p_i$ (resp. $p_j$) be the endpoint of the lower (resp. upper) strand starting at that cusp.
If $i=j$ we can remove the left cusp using the boundary skein relation \eqref{gs_2b} to get a link with fewer endpoints.
If there are no strands below the cusp, then $L$ is a product of the link with just the cusp, which is isotopic to a straight line segment, and a link with fewer endpoints.
If there is a strand below the cusp ending at $p_k$ then there are the following cases to consider.
\begin{itemize}
\item
If $k<i$ or $k>j$ then we can move the cusp past the strand below it by an isotopy.
\item
If $k=i$ or $k=j$ then we can move the cusp past the strand below it possibly modulo links with less endpoints by the boundary skein relation \eqref{gs_1b}.
\item
If $i<k<j$ then we can move the cusp past the strand below it modulo links with a left cusps which connects $p_r,p_s$ with $|r-s|<|i-j|$.
\end{itemize}
This is easier to see in the Lagrangian projection.
Thus by induction we eventually reduce to one of the base cases above.
\end{proof}

\begin{lemma} \label{lem_disk_span}
$\mathrm{Skein}$ is spanned, as a module, by products of the form
\[
E_{i_1,j_1}\cdots E_{i_m,j_m}
\]
where $i_1+j_1\geq i_2+j_1\geq \ldots \geq i_m+j_m$.
\end{lemma}

\begin{proof}
This follows from the previous lemma and its proof.
It is convenient to return to the disk model and the Lagrangian projection where the $E_{i,j}$ are straight lines.
The boundary skein relation \eqref{gs_1b} ensures that we can change the order of two straight line segments modulo terms where the total lengths of the segments is strictly smaller.
\end{proof}

\begin{proof}[Proof of Theorem~\ref{thm_phi_disk}]
We will define a linear map
\[
\Psi:\mathrm{Hall}(\mc F)\longrightarrow\mathrm{Skein}\otimes_{\ZZ[t^{\pm},(1-t)^{-1}]}\QQ
\]
and show that it is inverse to $\Phi$.
Let $E\in\Ob(\mc F)$, then it has a decomposition
\[
E=E_{i_1,j_1}\oplus \ldots \oplus E_{i_m,j_m}
\]
where $i_1+j_1\geq i_2+j_1\geq \ldots \geq i_m+j_m$.
The summands in this decomposition and their order are unique up to permutation of those $E_{i_k,j_k}$ with the same $i_k+j_k$.
Let
\[
\Psi([E]):=(q-1)^m E_{i_1,j_1}\cdots E_{i_m,j_m}
\]
which is well-defined since $E_{i,j}$ commutes with $E_{i',j'}$ in the skein algebra if $i+j=i'+j'$.

We first check $\Phi\Psi=\mathrm{id}$.
Let
\[
L:=E_{i_1,j_1}\cdots E_{i_m,j_m}
\]
where $i_1+j_1\geq i_2+j_1\geq \ldots \geq i_m+j_m$.
The way in which we have ordered the factors ensures that $e(L)=0$ an $\mc{MC}(L)=\{0\}$, so $\Phi(L)=(q-1)^{-m}L$ by \eqref{Phi_formula}.
The identity $\Psi\Phi=\mathrm{id}$ now follows from surjectivity of $\Psi$, which is Lemma~\ref{lem_disk_span}.
\end{proof}

\subsection{Annulus}
\label{subsec_annulus}

In this subsection we look at the case $S=\RR/\ZZ\times [-1,1]$, $N=\emptyset$, $\theta=-ydx$, and $\eta=\partial/\partial x$.
Let $\mc{F}=\mc{F}(S,N,\theta,\eta,\KK)$ be the wrapped Fukaya category, and $\mc{F}^{\vee}=\mc{F}^{\vee}(S,N,\theta,\eta,\KK)$ the subcategory of objects coming from compact curves.
Determining the structure of these categories is one of the first exercises in Homological Mirror Symmetry (in the non-compact setting), see \cite{aaeko} for the more involved case of spheres with $n\geq 3$ punctures.
The result is that
\[
\mc F\cong D^b(\mathrm{Mod}_{\mathrm{fg}}(\KK[x^{\pm}])),\qquad
\mc F^{\vee}\cong D^b(\mathrm{Mod}_{\mathrm{fd}}(\KK[x^{\pm}]))
\]
where $\mathrm{Mod}_{\mathrm{fg}}(\KK[x^{\pm}])$ (resp. $\mathrm{Mod}_{\mathrm{fd}}(\KK[x^{\pm}])$) is the category of finitely generated (resp. finite-dimensional) modules over the ring of Laurent polynomials.
Under this equivalence the generator of $\mc F$ corresponding to the free module of rank one is the curve $\{0\}\times [-1,1]$ and the simple modules, which are contained in both $\mc F$ and $\mc F^{\vee}$, correspond to the closed curve $\RR/\ZZ\times \{0\}$ with any rank one local system.

We wish to show the following.

\begin{theorem}
\label{thm_annulus}
The homomorphism $\Phi$ from the skein algebra of the standard annulus as above to the Hall algebra of the Fukaya category $\mc F$ is injective.
\end{theorem}

The image of $\Phi$ can be described more explicitly and we will do so below after making appropriate definitions.
We will prove Theorem~\ref{thm_annulus} by finding an explicit basis of the skein algebra.
A basis of the ordinary (non-Legendrian) skein algebra was found by Turaev~\cite{turaev88,turaev}.
The Legendrian case is a bit more subtle, for example the algebra is noncommutative.

Generators of the skein algebra are closed curves which wind around the annulus some number of times.
To introduce some notation, let $C_k$ be the Legendrian curve as in Figure~\ref{fig_cyclic} which winds around the annulus $k>0$ times and with grading function $\phi$ with range in $(-1/2,1/2)$.
To see that the curve depicted in Figure~\ref{fig_cyclic} really is the Lagrangian projection of a Legendrian curve, at least up to planar isotopy, we need to show feasibility of a system of inequalities \cite{chekanov}. 
Let $A_i$, $i=1,\ldots,k-1$ be the area of the $i$-th region (from the top) cut out by the projection of $C_k$ and $h_i>0$, $i=1,\ldots,k-1$ the difference in $z$-coordinates of the strands over the self-crossings points, then by the Legendrian condition
\begin{equation}\label{cycle_leg_cond}
2h_i-h_{i-1}-h_{i+1}=A_i>0
\end{equation}
where $h_0=h_k=0$.
The system of inequalities \eqref{cycle_leg_cond} means geometrically that the piecewise linear function connecting the dots $(i,h_i)$, $i=0,\ldots,k$, is concave, and thus has many solutions, all of which satisfy $h_i>0$ for $i=1,\ldots,k-1$, so for the projection of a \textit{Legendrian} curve the crossings must be as shown in Figure~\ref{fig_cyclic}.

\begin{figure}[ht]
\centering
\begin{tikzpicture}[scale=.6]
\draw[dotted] (-4,-3) to (-4,3);
\draw[dotted] (4,-3) to (4,3);

\draw[thick] (-4,-2) to [out=0,in=180] (4,-1.5);
\draw[thick] (-4,-1.5) to [out=0,in=180] (4,-1);
\draw[thick] (-4,1) to [out=0,in=180] (4,1.5);
\draw[thick] (-4,1.5) to [out=0,in=180] (4,2);

\draw[white,line width=1mm] (-4,2) to [out=0,in=180] (4,-2);
\draw[thick] (-4,2) to [out=0,in=180] (4,-2);

\node at (-2,-.25) {\vdots};
\node at (2,.25) {\vdots};

\end{tikzpicture}
\caption{The Legendrian curve $C_k$ winding around the annulus $k$ times.}
\label{fig_cyclic}
\end{figure}
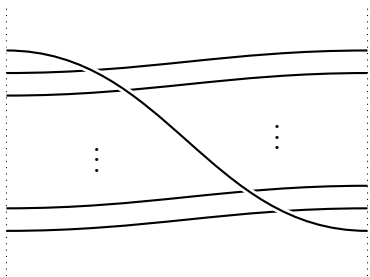

\begin{prop}\label{prop_annulus_skein_basis}
A basis of the graded Legendrian skein algebra of the annulus is given by links of the form
\begin{equation}\label{annulus_skein_basis}
C_{k_1}[n_1]\cdot C_{k_1}[n_2]\cdots C_{k_m}[n_m]
\end{equation}
where $m\geq 0$, $n_i\in\ZZ$, $n_1\geq n_2\geq \ldots \geq n_m$, $k_i>0$, and $k_i\geq k_{i+1}$ if $n_i=n_{i+1}$.
\end{prop}

The proof of this proposition has two parts. 
The first is to show that the elements \eqref{annulus_skein_basis} span the skein, which uses Proposition~\ref{prop_tangle_gen} for tangles, and the second is show linear independence, for which we will use the homomorphism $\Phi$ to the Hall algebra.

\subsubsection{Spanning the skein}

For the links in \eqref{annulus_skein_basis} the grading increases from top to bottom when viewed under front projection.
To set up some terminology, call a subset $E\subset\RR$ with grading $\deg:E\to\ZZ$ \textbf{sorted} if $\deg$ is an increasing function, and more generally define the \textbf{disorder} of a graded subset as
\[
\delta(E,\deg):=\left|\left\{x,y\in E\mid x<y, \deg(x)>\deg(y)\right\}\right|
\]
so $\delta(E,\deg)=0$ if and only if $(E,\deg)$ is sorted.

\begin{lemma}\label{lem_annulus_gen1}
Closures of tangles $L$ with $\partial_0L=\partial_1L$ sorted span the skein of the annulus.
\end{lemma}

\begin{proof}
Given an arbitrary tangle $L$ we need to show that its closure is a linear combination of closures of tangles with sorted boundary.
We will prove this by induction on the disorder $\delta=\delta(\partial_0L)$.
In the base case $\delta=0$ there is nothing to show.
Assume $\delta>0$, then the there is a pair of neighboring points $x_1<x_2$ in $\partial_0L$ with $m:=\deg(x_1)>\deg(x_2)=:n$.
Let $1_b$ (resp. $1_t$) be the tangle with horizontal strands corresponding to points in $\partial_0L$ which are below $x_1$ (resp. above $x_2$).
In the skein of tangles we have
\[
L = q^{-(-1)^{m-n}}L(1_b\otimes (\sigma_{m,n}\sigma_{n,m})\otimes 1_t) + \delta_{m,n+1}(q-1)L(1_b\otimes (\rho_n\lambda_n)\otimes 1_t)
\]
by \eqref{gs_1}.
The closure of the first tangle is equal to the closure of $(1_b\otimes \sigma_{n,m}\otimes 1_t)L(1_b\otimes \sigma_{m,n}\otimes 1_t)$ and the closure of the second tangle is equal to the closure of $(1_b\otimes \lambda_n\otimes 1_t)L(1_b\otimes \rho_n\otimes 1_t)$, both of which have boundary with strictly smaller disorder.
\end{proof}

\begin{lemma}\label{lem_annulus_gen2}
Closures of tangles $L$ of the form
\[
L=B_1\otimes \cdots \otimes B_m,
\]
where each $B_i$ is a permutation braid with all strands of the same degree $d_i\in\ZZ$ and $d_i\leq d_{i+1}$, generate the skein of the annulus.
\end{lemma}

\begin{proof}
Let $L$ be a tangle of the form described in Proposition~\ref{prop_tangle_gen}, with a sequence of right cusp tangles, a permutation braid, and a sequence of left cusp tangles.
Suppose $L$ has a right cusp, then looking at the endpoints of the two strands starting at that cusp we see that $\partial_0L$ is not sorted, similarly for left cusps and $\partial_1L$. 
Thus, if $L$ has sorted boundary, then it is just a permutation braid.
Moreover, by definition of a permutation braid any pair of strands crosses at most once, so if $\partial_0L=\partial_1L$ then no pair of strands with different grading can cross, thus $L$ is a vertical composition of permutation braids each which is made up of strands of the same degree, and so that degree increases when going upwards in the front projection.
Proposition~\ref{prop_tangle_gen} and Lemma~\ref{lem_annulus_gen1} thus imply the claim.
\end{proof}

It remains to show that the closure of a permutation braid with strands of the same degree is a linear combination of braids as in \eqref{annulus_skein_basis} with $n_i=0$.
A somewhat indirect argument which utilizes the existing literature goes as follows.
Fix $n$ and consider those braids in the skein of tangles which are have $n$ strands, all of which are in degree zero.
Among these the defining relations of the Iwahori--Hecke algebra $H_n$ hold: The quadratic relation \eqref{gs_1} and the braid relation aka third Reidemeister move.
We are interested in the quotient $H_n/[H_n,H_n]$ of $H_n$ by the submodule generated by commutators, which is, by taking the braid closure, a summand of the skein of the annulus.
In this description, the desired basis of $H_n/[H_n,H_n]$ in terms of partitions and the braids $C_k$ is found by Bigelow~\cite{bigelow06}.
On the other hand, $H_n/[H_n,H_n]$ maps to the graded \textit{Legendrian} skein of annulus and the image includes closures of permutations braids by definition, so we can transfer the result to the Legendrian setting.
This concludes the proof that the elements \eqref{annulus_skein_basis} span the skein of the annulus.

\subsubsection{Mapping to the Hall algebra}

\begin{lemma}\label{lem_cyclic_braid_comp}
The cyclic braid $C_k$ in the skein algebra of the annulus maps to the element 
\[
\Phi(C_k)=(q-1)^{-1}\sum_{\substack{a_1,\ldots,a_{k-1}\in\KK \\ a_0\in\KK^{\times}}}[(\KK^k,\mathrm{CMAT}(a_0,a_1,\ldots,a_{k-1}))]
\]
in the Hall algebra of $\mc F^{\vee}\cong D^b(\mathrm{Mod}_{\mathrm{fd}}(\KK[x^{\pm}]))$, where
\[
\mathrm{CMAT}(a_0,a_1,\ldots,a_{k-1}):=\begin{bmatrix} 
0 & 0 & & \cdots & & 0 & -a_0 \\ 
1 & 0 & & \cdots & & 0 & -a_1 \\ 
0 & 1 & & \cdots & & 0 & -a_1 \\ 
\vdots & & & \ddots & & & \vdots \\
0 & 0 & & \cdots & & 1 & -a_{k-1} 
\end{bmatrix}
\]
is the \textit{companion matrix}.
(The companion matrix is characterized as $k\times k$-matrix, up to similarity, by having minimal polynomial $x^k+a_{k-1}x^{k-1}+\ldots+a_1x+a_0$.)
\end{lemma}

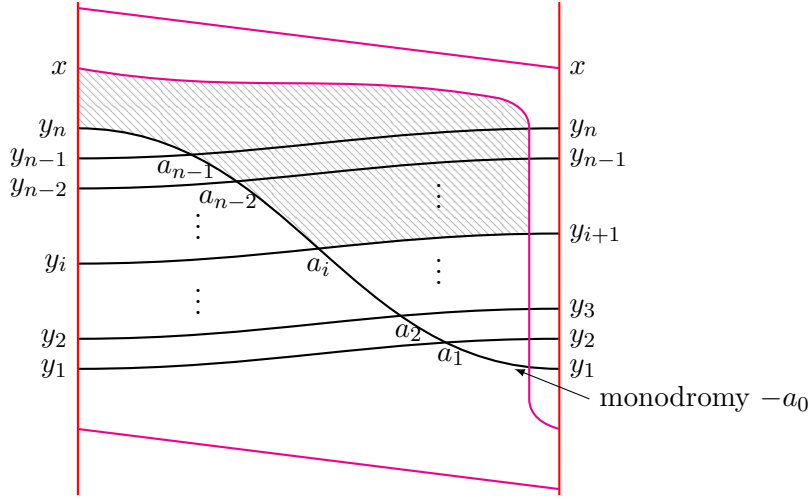
\begin{figure}[ht]
\centering
\begin{tikzpicture}[scale=.8]
\path[fill,pattern=north west lines,pattern color=lightgray] (-4,3) to [out=-10,in=170] (3,2.5) to [out=-15,in=90] (3.5,2) to (3.5,-2.5) to (3.5,.281) to (0,0) to [out=140,in=0] (-4,2) to (-4,3);

\draw[thick] (-4,-2) to [out=0,in=180] (4,-1.5);
\draw[thick] (-4,-1.5) to [out=0,in=180] (4,-1);
\draw[thick] (-4,-.25) to [out=0,in=180] (4,.25);
\draw[thick] (-4,1) to [out=0,in=180] (4,1.5);
\draw[thick] (-4,1.5) to [out=0,in=180] (4,2);

\draw[thick] (-4,2) to [out=0,in=180] (4,-2);

\node at (-2,-.75) {\vdots};
\node at (2,1) {\vdots};
\node at (-2,.5) {\vdots};
\node at (2,-.25) {\vdots};

\draw[thick,red] (-4,-4.1) to (-4,4.1);
\draw[thick,red] (4,-4.1) to (4,4.1);

\draw[thick,magenta] (-4,4) to (4,3);
\draw[thick,magenta] (-4,3) to [out=-10,in=170] (3,2.5) to [out=-15,in=90] (3.5,2) to (3.5,-2.5) to [out=-90,in=165] (4,-3);
\draw[thick,magenta] (-4,-3) to (4,-4);

\node[left] at (-4,3) {$x$};
\node[left] at (-4,2) {$y_n$};
\node[left] at (-4,1.5) {$y_{n-1}$};
\node[left] at (-4,1) {$y_{n-2}$};
\node[left] at (-4,-.25) {$y_i$};
\node[left] at (-4,-1.5) {$y_2$};
\node[left] at (-4,-2) {$y_1$};
\node[right] at (4,3) {$x$};
\node[right] at (4,2) {$y_n$};
\node[right] at (4,1.5) {$y_{n-1}$};
\node[right] at (4,.25) {$y_{i+1}$};
\node[right] at (4,-1) {$y_3$};
\node[right] at (4,-1.5) {$y_2$};
\node[right] at (4,-2) {$y_1$};

\node at (-2.2,1.3) {$a_{n-1}$};
\node at (-1.5,.8) {$a_{n-2}$};
\node[below] at (0,0) {$a_i$};
\node at (1.5,-1.4) {$a_2$};
\node at (2.2,-1.8) {$a_1$};
\draw[->] (4.5,-2.5) to (3.25,-2);
\node[right] at (4.5,-2.5) {monodromy $-a_0$};

\end{tikzpicture}
\caption{Computing the $\KK[x^{\pm}]$ module corresponding to $C_k$ with given monodromy and Maurer--Cartan element. The generator of $\mc F$ is drawn in red, its wrapped copy in magenta. Quadrilaterals like the one drawn contribute to the action of $x$ on the vector space with basis $y_1,\ldots,y_n$}
\label{fig_cyclic_arc}
\end{figure}

\begin{proof}
The morphism space $\Hom((C_k,E),(C_k,E))_{>0}$, where $E$ is any local system on $C_k$, is concentrated in degree one with basis the $k-1$ self-intersection points of $C_k$, so $\mc{MC}(C_k,E)=\Hom((C_k,E),(C_k,E))_{>0}$, for which we get a basis by trivializing $E$ away from the point at the very bottom of $C_k$.
To find the representation of $\KK[x^{\pm}]$ corresponding to a given choice of monodromy and Maurer--Cartan element we need to intersect with the arc $L=\{0\}\times [-1,1]$, which gives a complex concentrated in a single degree with basis $y_1,\ldots,y_n$, and look for triangles with edges on a perturbed copy of $L$, $L$, and $C_k$, see Figure~\ref{fig_cyclic_arc}.
We get $\mathrm{CMAT}(a_0,a_1,\ldots,a_{k-1})$ where $-a_0$ is the monodromy of $E$ and $a_i$, $i=1,\ldots,k-1$, is the  coefficient of the Maurer--Cartan element at the $i$-th self-intersection point from the bottom.
\end{proof}

In order to prove injectivity of $\Phi$ we will consider the full subcategory
$\mathrm{Mod}_{\mathrm{un}}(\KK[x^{\pm}])$ of $\mathrm{Mod}_{\mathrm{fd}}(\KK[x^{\pm}])$ consisting of those modules where $x$ acts by a unipotent endomorphism $U$, or equivalently the spectrum of $U$ is concentrated at $1\in\overline{\KK}$.
To simplify the notation write $\mathrm{Mod}_{\mathrm{un}}:=\mathrm{Mod}_{\mathrm{un}}(\KK[x^{\pm}])$ and $\mathrm{Mod}_{\mathrm{fd}}:=\mathrm{Mod}_{\mathrm{fd}}(\KK[x^{\pm}])$.
The category $\mathrm{Mod}_{\mathrm{un}}$ has, for each $k\geq 1$, a unique up to isomorphism indecomposable object $I_k$ of dimension $k$ where $x$ acts by an endomorphism with a single Jordan block and eigenvalue $1$.
The Hall algebra of $\mathrm{Mod}_{\mathrm{un}}$ is the \textit{classical Hall algebra} studied by Steinitz and Hall and is the free commutative algebra with a generator $[I_k]$ in each degree $k\geq 1$.
For the derived category $D^b(\mathrm{Mod}_{\mathrm{un}})$ we can use \eqref{hall_tprod_decomp} applied to slicing given by the standard t-structure to find the following basis of the Hall algebra.

\begin{lemma}\label{lem_hall_un_basis}
The Hall algebra $\mathrm{Hall}(D^b(\mathrm{Mod}_{\mathrm{un}}))$ has a basis given by the elements
\begin{equation}
[I_{k_1}[n_1]]\cdot [I_{k_1}[n_2]]\cdots [I_{k_m}[n_m]]
\end{equation}
where $m\geq 0$, $n_i\in\ZZ$, $n_1\geq n_2\geq \ldots \geq n_m$, $k_i>0$, and $k_i\geq k_{i+1}$ if $n_i=n_{i+1}$.
\end{lemma}

To relate this to the elements in \eqref{annulus_skein_basis} indexed by the same set, we proceed as follows.
Let
\[
J:D^b(\mathrm{Mod}_{\mathrm{un}}(\KK[x^{\pm}]))\hookrightarrow D^b(\mathrm{Mod}_{\mathrm{fd}}(\KK[x^{\pm}]))
\]
be the induced inclusion of derived categories.
Both the inclusion of abelian categories and the inclusion of derived categories split in the sense that the bigger category is a direct sum of the smaller category and another category, which includes those modules where $x$ acts by an operator with spectrum not including $1$.
The functor $J$ induces a linear surjection between Hall algebras
\[
J^!:\mathrm{Hall}(D^b(\mathrm{Mod}_{\mathrm{fd}}))\longrightarrow \mathrm{Hall}(D^b(\mathrm{Mod}_{\mathrm{un}}))
\]
where by definition we have $J^!([A])=[A]$ if $A\in D^b(\mathrm{Mod}_{\mathrm{un}})$ and $J^!([A])=0$ else.
Our goal is to show:

\begin{lemma}
The composition $J^!\circ\Phi$ maps
\begin{equation}
(J^!\circ \Phi)(C_{k_1}[n_1]\cdots C_{k_m}[n_m])=(q-1)^{-m}[I_{k_1}[n_1]]\cdots [I_{k_m}[n_m]]
\end{equation}
where $k_i$, $n_i$ are as in Proposition~\ref{prop_annulus_skein_basis} and Lemma~\ref{lem_hall_un_basis}.
\end{lemma}

\begin{proof}
First of all, $J^!(\Phi(C_k))=(q-1)^{-1}I_k$ by Lemma~\ref{lem_cyclic_braid_comp} and the fact that the companion matrix is unipotent only for a single choice of $a_i$'s, in which case we get a single Jordan block.

It is not true for general $A,B\in D^b(\mathrm{Mod}_{\mathrm{fd}})$ that
\begin{equation}\label{J_mult}
J^!([A][B])=J^!([A])J^!([B])
\end{equation}
however this does hold in the following cases, which together take care of the products of interest:
\begin{enumerate}[1)]
\item
$A,B\in \mathrm{Mod}_{\mathrm{fd}}$
\item
$\mathrm{Ext}^1(B,A)=0$
\end{enumerate}
First, if $A,B\in \mathrm{Mod}_{\mathrm{un}}$ then \eqref{J_mult} holds because the subcategory is extension closed.
On the other hand, if $A,B\in \mathrm{Mod}_{\mathrm{fd}}$ but one of $A$ or $B$ is not in $\mathrm{Mod}_{\mathrm{un}}$, then no extension of $A$ and $B$ is in $\mathrm{Mod}_{\mathrm{un}}$ either, so both sides of \eqref{J_mult} vanish.
In the second case, $\mathrm{Ext}^1(B,A)=0$, the product in the Hall algebra has only one term coming from the direct sum so \eqref{J_mult} holds by a similar reasoning.
\end{proof}

\begin{proof}[Proof of Theorem~\ref{thm_annulus} and Proposition~\ref{prop_annulus_skein_basis}]
By the previous lemma the elements \eqref{annulus_skein_basis} map, up to scalar factor, to a basis, but also span the skein, so must themselves form basis.
Furthermore, $J^!\circ\Phi$ is an isomorphism, thus $\Phi$ injective.
\end{proof}

\bibliographystyle{alpha}
\bibliography{skeinhall}

\Addresses

\end{document}